\pdfoutput=1 
\documentclass[11 pt]{article}
\usepackage{amscd}
\usepackage{makeidx,graphics}
\usepackage[pdftex]{graphicx}
\usepackage{color}
\color[cmyk]{0,0,0,1}

\usepackage{geometry}
\usepackage{hyperref}
\usepackage{amsmath,amsthm,amssymb,amscd}
\usepackage{array}
\usepackage{multirow}

\DeclareMathOperator{\id}{\operatorname{id}}

\DeclareSymbolFont{AMSb}{U}{msb}{m}{n}
\DeclareMathSymbol{\N}{\mathbin}{AMSb}{"4E}
\DeclareMathSymbol{\Z}{\mathbin}{AMSb}{"5A}
\DeclareMathSymbol{\R}{\mathbin}{AMSb}{"52}
\DeclareMathSymbol{\Q}{\mathbin}{AMSb}{"51}
\DeclareMathSymbol{\I}{\mathbin}{AMSb}{"49}
\DeclareMathSymbol{\C}{\mathbin}{AMSb}{"43}

\newcommand{\rr}[1]{\textcolor{red}{#1}}
\newcommand{\bb}[1]{\textcolor{blue}{#1}}

\newcommand{\bfsr}{\textcolor{red}{\bf r\ }}

\newcommand{\bfr}{\textcolor{red}{{\bf r}}}
\newcommand{\bfb}{\textcolor{blue}{{\bf b}}}
\newcommand{\bfk}{\textcolor{black}{{\bf k}}}

\newtheorem{thm}{Theorem}[section]

\title{On the inner structure of a permutation:\\
Bicolored Partitions and Eulerians\\ 
Trees and Primitives}

\author{Adrian Ocneanu}

\begin{document}
\maketitle

\begin{abstract}
We present a bijective algorithm with which an arbitrary permutation decomposes canonically into elementary blocks which we call families, which are sets with a specified number of ascents and descents.  We show that families, arranged in an arbitrary order in a sequence, are in bijection with permutations.  

The permutation decomposes canonically, by inserting parentheses, into a tree having as nodes a class of permutations which we call primitive.  Primitive permutations can be assembled from very simple data. 

The data for the trees into which a permutation decomposes can be written in a form similar to the decimal classification of a library. We axiomatize that data. It has a structure very different from the permutation which it encodes, with shuffles and pairings instead of reorderings. These structures are similar to the fundamental processes in quantum field theory.

Our main bijective structure algorithm gives explicit, additive multinomial formulae for the number of permutations with given sets of elements under and over the diagonal, or with given ascent and descent values.  

The multinomial expressions obtained this way give a new class of bicolored set statistics, between set partitions and set compositions, called shifted multinomials.  These provide for the first time additive multinomial expressions for Eulerian numbers and derangements, as part of a sequence of new combinatorial objects.  

These multinomial expressions satisfy inductive relations involving only immediate neighbors, similar to the relations satisfied by the Eulerian numbers.
 
\end{abstract}

\section{Preintroduction}

Much of this paper is based on the structure of a {\bf family}.  There are different equivalent ways to look at it.  First, a family $S$ is a set, together with a specified number $k$ ascent values, or reds, and $l$ descent values, or blues, with $k,l\ge 1$, and $k+1=\vert S\vert$.  The information in the bicoloring of $S$ can equivalently be given by making $S$ into a sequence $(r_1,r_2\ldots,r_k,b_l,b_{l-1},\ldots,b_1)$, where $b_1<b_2<\ldots b_l<r_1<r_2<\ldots<r_k$, and $\{r_1,\ldots,r_k\}$ are the reds, the elements colored red in the set $S$, and $\{b_1,\ldots,b_l\}$ are the blues, the elements colored blue. We put both kinds of information together by viewing $S$ as a colored sequence.

\begin{center}
\includegraphics[width=1.0\linewidth]{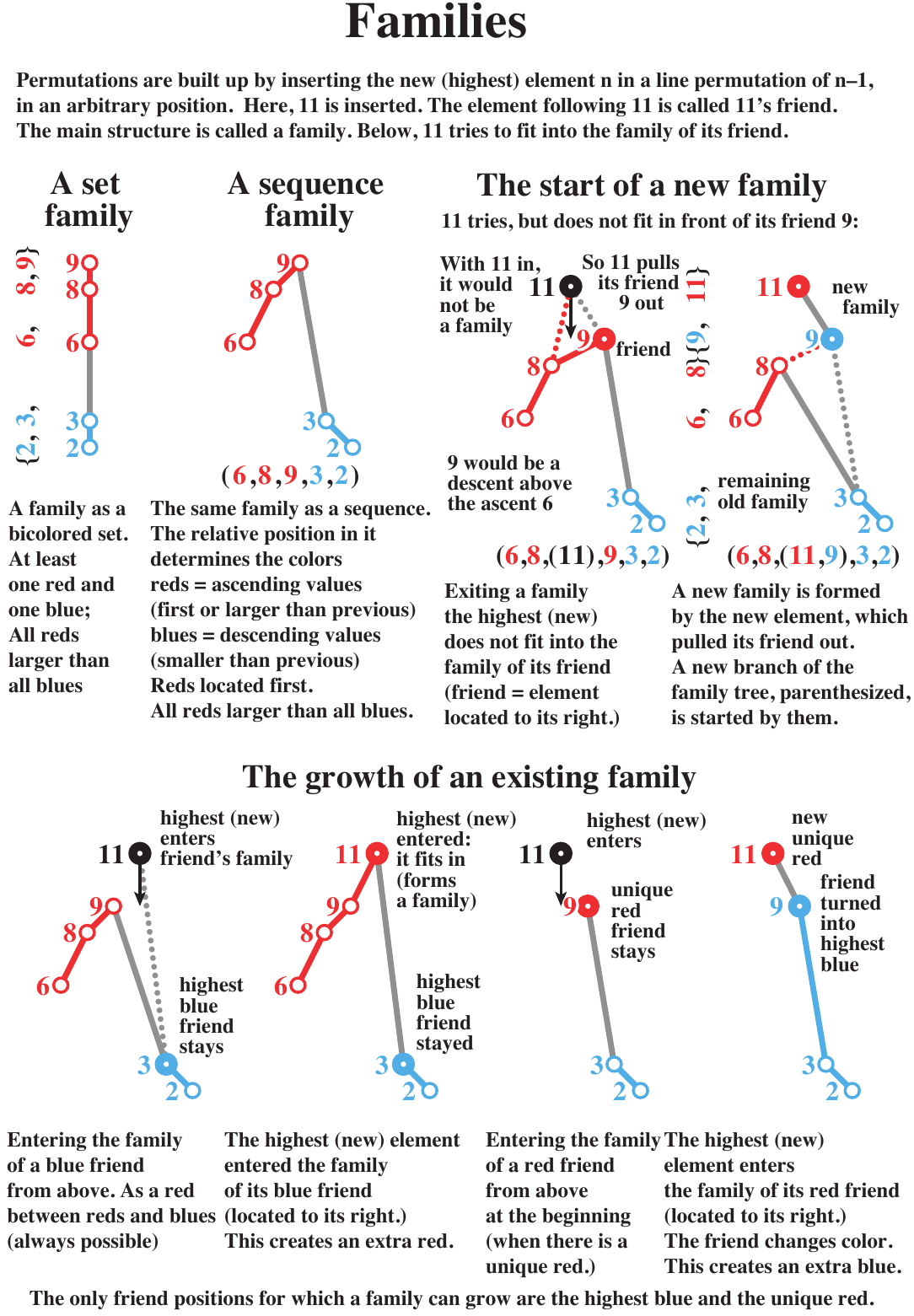}
\end{center}

The main function of a family $S$ is dynamical. We shall build up permutations of $\overline{n}=\{1,\dots,n\}$ in the manner of Euler, by inserting the new element $n$ into the line form of a permutation of $n-1$, on one of the $n$ positions available. Unless the new $n$ is at the end, where it would become a singlet, $n$ is followed in the line permutation by an element
$m$, which we call $n$'s \textbf{friend}. The friend $m$ is in  a family $S$, and we try to see whether $n$ fits into that family, positioned immediately before $m$. 

Either $S$ with $n$ inserted in front of $m$ is a family, in which case we say that $n$ fits in $S$, or else $n$ pulls out its friend $m$, and they start a new family $(\textbf{\rr{n}},\textbf{\bb{m}})$.  If $n$ is inserted at the end of $S$, it forms a new singlet family $(\textbf{n})$.

This is the main dynamics of the inner structure of a permutation.  The insertion of new elements is repeated, as the permutation is built, by inserting new, largest elements, into its line description.  The tree of possibilities in this process, enlarging families and shrinking them as they divide to create new families, is extremely complex, and gives rise to several tree and shuffle type structures.  Several algorithmic methods, which are the main results of this paper, show that the dynamics of family growth and division can be controlled by simple data, or reversed to yield bijective structure results on permutations.

The processes involved are similar to creation and annihilation operators in particle physics, or cell growth processes in biology.  

In the paper which follows, we shall study the inner structure of a single permutation.  Composing two permutations gives very rich structures, not studied here.

\section{Introduction} 
Permutations are central to the study of symmetry in mathematics and physics.  There are many interesting algebraic structures on the linear combinations of permutations.  The most remarkable of these is the block decomposition of their convolution algebra, namely the study of their representations.  This has been an active subject of research for over a century, and gave birth to modern representation theory. 

In contrast to the above, we shall concentrate in what follows on a new internal structure of each individual permutation.  Leonhard Euler found the recurrence relations on the number of permutations with a given count of ascents and descents, numbers which are now called the Eulerian.  These have found applications as numbers and operators in many areas of mathematics and physics.  After three centuries of study of permutations with given numbers $a$ of ascents and $b$ of descents, there appear to be no simple additive combinatorial formulae in terms of $a$ and $b$.  Such an expression would correspond to a construction of a permutation from a list of subsets involving $a$ and $b$.  

The central result of our study is such a formula.  The theory behind it is a way to decompose an arbitrary permutation of $\{1,\dots,n\}$ into an ordered sequence of mutually disjoint subsets $S_1 \cup \ldots \cup S_n$ in which for every $S_i$, which we call a family, we specify a number $a_i\in \{1,\ldots, \vert S_i\vert -1\}$ of ascent values, with the remaining number $b_i=\vert S_i\vert - a_i$ counting descent values.  The algorithm is bijective and provides an explicit way of building the permutation out of the list $(S_1,\ldots, S_k)$ which we call a registry of families. The numbers not used in the registry become singlet blocks.  A transformation, called the cycle transform, sends a permutation with given ascent values, descent values and singlets into a permutation with the same values which are over, under, and respectively on the diagonal.  Our algorithm gives thus an explicit construction of the permutations in which such sets, not only their cardinality, are specified. 

We show that any permutation has a canonical forest structure, in which nodes of trees consist of a special class of permutations, which we call primitive.  A primitive permutation consists of an active family, with ascent and descent values as described before, together with insertion points, called buds, for higher branches.  For the tree structure of a permutation, primitive permutations play a role analogous to prime numbers with the commutative multiplication replaced by a tree structure.

We give a bijective algorithm for constructing directly the primitive permutations from simple combinatorial data such as a sequence of positive and negative numbers -- or equivalently a Dyck path with decorated descents -- together with the desired number of ascents of the active family.  The explicit nature of the data leads to a differential equation with which we compute the generating function for the number of primitive permutations with an active family of size $m$ and with  $k$ buds.

The data for primitive permutations, which are the nodes of the tree structure of a permutation, is assembled in a form similar to the decimal classification of a library, which we call a decimal code of a permutation. The decimal code, which we axiomatize and show to be equivalent to the permutation data, replaces the change of order in the permutation with creation - annihilation pairing data and shuffling.

The description of permutations in terms of sets with a given number of ascents and descents gives rise to a study of partition of sets into bicolored subsets.  While such partitions were first studied half a century ago, we obtain for the first time Eulerian type inductive relations for these, involving only neighboring terms, as well as a new type of expressions for their count, as sums of shifted multinomials.  Such expressions interpolate between bicolored partitions and compositions.

The decomposition of a permutation into a forest is constructed inductively by inserting the new highest element in the line description of a permutation, in the manner of Euler.  The growth of the forest in this process has similarities with biological growth patterns. Small nodes, or families, mostly grow while larger nodes tend to shrink in size due to a process of  reproductive division.  This process accounts for the extraordinary complexity of the end result of the forest and node structure as well as for the connections to particle creation and annihilation of particles in quantum field theory, as well as to possible connections with mathematical biology. 

\subsection{The library metaphor} We are in a big library, in which books are labeled by their date of entry. A new book is put in a random position every day by a careless librarian assistant.

The books are arranged on shelves, some, on the left, leaning left, and some, on the right, leaning right, at least one on each side.  On both sides, numbers should increase toward the center of the shelf, and all books on the right must be older than all books on the left.  As the new, latest book, is put in a random position on a random shelf, there are two cases.  The new book may fit on the shelf, and this happens in two cases.  Either the book is placed in the middle, in which case being newest, it will lean left.  Or, there is only one book on the left and the new book is put against the left wall.  Then, the former single left book must lean right.  In all other cases, the new book does not fit.  Then, it picks the book immediately following it, and with them a new shelf is started.  There are now two ways to proceed, organizing the whole library.  In the first, the new shelf and the remains of the previous one are arranged in the library in a very elaborate way.  That allows, with an intricate algorithm, to figure out the exact history of the library, just from the positions and contents of the shelves.  Moreover, every shelf arrangement corresponds to a history.  This is our main bijective algorithm.

Another system is to keep track of the books that were taken away and repositioned, by leaving markers for each pair which was taken away, and by assigning these books on their new shelf a longer decimal number.  The list of these decimal numbers is in mostly ascending order, but has books from different shelves shuffled relative to each other, and keeps track of the pairings between the books which started a new shelf.  This is a decimal code for our library, and again it encodes the whole library structure.  In this system, which is very far from the original permutation, we can think of the pair of books which went missing from a shelf, in a grand way, as an annihilation in a universe.  As they reappear on their own new shelf, they form a new universe.  And indeed, the modern type of library, which is the online encyclopedia of integer sequences, shows that in the simplest case, when universes disappear without creating new ones, on a unique shelf with only two books left on it, the number of possible histories matches one to two the number of indecomposable fermionic Feynman diagrams with a given number of creation -- annihilation pairs.

\begin{center}
\includegraphics[width=6.1 in]{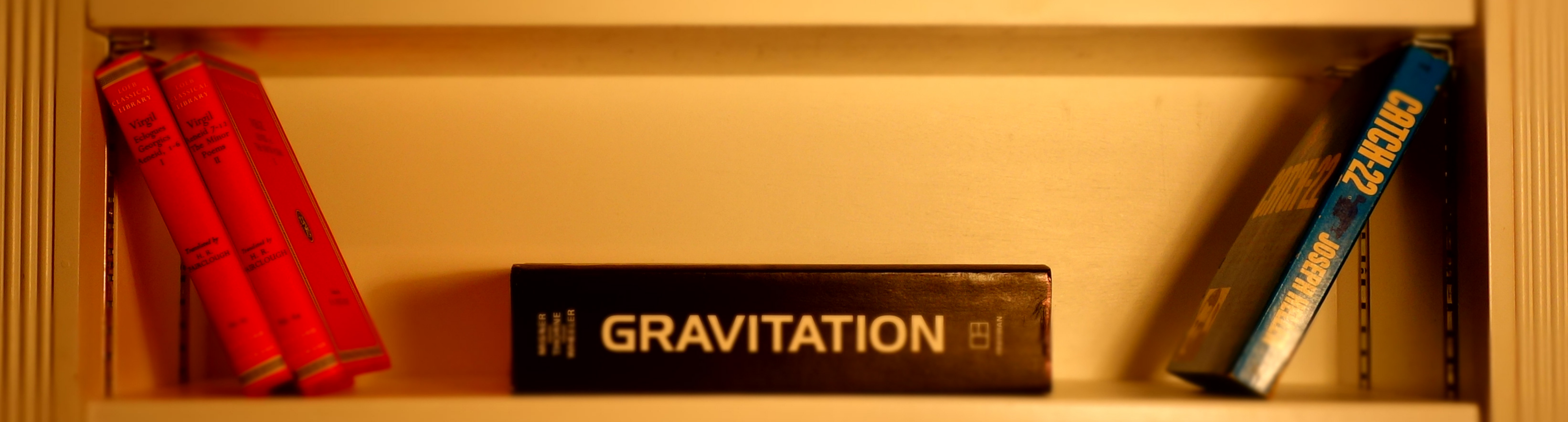}
\end{center}

\section{Main theorem: permutations with prescribed ascent and descent values}

\subsection{The cycle transform}

Let $\pi$ be a permutation of $\{1,\ldots,n\}$.  Extend it with zero at the beginning, i.e., $\pi(0)=0$.  A position $i$ is called an ascent  (position) if $\pi(i-1)<\pi(i)$, and a descent (position) otherwise.  If $\pi(i-1)<\pi(i)$, the value $\pi(i)$ is called an ascent value, else $\pi(i)$ is a descent value.  If $\pi(i)>i$, then $i$ is called an over diagonal position and $\pi(i)$ an over diagonal value.  If $\pi(i)<i$, then $i$ is called an under diagonal position and $\pi(i)$ an under diagonal value.  Elements with $\pi(i)=i$ are called on diagonal.  Note that $\pi$ maps over/under positions into over/under diagonal values.  That is not the case with ascent/descent values and positions.  Ascents and descents are defined between adjacent positions, while the images of ascents and descents are not adjacent to each other.

For the permutation $\pi$, consider the absolute minimum value 1, with $\pi(i_1)=1$.  Let $m_1=1$, and let $m_2=\text{min}_{i>i_1}\pi(i)$, with $\pi(i_2)=m_2$.  Continue similarly the sequence of \textbf{successive minima} $m_1<m_2<m_3<\ldots$, and the corresponding successive minima positions $i_1<i_2<i_3<\ldots$, with $\pi(i_k)=m_k$.  We call the intervals $\{1,\ldots, i_1\}, \{i_1+1,\ldots,i_2\}, \ldots$ the  \textbf{blocks} of the permutation, with $i_1,i_2,\ldots$ the block \textbf{anchors}, and $m_1, m_2,\ldots$ the {\bf block anchor values}.  A block consisting of a single element will be called a {\bf singlet}.  The similarly defined successive maxima, not used here, are called salients in the literature.

A remarkable transformation was discovered half a century ago by Alfr\'{e}d R\'{e}nyi, and called an ``unusual transformation'' by Donald Knuth.  It is the {cycle transform}, and is defined as follows.

Decompose a permutation $\sigma$ into cycles, arranged so that the minimum of each cycle is on the last position, and so that these minima are in ascending order.  Write now the list of elements of these cycles, and read it as a line permutation $\pi$.  For instance, if $\sigma$, written as a line permutation, is $524361$, then its cycles ordered as above are $(561)(2)(43)$.  The cycle transform $\pi$ is $561243$ in line notation, with successive minima $1,2,3$, and blocks $561,2,43$, which are the cycles of $\sigma$.  $\sigma$ has $4,5,6$, respectively $2$, respectively $1,3$, as over/on/under diagonal values, which are also the ascent values, singlet blocks, and, respectively, descent values of $\pi$.

We call the permutation $\pi$ the \textbf{cycle transform} of $\sigma$.  Remark that the minima of cycles of $\sigma$ become the successive minima of $\pi$, and the cycles of $\sigma$ are the blocks of $\pi$, as defined above.  $\sigma$ is the {\bf inverse cycle transform} of $\pi$, and is obtained by interpreting each block between successive minima positions of $\pi$ as a cycle of $\sigma$.

The on diagonal elements of $\sigma$ are the singlet blocks of $\pi$.  The over/under diagonal values in $\sigma$ are the ascent/descent values of $\pi$, which are not singlets.  Depending on the problem studied, it is sometimes convenient to view singlets as ascents, while in other applications they are set apart in statistics.

The structure studied in what follows will concern the ascent/descent values and the blocks between successive minima of a permutation $\pi$.  We shall use the inverse cycle transforms of these to over/under diagonal values and cycles only for the interpretation of the results. The cycle transform as a three dimensional picture is shown below.

\begin{center}
\includegraphics[width=1.0\textwidth]{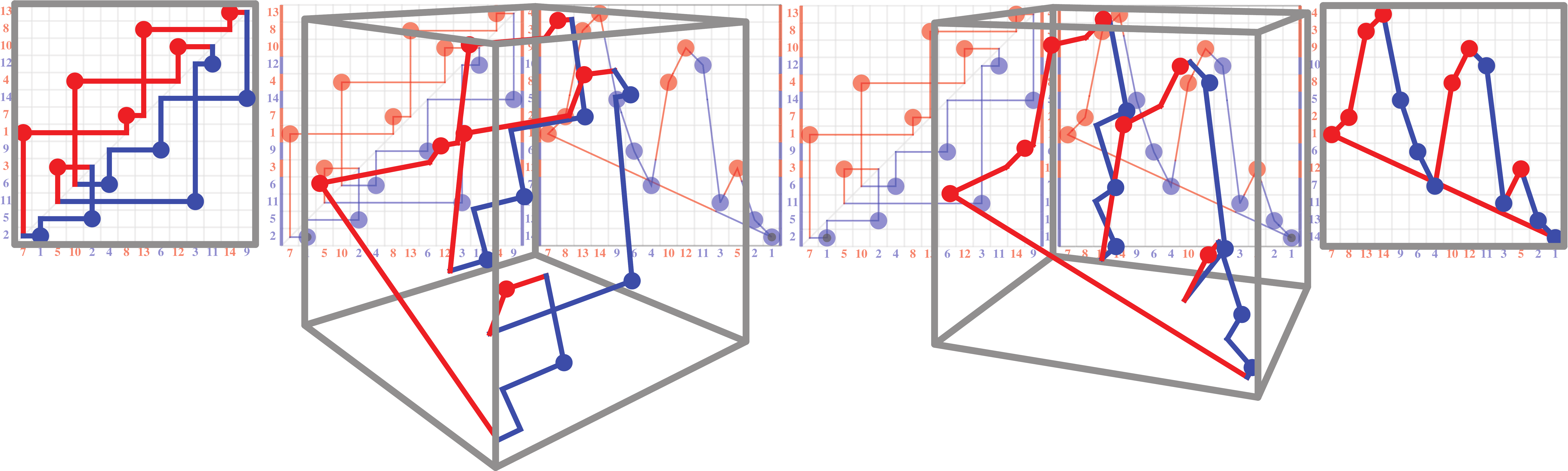}
\end{center}

\subsection{Derangements, Eulerians, and Eulerian lift}

We call a permutation $\sigma$ a \textbf{derangement} if it has no diagonal elements, or cycles of length one.  We denote by $D_{a,b}$ the number of derangements with $a$ values over diagonal and $b$ values under diagonal.  The cycle transform sends these derangements into permutations without singlet blocks, with $a$ ascent values and $b$ descent values.  

Extend a permutation $\pi$ with zeroes to  $(0,\pi(1),\pi(2),\ldots,\pi(n),0)$.  Euler studied the number of permutations with a given number $a$ of ascent adjacent pairs and a number $b$ of descent adjacent pairs in that sequence.  The number of such permutations are the Eulerian numbers, denoted in the literature by $A(a+b,a)$, and here by $E_{a,b}$, with $\sum_{a+b=n+1}E_{a,b}=n!$.  

Given a permutation $\pi$ of $\{1,\ldots,n\}$ with $a$ ascent values, including the singlet blocks, and $b$ descent values, as counted by Euler, we define a new permutation $\tilde{\pi}$ of $\{1,\ldots,n+1\}$, which we call the Eulerian lift of $\pi$, by letting $\tilde{\pi}(i)=\pi(i)+1$ for $i\le n$ and $\tilde{\pi}(n+1)=1$.  Note that $\tilde{\pi}$ has a single block, and the number of its ascent and descent values is the same as the number of ascent and descent pairs counted by Euler.  For instance, $\pi=312$ has two ascent pairs $0312$ and two descent pairs $31,20$ in $03120$.  Its Eulerian lift is $\tilde{\pi}=4231$, and has in $04231$ ascent values $4,3$ and descent values  $2,1$.  The inverse cycle transform of $\tilde{\pi}$ is the cycle $(4,2,3,1)$, written $4312$ as a line permutation.  It has $4,3$ as over diagonal values, and $1,2$ as under diagonal values.

The transformation which corresponds to the Eulerian lift by the inverse cycle transform is {\bf daisy chaining}. Let $\sigma$ be a permutation of $\{1,\dots,n\}$ and let $c_1,\dots,c_k$ be the list of its cycle heads, the smallest elements in each cycle, arranged increasingly. Let $\sigma'$ denote the permutation of  $\{0,1,\dots,n\}$ defined by $\sigma' c_i = \sigma c_{i+1}$ for $i<k$, $\sigma' c_k = 0$, $\sigma' 0 = \sigma c_1$ and $\sigma' j = \sigma j$ otherwise.  Now lift $\{0,\dots,n\}$ to $\{1,\dots,n+1\}$ to obtain $\tilde{\sigma}$. Daisy chaining puts in bijection permutations of $\overline n$ with single cycle permutations of $\overline {n+1}$.

The Eulerian numbers $E_{k,l}$ count thus the number of permutations of single cycle permutations with $k$ elements over and $l$ elements under the diagonal.

\subsection{Patterns and shuffles}

In the sequel, we shall make use of \textbf{patterns and shuffles.} For a finite $F\subset\R$, an injective function $f:F\to \R$ has a {\bf pattern} which is a permutation of $(1,\dots,|F|)$, obtained by composing $f$ on both sides with order preserving maps. 

Given an injective function $f:F\to \R$ on a finite $F\subset\R$ we call $f$ a shuffle of $f_1,\dots,f_n$ if there is a partition $F=\bigcup_i F_i$ so that the restriction of $f$ to any $F_i$ has the same pattern as $f_i$. 

More generally, for an injective function $f:F=\bigcup_i F_i\to \R$ with $F_i$ not necessarily mutually disjoint, we call $f$ a {\bf shuffle of  $f_1,\dots,f_n$  along $(F_1,\dots,F_n)$} if the restriction of $f$ to any $F_i$ has the same pattern as $f_i$.

\subsection{Families and registries}

In what follows, we shall develop structural results which will put a permutation in bijection with a \textbf{composition}, an ordered partition, of part of its positions into bicolored sets, called families, with the remaining elements singlet blocks.  The algorithm will construct explicitly, and thus count, permutations with given sets of ascent/descent values, or equivalently, via the inverse cycle transform, of given sets of over/under diagonal values or over/under diagonal positions.

Recall that a \textbf{set family} $S$ is a subset of $\{1,\ldots,n\}$ with a color, red or blue, assigned to each element, with $r\in\{1,\ldots,\vert S\vert -1\}$ \textbf{reds}, and \textbf{blues}, such that any red element is larger than any blue element.  Thus, a set family is completely determined by its underlying set $S$ with $\ge 2$ elements, together with the number $r\in\{1,\ldots,\vert S\vert-1\}$ of reds.

A \textbf{sequence family} is a sequence of integers $(a_1,\dots,a_k,b_l,\dots,b_1)$ with $k,l \ge 1$ and $b_1<b_2<\cdots<b_l<a_1<a_2<\cdots<a_k$.  The elements $r_1,r_2,\dots,r_k$ are called ascent values, or reds, and the elements $b_1,b_2,\dots,b_l$ are called descent values, or blues, and we color the sequence as $(\rr{a_1},\dots,\rr{a_k},\bb{b_l},\dots,\bb{b_1})$. Note that the reds and the blues in a nonsinglet sequence family are completely determined by their relative position in the sequence, and all reds have values larger than all blues. Thus the subjacent bicolored set of a sequence family is a set family. 

Conversely the set family $\{\bb{b_1},\dots,\bb{b_l},\rr{a_1},\dots,\rr{a_k}\}$ with $b_1<\cdots<b_l<a_1\cdots<a_k$ determines the sequence family $(a_1,\dots,a_k,b_l,\dots,b_1)$, or $(\rr{a_1},\dots,\rr{a_k},\bb{b_l},\dots,\bb{b_1})$. We shall consider all these as set and sequence forms of a same family.

Given a subjacent set $S$ with $\ge 2$ elements, there is precisely one family with a given number $r\in\{1,\ldots,\vert S \vert-1\}$ reds.

We shall also sometimes consider singlets as families (colored black in the pictures).

A \textbf{family registry} is an ordered sequence of mutually disjoint set families. 
For various purposes, we may skip the singlet families from registries.

{\bf Examples}

$S=\{2,3,5,6,7\}$ colored with 2 reds is the set family $\{\bb{2},\bb{3},\bb{5},\rr{6},\rr{7}\}$, and the sequence family $(6,7,5,3,2)$ colored as $(\rr{6},\rr{7},\bb{5},\bb{3},\bb{2})$, or $\rr{6}\rr{7}\bb{5}\bb{3}\bb{2}$.

A set with four elements $\{1234\}$ can be organized as a family in 3 ways:
$\rr{4}\bb{3}\bb{2}\bb{1}$, 
$\rr{3}\rr{4}\bb{2}\bb{1}$, 
$\rr{2}\rr{3}\rr{4}\bb{1}$.  
These are all the patterns of families with four elements.

The colored sequences \rr{58}\bb{431} and \rr{7}\bb{52} and 9 are families while \rr{58}\bb{341}, \rr{58}\bb{631}, \rr{75}\bb{2} and \rr{58} and \bb{531} are not.

\section{The main counting theorem}

Our main permutation construction result is the following.  
\begin{thm}

The permutations with a given set $A$ of elements over the diagonal, a given set of elements $B$ under the diagonal and the rest on the diagonal are in bijection with the number of compositions $A=A_1\cup \dots \cup A_k$, $B=B_1\cup \dots\cup B_k$ with $\min A_i >\max B_i$ for all $i$.

\end{thm}
{\bf An example} 

Consider the subsets $A=\{2,4,5\}$ and $B=\{1,3\}$ of $\overline 5 = \{1,2,3,4,5\}$.   We want to construct all the permutations with $A$ as ascent values and $B$ as descent values, or equivalently all the permutations with $A$ as over diagonal values and $B$ as under diagonal values, or equivalently  all the permutations with $B$ as under diagonal positions and $A$ as under diagonal positions.

Consider the compositions $A=\cup_{j\in J} A_j$,  $B=\cup_{j\in J} B_j$ with $\min A_j>\max B_j$ for all $j$. 

There are 6 such compositions, written as vectors
$$\begin{matrix}{245}\\{13}\end{matrix} : 
\left(\begin{matrix}{2}\\{1}\end{matrix}\right)\cup\left(\begin{matrix}{45}\\{3}\end{matrix}\right) = 
\left(\begin{matrix}{45}\\{3}\end{matrix}\right)\cup\left(\begin{matrix}{2}\\{1}\end{matrix}\right) = 
\left(\begin{matrix}{24}\\{1}\end{matrix}\right)\cup\left(\begin{matrix}{5}\\{3}\end{matrix}\right) = 
\left(\begin{matrix}{5}\\{3}\end{matrix}\right)\cup\left(\begin{matrix}{24}\\{1}\end{matrix}\right) = 
\left(\begin{matrix}{25}\\{1}\end{matrix}\right)\cup\left(\begin{matrix}{4}\\{3}\end{matrix}\right)= 
\left(\begin{matrix}{4}\\{3}\end{matrix}\right)\cup\left(\begin{matrix}{25}\\{1}\end{matrix}\right).$$
A pair $A_j>B_j$ forms a set family $A_j\cup B_j$ with $A_j$ as reds and $B_j$ as blues.
The families given by the compositions are arranged as registries, i.e. ordered lists of families, as follows.
$$
((\rr{2}\bb{1}),(\rr{4}\rr{5}\bb{3})),\  
((\rr{4}\rr{5}\bb{3}),(\rr{2}\bb{1})),\ 
((\rr{2}\rr{4}\bb{1}),(\rr{5}\bb{3})),\ 
((\rr{5}\bb{3}),(\rr{2}\rr{4}\bb{1})),\ 
((\rr{2}\rr{5}\bb{1}),(\rr{4}\bb{3})),\ 
((\rr{4}\bb{3}),(\rr{2}\rr{5}\bb{1})).
$$
By our reverse algorithm these registries are mapped into 6 permutations $\pi$ of $\overline 5$ with $A=\{\rr{2},\rr{4},\rr{5}\}$ as ascent values and $B=\{\bb{1},\bb{3}\}$ as descent values, respectively (using the block algoritm).$$
\rr{2}\bb{1}\rr{4}\rr{5}\bb{3}, 
\rr{2}\rr{4}\rr{5}\bb{3}\bb{1}, 
\rr{2}\rr{4}\bb{1}\rr{5}\bb{3}, 
\rr{2}\rr{5}\bb{3}\rr{4}\bb{1}, 
\rr{2}\rr{4}\bb{3}\rr{5}\bb{1}, 
\rr{2}\rr{5}\bb{1}\rr{4}\bb{3}
,$$

The blocks between successive minima of these are respectively

$$
(\rr{2}\bb{1})(\rr{4}\rr{5}\bb{3}), 
(\rr{2}\rr{4}\rr{5}\bb{3}\bb{1}),
(\rr{2}\rr{5}\bb{3}\rr{4}\bb{1}),  
(\rr{2}\rr{4}\bb{1})(\rr{5}\bb{3}), 
(\rr{2}\rr{5}\bb{1})(\rr{4}\bb{3}), 
(\rr{2}\rr{4}\bb{3}\rr{5}\bb{1})
.$$ 
These blocks give the cycles of the inverse cycle transform. Written in line notation, the inverse cycle transforms are

$$
\rr{2}\bb{1}\rr{4}\rr{5}\bb{3}, 
\rr{2}\rr{4}\bb{1}\rr{5}\bb{3},
\rr{2}\rr{5}\rr{4}\bb{1}\bb{3},
\rr{2}\rr{4}\rr{5}\bb{1}\bb{3},
\rr{2}\rr{5}\rr{4}\bb{3}\bb{1},
\rr{2}\rr{4}\rr{5}\bb{3}\bb{1}.
$$ 
In these, the elements $\{\rr{2},\rr{4},\rr{5}\}$ of $A$ and respectively $\{\bb{1},\bb{3}\}$ of $B$, are the values over, respectively under, the diagonal, as desired. 

We now replace each permutation $\sigma$ above by $\sigma^{-1}$. This will make $A=\{\bb{2},\bb{4},\bb{5}\}$ and $B=\{\rr{1},\rr{3}\}$ the sets of positions mapped under $(\bb{u})$ respectively over $(\rr{o})$ the diagonal. That is, we obtain the 6 permutations with the pattern $(\rr{o}\bb{u}\rr{o}\bb{u}\bb{u})$. 
$$
\rr{2}\bb{1}\rr{5}\bb{3}\bb{4},
\rr{3}\bb{1}\rr{5}\bb{2}\bb{4},
\rr{4}\bb{1}\rr{5}\bb{3}\bb{2},
\rr{4}\bb{1}\rr{5}\bb{2}\bb{3},
\rr{5}\bb{1}\rr{4}\bb{3}\bb{2},
\rr{5}\bb{1}\rr{4}\bb{2}\bb{3}
$$ 

In our study of the structure of a permutation, we shall decompose it with nested parentheses, such that the outermost parentheses are its blocks (between successive minima positions).  The contents of each parenthesis, in which the included parentheses and their contents are removed, form a sequence family.  This forest structure, in which each block is a tree, is canonical, and is called the parenthesized form of the permutation.

The parenthesized forms of the permutations $\pi$, which refine their block decompositions as written before, are respectively

$$
(\rr{2}\bb{1})(\rr{4}\rr{5}\bb{3}), 
(\rr{2}(\rr{4}\rr{5}\bb{3})\bb{1}),
(\rr{2}(\rr{5}\bb{3})\rr{4}\bb{1}),  
(\rr{2}\rr{4}\bb{1})(\rr{5}\bb{3}), 
(\rr{2}\rr{5}\bb{1})(\rr{4}\bb{3}), 
(\rr{2}(\rr{4}\bb{3})\rr{5}\bb{1})
.$$ 

The above parenthesized forms show that the compositions of the ascent/descent sets $A, B$ from which we started, written as sequence families
$$
\{\rr{2},\rr{4},\rr{5},\bb{1},\bb{3}\} = 
(\rr{2}\bb{1})\cup(\rr{4}\rr{5}\bb{3}) =   
(\rr{4}\rr{5}\bb{3})\cup(\rr{2}\bb{1}) =  
(\rr{2}\rr{4}\bb{1})\cup(\rr{5}\bb{3}) =  
(\rr{5}\bb{3})\cup(\rr{2}\rr{4}\bb{1}) =  
(\rr{2}\rr{5}\bb{1})\cup(\rr{4}\bb{3}) = 
(\rr{4}\bb{3})\cup(\rr{2}\rr{5}\bb{1})
$$
appear, very transparently, in the very body of the corresponding permutations, inside the parentheses. 

Our direct algorithm gives, starting from these permutations, an ordering of the families which appear in parentheses as a registry of families. This way we recover the compositions of $A,B$ from which we started.

\section{Bicolored partitions and their Eulerian type recurrence relations}

We shall count Eulerian numbers $E_{a,b}$, derangements $D_{a,b}$, and similar classes of bicolored permutations, by a new type of statistics of bicolored sets, which we now describe.

Let $k,l\in\{1,2,\ldots\}$.  Decompose as vectors $(k,l) = (k_1,l_1)+\dots+(k_r,l_r)$ into an (unordered) partition,  with $k_i,l_i\in\{1,2,\dots\}$. Let $r_1,...,r_p$ be the number of identical parts of each type, with $r=r_1+\dots+r_p.$  

The multinomial sum 
$$\sum \frac{(k+l)!}{(k_1+l_1)!\dots(k_r+l_r)!}$$ 
counts the decompositions of the set $\{1,\ldots,k+l\}$ into $r$ subsets of sizes $k_1+l_1,\ldots,k_r+l_r$.

The sum 
$$N^{[r]}_{k,l}=\sum \frac{(k+l)!}{(k_1+l_1)!\dots(k_r+l_r)!}\frac{1}{r_1!\dots r_p!}$$
counts the corresponding set partitions, i.e. the decompositions in which the order of the subsets is neglected.  Note that subsets with the same cardinality $k_i+l_i=k_j+l_j$, in which the set labels $(k_i,l_i)$, $(k_j,l_j)$ are different are counted as different in the partition.  
We call the numbers $N^{[r]}_{k,l}$ the \textbf{fundamental $r-$part multinomials}. 

If we now take 
$$r!N^{[r]}_{k,l}=\sum \frac{(k+l)!}{(k_1+l_1)!\dots(k_r+l_r)!}\frac{r!}{r_1!\dots r_p!},$$
we count the above as set compositions, or ordered partitions.  More generally, we shall count shifted partitions, defined by 
$$ N_{k,l}^{(s)} = \sum_{r\ge\max(-s,1)}(r+s)!N_{k,l}^{[r]} = \sum_{r\ge\max(-s,1)}\sum \frac{(k+l)!}{(k_1+l_1)!\dots(k_r+l_r)!}\frac{(r+s)!}{r_1!\dots r_p!}.$$
Here $s\in\{\ldots,-1,0,1,\ldots\}$ is a \textbf{shift}, and we take negative factorals to be zero.  We call the numbers $N^{(s)}_{k,l}$ the \textbf{shifted multinomials of shift $s$.}

Our main enumeration result is the following.  

\begin{thm}
The Eulerian numbers have the multinomial expression 
$$E_{k,l}=N^{(-1)}_{k,l} = \sum \frac{(k+l)!}{(k_1+l_1)!\dots(k_r+l_r)!}\frac{(r-1)!}{r_1!\dots r_p!}.$$
The derangements have the expression 
$$D_{k,l}=N^{(0)}_{k,l} = \sum \frac{(k+l)!}{(k_1+l_1)!\dots(k_r+l_r)!}\frac{r!}{r_1!\dots r_p!},$$
summing  over the unordered partitions $(k,l)=(k_1,l_1)+\cdots+(k_r,l_r)$, with $k_i,l_i\in\{1,2,\ldots,\}$, $r_1,\ldots,r_p$ the number of identical parts of each type, and $r=r_1+\cdots+r_p$
\end{thm}
We shall show, with a method which relies on the decomposition of partitions into set families, the following Eulerian type recurrence relations, between the fundamental and shifted multinomials.
\begin{thm}
\begin{eqnarray*}
N_{k,l}^{[r]} &=& kN_{k,l-1}^{[r]}+lN_{k-1,l}^{[r]}+(k+l-1)(N_{k-1,l-1}^{[r-1]}-r N_{k-1,l-1}^{[r]}) \\
N_{k,l}^{(s)} &=& kN_{k,l-1}^{(s)} +lN_{k-1,l}^{(s)} +(k+l-1)(N_{k-1,l-1}^{[-s-1]}+(s+1)N_{k-1,l-1}^{(s)}).
\end{eqnarray*}
\end{thm}

\begin{thm}
For negative shifts $s\le 0$, using compositions (ordered partitions) with $r\ge -s$ parts 
$$ (k,l)=(k_1,l_1)+\cdots + (k_r,l_r),$$
where $k_i,l_i\in\{1,2,\ldots\}$, letting $n_i=k_i+l_i$, the shifted multinomials $N^{(s)}_{k,l}$ decompose as 

\begin{eqnarray*}
N_{k,l}^{(s)} &=& \sum \binom{n_1+\dots +n_r-1}{n_1-1}\binom{n_2+\dots +n_r-1}{n_2-1}\dots\\
 &\ & \dots \binom{n_{-s-1}+\dots +n_r-1}{n_{-s-1}-1} \binom{n_{-s}+\dots +n_r -1}{n_{-s}-1,n_{-s+1},\dots,n_{r}}.
\end{eqnarray*}

\end{thm}
This generalizes the expression of multinomials in terms of binomials. The identity decomposes the shift $s\le 0$ into $|s|$ binomial shifts of $-1$ (and does not work for other such decompositions of the shift). 

\subsection{Combinatorial interpretation} 

Shifted multinomials with shift 0 count compositions of a bicolored set with $a$ reds and $b$ blues into parts $S_i$ which have each a number of reds $a_i\ge 1$ and blues $b_i\ge 1$, $|S_i| = a_i+b_i$.

For a negative shift $-s<0$, $s$ of the parts are on fixed positions. For instance let $S_1$ be the part containing the smallest element 1, $S_2$ the part containing the smallest element not in $S_1$, and so on.
Then $N_{k,l}^{(s)}$ counts compositions which keep $S_1,\dots,S_s$, in order, on the first $s$ positions. In the case of our algorithm which maps permutations into ordered lists of bicolored sets, $S_1,\dots,S_s$ are in order on the last $s$ positions. In particular single block permutations are mapped onto compositions with $S_1$ last. As we map permutations of $n$, with the Eulerian lift, onto single block permutations of $n+1$ and count Eulerian numbers on the image, we obtain expressions of Eulerian numbers in terms of shifted multinomials with shift $-1$. 

For positive shifts $s>0$, shifted multinomials count compositions in which $s$ dummy parts (markers or jokers) are inserted as parts. Thus positively shifted multinomials are divisible by $s!$, accounting for the relative order of the extra parts.

\subsection{Examples}  The above formulae show that the Eulerian number $E_{3,3}$ is given by the sum over the partitions $$(3,3)=(2,2)+(1,1)=(2,1)+(1,2)=(1,1)+(1,1)+(1,1)$$
$$E_{3,3} = \frac{6!}{6!} \frac{0!}{1!}+\frac{6!}{4!2!} \frac{1!}{1!1!}+\frac{6!}{3!3!} \frac{1!}{1!1!}+\frac{6!}{2!2!2!} \frac{2!}{3!}=1+15+20+90\cdot \frac 1 3=66.$$

The fundamental multinomials are $N_{3,3}^{[1]}=1,\ N_{3,3}^{[2]}=35,\ N_{3,3}^{[3]}=15$.  

From these, we compute the Eulerian
$$E_{3,3}=0! N_{3,3}^{[1]}+1! N_{3,3}^{[2]}+2! N_{3,3}^{[3]}=1+35+2\times 15=66.$$ 
From the same fundamental multinomials, with different coefficients, we compute the derangement number
$$D_{3,3}=1! N_{3,3}^{[1]}+2! N_{3,3}^{[2]}+3! N_{3,3}^{[3]}=1+2\times 35+6\times 15=161.$$ 

The Eulerian number $E_{3,3}$ also has the following expression, defined in terms of compositions of $(3,3)$,
$$(3,3)=(2,2)+(1,1)=(1,1)+(2,2)=(2,1)+(1,2)=(1,2)+(2,1)=(1,1)+(1,1)+(1,1)$$ with sums $$6=2+4 = 4+2 = 3+3 = 3+3=2+2+2$$ 
is, after shifting the size of the first parts  by $-1$, as sums of multinomials
$$E_{3,3}=\binom{5}{5}+\binom{5}{1,4}+\binom{5}{3,2}+\binom{5}{2,3}+\binom{5}{2,3}+\binom{5}{1,2,2}=1+5+10+10+10+30=66.$$

Eulerians, derangements, multinomials, and their relations are illustrated in the following picture.
\vskip 0.2 in
\includegraphics[width=5.9 in]{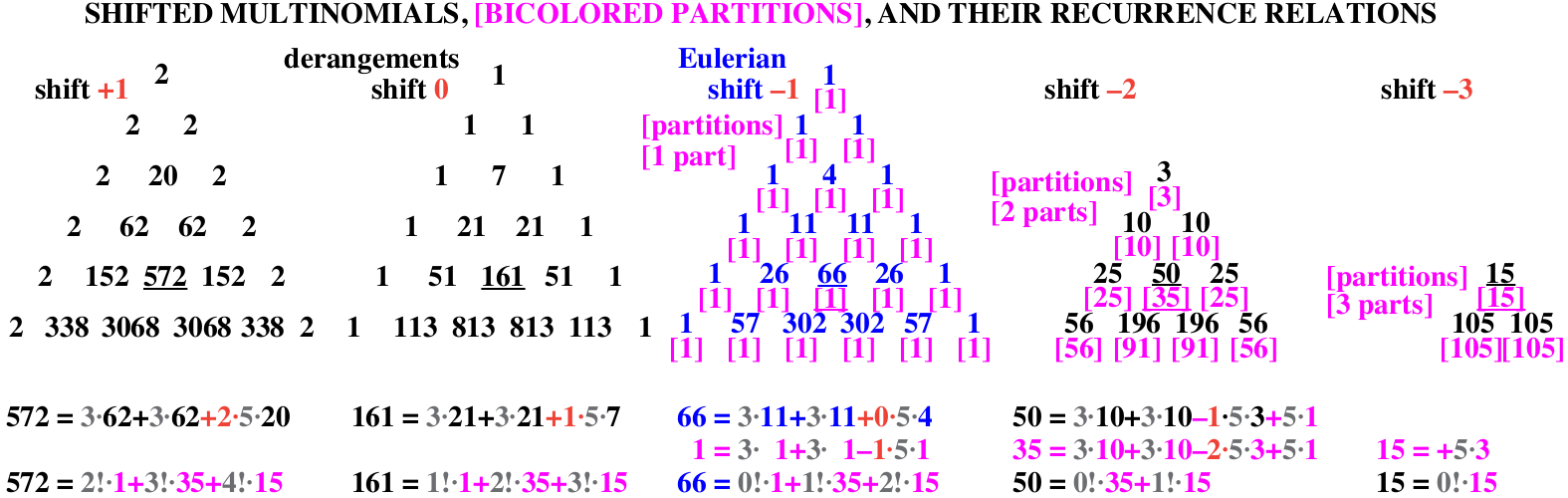}
\vskip 0.2 in

\subsection{Comments}
In the middle of the last century, colored partitions were studied by analogy with the rich structure of number partitions, but the interest waned without the realization that bicolored partitions were fundamentally related to permutations. In particular our shifted multinomials appear to be a new concept, missing at that time. In our case, the shifts in factorials are due to linear relations in the Quantum Field Theory model which motivated the combinatorics.

The shifted multinomials $N^{(s)}_{k,l}$ satisfy $N_{k,l}^{(-1)}=E_{k,l}$, the Eulerian numbers, and $N_{k,l}^{(0)}=D_{k,l},$ the derangement numbers.   $N_{k,l}^{(+1)}$ appeared in the count of certain hyperdeterminants. The boundary cases $N_{k,1}^{(-k-1)}$ are associated Stirling numbers of the second kind. 

Except for these few, according to the oeis database, no other numbers  $N_{k,l}^{(s)},\ N_{k,l}^{[r]}$ which we described above appear to have been studied before. 

The connection between the Eulerian $N_{k,l}^{(-1)}$ and derangement $N_{k,l}^{(0)}$ numbers, both obtained from the same fundamental numbers $N_{k,l}^{[r]}$ with different factorial multiplicities appears to be new as well.

In the literature, there are relations for bicolored partitions involving an arbitrarily large number of terms.  Our Eulerian type relations, which involve only immediate neighbors, appear to be new.

\section{Recurrence relations for multinomials.  Algorithms and proofs.}

We now prove the Eulerian type relations 
\begin{eqnarray*}
N_{k,l}^{[r]} &=& kN_{k,l-1}^{[r]}+lN_{k-1,l}^{[r]}+(k+l-1)(N_{k-1,l-1}^{[r-1]}-r N_{k-1,l-1}^{[r]})
\end{eqnarray*}
combinatorially, by providing a bijection between the two members.  The equality involves the unordered partitions number $N^{[r]}_{k,l}$, which counts bicolored subsets which are mutually disjoint, each with at least one element of each color, with a total of $k$ respectively $l$ elements of each color, and having the elements labeled by the numbers $1,2,\ldots,k+l$.  The sequence of sets is considered unordered, and inside each set the assignment of colors to elements is ignored.  In other words, we look at an unordered partition of \[\{1,2,\ldots,k+l\}=S_1\cup\cdots\cup S_r,\] and numbers $k_1,l_1,\ldots,k_r,l_r\in\{1,2,\ldots,\}$, with $k_i+l_i= \left\vert S_i\right\vert$ and $\sum_i k_i=k,$ $\sum_i l_i=l$.  We color each set $S_i$ as a set family, the highest $k_i$ numbers in each $S_i$ red and the other $l_i$ blue. In what follows we shall denote by ${\cal N}^{[r]}_{k,l}$ the set, having cardinality $N^{[r]}_{k,l}$, of all partitions of $\{1,2,\ldots,\}$ into nonsinglet set families with a total of $k$ reds and $l$ blues.  For instance, ${\cal N}^{[2]}_{2,2}$ consists of the partitions 

\[\{\bb{1},\rr{2}\}\cup \{\bb{3},\rr{4}\},
\{\bb{1},\rr{3}\}\cup \{\bb{2},\rr{4}\},
\{\bb{1},\rr{4}\}\cup \{\bb{2},\rr{3}\},\]
or more simply 
\[\bb{1}\rr{2}\cup \bb{3}\rr{4},\bb{1}\rr{3}\cup \bb{2}\rr{4},\bb{1}\rr{4}\cup \bb{2}\rr{3}.\]

We write the induction relation in a positive form, as 
\begin{eqnarray*}
N_{k,l}^{[r]} + r(k+l-1)N_{k-1,l-1}^{[r]} & = & (k+l-1)N_{k-1,l-1}^{[r-1]}+kN_{k,l-1}^{[r]}+lN_{k-1,l}^{[r]}.
\end{eqnarray*}
We shall describe a bijection between sets counted by the two members above.  $N^{[r]}_{k,l}$ counts partitions of $\{1,\dots,k+l\}$ into $r$ families, with a total of $k$ reds and $l$ blues. For such a partition $\cal P$ we consider 3 cases, depending on the family $S$ of its largest element $k+l$, which being largest is colored red,

\textbf{Case (A)} The family $S$ is small, i.e. it has two elements.

\textbf{Case (B)} The family $S$ has $\ge 2$ red elements (including $k+l$).

\textbf{Case (C)} The family $S$ has $k+l$ as the unique red element, and $\ge 2$ blue elements.

The second term $r(k+l-1)N_{k-1,l-1}^{[r]}$ counts triples $({\cal P},m,S)$ where ${\cal P}$ is a partition of $\{1,\dots,k+l-2\}$ into $r$ families, with a total of $k-1$ reds and $l-1$ blues, $m\in\{1,\ldots,k+l-1\}$, and $S$ is a choice of one of the $r$ families of $\cal P$.

We consider two cases. Decompose $S = R\cup B$ into its red and blue elements, where $r>b$ for any $r\in R$ and $b\in B$. Now either

\textbf{Case (D)} $m \le \max B$ or 

\textbf{Case (E)} $m > \max B$.

The corresponding cases counted by the right side member are the following.

\textbf{Case (A$'$)} The first term $(k+l-1)N_{k-1,l-1}^{[r-1]}$ counts pairs $({\cal P},m)$ between a partition ${\cal P}$ into $r-1$ families having a total of $k-1$ reds and $l-1$ blues and a number $m\in \{1,\dots,k+l-1\}$.

The second term $kN_{k,l-1}^{[r]}$ counts pairs $({\cal P},m)$ between a partition ${\cal P}$ of $\{1,\dots,k+l-1\}$ into $r$ families having a total of $k$ reds and $l-1$ blues and a number $m$ which is among the $k$ reds in ${\cal P}$. Let $S$ be the family containing $m$. We distinguish two cases.

\textbf{Case (C$'$)} The family $S$ has $m$ as its unique red element, or

\textbf{Case (E$'$)} The family $S$ has $\ge 2$ red elements, including $m$.

The third term $lN_{k-1,l}^{[r]}$ counts pairs $({\cal P},m)$ between a partition ${\cal P}$ of $\{1,\dots,k+l-1\}$ into $r$ families having a total of $k-1$ reds and $l$ blues and a number $m$ which is among the $l$ blues in ${\cal P}$.
Let $S$ be the family containing $m$. We distinguish two cases.

\textbf{Case (B$'$)} $m$ is the largest blue in $S$, or

\textbf{Case (D$'$)} $m$ is not the largest blue in $S$.

We now describe a bijection between each case $A,\dots,E$ and the corresponding case primed.

\textbf{Case (A)} $\cal P$ is a partition of $\{1,\dots,k+l\}$ into $r$ families, with a total of $k$ reds and $l$ blues. Its largest element $k+l$ is in a small family $S = \{\bb{m},\rr{k+l}\}$. Remove the family $S$ from the partition $\cal P$. Lower all labels $>m$ by 1, to obtain a partition ${\cal P}'$ of $\{1,\dots,k+l-2\}$ into $r-1$ families  having a total of $k-1$ reds and $l-1$ blues. The pair $({\cal P}',m)$ is in the \text{Case (A$'$)}.

\textbf{Case (A$'$)} Given a pair $({\cal P},m)$ between a number  $m\in \{1,\dots,k+l-1\}$  and a partition ${\cal P}$ of $\{1,\dots,k+l-2\}$ into $r-1$ families having a total of $k-1$ reds and $l-1$ blues, increase all labels $>m$ in ${\cal P}$ by 1 and append to  ${\cal P}$ the family $\{\bb{m},\rr{k+l}\}$. We obtain this way a partition ${\cal P}'$ of $\{1,\dots,k+l\}$ into $r$ families having a total of $k$ reds and $l$ blues, with its largest element $k+l$ in a small family, which is the \text{Case (A)}.

\textbf{Case (B)} We have a partition ${\cal P}$ of $\{1,\dots,k+l\}$ into $r$ families having a total of $k$ reds and $l$ blues, with the family $S$ of the largest element $k+l$ having $\ge 2$ red elements. Let $m$ denote the largest blue in $S$. Remove the element $k+l$. The remaining partition ${\cal P}'$ of $\{1,\dots,k+l-1\}$ into $r$ families has a total of $k-1$ reds and $l$ blues. The pair $({\cal P}',m)$ is in the \text{Case (B')}.

\textbf{Case (B$'$)} We have a pair $({\cal P}',m)$ between a partition ${\cal P}$ of $\{1,\dots,k+l-1\}$ into $r$ families having a total of $k-1$ reds and $l$ blues, and a number $m$ marked blue, which is the largest blue in its family $S$. Insert the element $k+l$ colored red into $S$ to obtain a partition ${\cal P}$ of $\{1,\dots,k+l\}$ into $r$ families with a total of $k$ reds and $l$ blues, with the family $S$ of its largest element $k+l$ having $\ge 2$ red elements. Then  ${\cal P}$ is in the \text{Case (B)}.

\textbf{Case (C)} We have a partition ${\cal P}$ of $\{1,\dots,k+l\}$ into $r$ families having a total of $k$ reds and $l$ blues, with the family $S$ of the largest element $k+l$ having no other red element. Remove the element $k+l$. Let $m$ be the remaining largest element in $S$.  Change the color of $m$ from blue to red. The remaining partition ${\cal P}'$ of $\{1,\dots,k+l-1\}$ into $r$ families has a total of $k$ reds and $l-1$ blues. The pair $({\cal P}',m)$ has $m$ as the unique red in its family, so is in the case \text{Case (C$'$)}.

\textbf{Case (C$'$)} We have a pair $({\cal P}',m)$ between a partition ${\cal P}'$ of $\{1,\dots,k+l-1\}$ into $r$ families having a total of $k$ reds and $l-1$ blues, and $m$ which is a number in $\{1,\dots,k+l-1\}$ marked red in ${\cal P}'$, with $m$ being the unique red in its family $S$. Insert $k+l$ into $S$ and change the color of $m$ from red to blue. The corresponding partition ${\cal P}$ of $\{1,\dots,k+l\}$ into $r$ families has a total of $k$ reds and $l$ blues. The largest element $k+l$ is the unique red in its family. The partition ${\cal P}$ is in the \text{Case (C)}.

\textbf{Case (D)} We have a triple $({\cal P},m,S)$ where $1\le m\le k+l-1$, ${\cal P}$ is a partition of $\{1,\dots,k+l-2\}$ into $r$ families, with a total of $k-1$ reds and $l-1$ blues and $S$ is a family of $\cal P$. The family $S = R\cup B$ is decomposed into its red and blue elements, and $m \le \max B$. Raise the labels of elements $\ge m$ in ${\cal P}$ by 1 and insert $m$ colored blue into $S$. We obtain a partition ${\cal P}'$ of $\{1,\dots,k+l-1\}$ into $r$ families, with a total of $k-1$ reds and $l$ blues. In it $m$ is not the largest blue in its family. The pair $({\cal P}', m)$ is in the \text{Case (D$'$)}.

\textbf{Case (D$'$)} We have a pair $({\cal P}',m)$ between a partition ${\cal P}'$ of $\{1,\dots,k+l-1\}$ into $r$ families having a total of $k-1$ reds and $l$ blues, and $m$ which is a number in $\{1,\dots,k+l-1\}$ marked blue in ${\cal P}'$ which is not the largest blue in its family $S'$. Remove $m$ and lower the elements $> m$ by 1 to get a set $S$ and a partition ${\cal P}$ of $\{1,\dots,k+l-2\}$ into $r$ families having a total of $k-1$ reds and $l-1$ blues. Then the triple $({\cal P},m,S)$ is in the \text{Case (D)}

\textbf{Case (E)} We have a triple $({\cal P},m,S)$ where $1\le m\le k+l-1$, ${\cal P}$ is a partition of $\{1,\dots,k+l-2\}$ into $r$ families, with a total of $k-1$ reds and $l-1$ blues and $S$ is a family of $\cal P$. The family $S = R\cup B$ is decomposed into its red and blue elements, and $m > \max B$. Raise the labels of elements $\ge m$ in ${\cal P}$ by 1 and insert $m$ colored red into $S$. We obtain a partition ${\cal P}'$ of $\{1,\dots,k+l-1\}$ into $r$ families, with a total of $k$ reds and $l-1$ blues, so that $m$ is not the unique red in its family. The pair $({\cal P}', m)$ is in the \text{Case (E$'$)}.

\textbf{Case (E$'$)} We have a pair $({\cal P}',m)$ between a partition ${\cal P}'$ of $\{1,\dots,k+l-1\}$ into $r$ families having a total of $k$ reds and $l-1$ blues, and $m$ which is a number in $\{1,\dots,k+l-1\}$ marked red in ${\cal P}'$, so that $m$ is not the unique red in its family $S'$. Remove $m$, lower the elements $> m$ by 1 to get a family $S$ and a partition ${\cal P}$ of $\{1,\dots,k+l-2\}$ into $r$ families having a total of $k-1$ reds and $l-1$ blues. Then the triple $({\cal P},m,S)$ is in the \text{Case (E)}

This ends the proof.  

We choose an arbitrary integer $s$ and prove the recurrence relation for $s$-shifted multinomials.  We multiply the relation obtained
\begin{eqnarray*}
N_{k,l}^{[r]} &=& kN_{k,l-1}^{[r]}+lN_{k-1,l}^{[r]}+(k+l-1)(N_{k-1,l-1}^{[r-1]}-r N_{k-1,l-1}^{[r]}).
\end{eqnarray*}
with $(r+s)!$, where negative factorials are taken to be 0. We obtain
\begin{eqnarray*}
(r+s)!N_{k,l}^{[r]} &=& k(r+s)!N_{k,l-1}^{[r]}+l(r+s)!N_{k-1,l}^{[r]}+ \\
&& +(k+l-1)((r+s)!N_{k-1,l-1}^{[r-1]}-r (r+s)!N_{k-1,l-1}^{[r]})
\end{eqnarray*}
which we write as
\begin{eqnarray*}
(r+s)!N_{k,l}^{[r]} &=& k(r+s)!N_{k,l-1}^{[r]}+l(r+s)!N_{k-1,l}^{[r]}+\\&&+ (k+l-1)(((r-1)+(s+1))!N_{k-1,l-1}^{[r-1]}-(r+(s+1)) (r+s)!N_{k-1,l-1}^{[r]})+\\
&&+(k+l-1)((s+1) (r+s)!N_{k-1,l-1}^{[r]}), \\
\end{eqnarray*}
We now sum over all $r\ge \max(-s,1)$. The sum on the second row is telescopic with respect to $r$, and we are left with   $$(k+l-1)(\max(-s,1)+s)!N_{k-1,l-1}^{[\max(-s,1)-1]}=(k+l-1)(\max(s+1,0))!N_{k-1,l-1}^{[\max(-s-1,0)]}$$
If  $s\ge-1$ then $N_{k-1,l-1}^{[\max(-s-1,0)]}=N_{k-1,l-1}^{[0]}=0$. Else if $s\le -2$ then $(\max(s+1,0))!N_{k-1,l-1}^{[\max(-s-1,0)]}=N_{k-1,l-1}^{[-s-1]}$. We obtain in all cases
\begin{eqnarray*}
N_{k,l}^{(s)} &=& kN_{k,l-1}^{(s)} +lN_{k-1,l}^{(s)} +(k+l-1)(N_{k-1,l-1}^{[-s-1]}+(s+1)N_{k-1,l-1}^{(s)})
\end{eqnarray*}
Note that when $s\ge-1$ the last relation simplifies to the positive 
$$N_{k,l}^{(s)} = kN_{k,l-1}^{(s)} +lN_{k-1,l}^{(s)} +(k+l-1)(s+1)N_{k-1,l-1}^{(s)},$$ In particular this gives for $s=-1$ the relations satisfied by Eulerian permutations 
\begin{eqnarray*}
N_{k,l}^{(-1)} &=& kN_{k,l-1}^{(-1)} +lN_{k-1,l}^{(-1)}
\end{eqnarray*}
and for $s=0$ the relations satisfied by derangements.  
\begin{eqnarray*}
N_{k,l}^{(0)} &=& kN_{k,l-1}^{(0)} +lN_{k-1,l}^{(0)} +(k+l-1)N_{k-1,l-1}^{(0)}
\end{eqnarray*}

\section{The structure of a permutation: families as building blocks}
\subsection{The entry/exit lemma}

Recall that a nonsinglet family $S$ written as a bicolored set is $S=\{\bb{b_1},\bb{b_2},\dots,\bb{b_l},\rr{a_1},\rr{a_2},\dots,\rr{a_k}\}$, with $\ge 1$ reds and $\ge 1$ blues,  in which any red is larger than any blue.  The same family written as a sequence is $S=(a_1,\dots,a_k,b_l,\dots,b_1)$ with $k,l \ge 1$ and $b_1<b_2<\cdots<b_l<a_1<a_2<\cdots<a_k$, or as a colored sequence $S=(\rr{a_1},\dots,\rr{a_k},\bb{b_l},\dots,\bb{b_1})$.

{\bf Definition.} Let $S$ be a nonsinglet family, written as a sequence. Let $m\in \R$ be such that $m>\max S$. Denote by $S'$ the sequence obtained by inserting $m$ in the sequence $S$ on any position other than the last at the end of $S$. Let $x\in S$ be the element located immediately after $m$ in the sequence $S'$, which we call the \textbf{friend of $m$ in $S$.} We call $x$ an {\bf exit} element of $S$ (or say that $x$ is in an exit position in $S$) if $S'$ is \textbf{not} a family. 

\textbf{Remarks.} A family with 2 elements, which we call a {\bf small family}, has no exit elements, while a family with $\ge 3$ elements, which we call a {\bf large family}, remains a family after the removal of any of its elements.

In the algorithms which follow, the friend $x$ of $m$ as above will be pulled out of its family $S$ and will start a new sequence family, $(m,x)$. From the perspective of the old family $S$, $(m,x)$ will be called a {\bf bud}, as it starts a new branch of the family tree. From the point of view of the newly started family, $(m,x)$ will be called the {\bf family root}, with $x$ called the {\bf family founder} and $m$ the {\bf family cofounder}.

{\bf Definition}
Let $S$ be a nonsinglet family, written as a bicolored set. Let $x\in \N$ be such that $x\not\in S$. Let $S'=S\cup \{x\}$, bicolored so that it is a family and the restriction of its coloring to $S$ is the coloring of $S$. We say that $x$ had a {\bf regular entry} in $S'$, or was {\bf regularly inserted} into $S'$, if the color of $x$ in $S'$ is determined as follows.  In all cases except for one, the color of is determined by the requirement that after $x$ enters, $S$ is a family.  The only exception is when the value $x$ is between the highest blue and the lowest red.  In that case, placing $x$ at the beginning, as a red, or between the reds and blues, as a blue, will both give families.  The algorithms in which this process is used require that after $x$ enters $S'$, then $x$, as a friend should be possible to extract by a higher element placed before it in $S$.  This is possible only in the case in which between these two choices $x$ enters as a red in front, and we shall call this case a regular entry.  Thus a regular entry is characterized by the fact that it can be reversed by an exit, which is the content of the entry/exit lemma.  Let $S=\{\bb{b_1},\bb{b_2},\dots,\bb{b_l},\rr{a_1},\rr{a_2},\dots,\rr{a_k}\}$, with  $b_1<b_2<\cdots<b_l<a_1<a_2<\cdots<a_k$ as above.  If $x<b_l$ then $x$ is colored {\bf blue} in $S'$.  If $x>b_l$ then $x$ is colored {\bf red} in $S'$.

{\bf Remark} If $x<b_l$, respectively $x>a_1$ then $x$ must be colored blue, respectively red in order for $S'$ to be a family.  The only multiple choice of color is when $b_l<x<a_1$ and in that case a regular entry will color $x$ red, in view of the result which follows, which connects the horizontal (sequence order) of a family in sequence form with the vertical structure of the same family as a set family. 

\begin{thm} {\bf (the entry/exit lemma)}

If $S$ is a nonsinglet family, $x\in S$ is in an exit position, and $S'$ denotes the family $S$ with $x$ removed, then the regular insertion of $x$ into $S'$ gives back $S$.

If $S$ is a nonsinglet family, and $x\not\in S$ is inserted regularly into $S$ to obtain a family $S'$, then $x$ is in an exit position in $S'$.
\end{thm}
\begin{proof}
The only nonexit positions in a family $S$ are the {\bf highest blue} and the {\bf unique red}.  If $S=(a_1,\dots,a_k,b_l,\dots,b_1)$ with $b_1<b_2<\cdots<b_l<a_1<a_2<\cdots<a_k$ and $m>\max S$, then $S'=(a_1,\dots,a_k,m,b_l,\dots,b_1)$ with $b_1<b_2<\cdots<b_l<a_1<a_2<\cdots<a_k<m$ is a family, in which the friend of $m$ is $b_1$, the highest blue.  

In case $k=1$, i.e. if there is a unique red, then $S'=(m,a_1,b_l,\dots,b_1)$ with $b_1<b_2<\cdots<b_l<a_1<m$ is a family in which the unique red $a_1$ in $S$ has {\bf changed color to blue} in $S'$, and $m$ is now the unique red in $S'$. It is easy to check, and left to the reader (see also the illustration) that if the friend of $m$ is any element other than the highest blue or the unique red, then $S'$ is not a family, so that a position in $S$ is an exit position iff it is not the highest blue or the unique red. Note that in a small family $\vert S\vert = 2$, the 2 elements are the unique red and the highest blue, so there are no exit positions.

We now check that if $x\in S$ is in an exit position and is taken out of $S$ to obtain $S'$, then if $x$ is reintroduced regularly in $S'$ then $x$ recovers the color it had in $S$, and thus $x$ recovers its position in $S$ as a sequence family.  If $x$ was blue, but not the highest blue, it will enter below the highest blue and be colored blue. If $x$ is red, but not the unique red, then the highest blue in $S$ will remain the highest blue in $S'$, so $x$ will enter above the highest blue in $S'$ and with a regular entry will be colored red as it was before.

For the second part of the statement, remark that in a regular entry of $x$ in $S$, either $x$ is below the highest blue of $S$ and will be colored blue, but will not become the highest blue, or $x$ is above the highest blue of $S$ and will be colored red. In the latter case $x$ will not become the unique red, since $S$ had at least one red already. Thus in both cases $x$ will be in an exit position in $S'=S\cup \{x\}$, since the only nonexit positions are the highest blue and the unique red.
\end{proof}

Recall that a nonsinglet family $F$ written as a bicolored set is $\{\bb{b_1},\bb{b_2},\dots,\bb{b_l},\rr{a_1},\rr{a_2},\dots,\rr{a_k}\}$, written as a sequence is $(a_1,\dots,a_k,b_l,\dots,b_1)$ or as a colored sequence $S=(\rr{a_1},\dots,\rr{a_k},\bb{b_l},\dots,\bb{b_1})$, with $k,l \ge 1$ and $b_1<b_2<\cdots<b_l<a_1<a_2<\cdots<a_k$; its subjacent set is $\{b_1,\dots,b_l,a_1,\dots,a_k\}$.  A singlet family is as a set $\{s\}$ and as a sequence $(s)$.

{\bf Definition} A (family) \textbf{registry} ${\cal R}$ is an ordered list, or sequence, of mutually disjoint families ${\cal R}=(F_1,\dots,F_k)$. The set $S=\bigcup_i F_i$ is called the range of ${\cal R}$. Let $\phi$ be the order preserving bijection $S\to \{1,\dots,|S|\}$. The pattern ${\cal R}'=(F'_1,\dots,F'_k)$ of ${\cal R}$ is the registry in which $F'_i = \phi(F_i)$ and $\phi$ preserves colors.

A {\bf stripped registry} is a registry with no singlets, From a full registry one obtains a stripped registry by removing the singlets.

A registry ${\cal R}$ is decomposed into blocks by successive minima in the same way as a permutation. Namely, let $m_1$ be the minimum of all members of all families in ${\cal R}$, found in a family $F_{i_1}$ Let $m_2$ be the minimum of all members of all families located after $F_{i_1}$, found in a family $F_{i_2}$, etc. With $i_0=0$ we have now registries ${\cal R}^{(l)}=(F_{i_{l-1}+1},\dots,F_{i_l})$ called block subregistries which concatenate to ${\cal R}$.
For a block of a permutation, its last element, which is also its minimum, is called the block anchor. Similarly, for a block subregistry, the minimum of all the families in it is called the anchor.

Suppose that $\cal R$ is a registry with support $\overline n$ in which the only singlets are singlet blocks. Let ${\cal R}'$ be the registry from which all the singlet blocks have been stripped, i.e. removed. Let $S$ denote the complement of the range of ${\cal R}'$ in $\overline n$. Insert the elements of $S$ as singlets into ${\cal R}'$ successively, starting with the smallest one, in increasing order. If the smallest singlet is $(1)$, insert it on the first place. Else compute the blocks of ${\cal R}'$, which end in the sequence of successive minima. Insert the new singlet $(s)$ on the position immediately to the right of the last block which has minimum $\ge s$. By induction, the union of the subjacent sets of the blocks with minima $< s$ is the interval $\overline{s-1}$. The last successive minimum of these blocks is $<s$. We insert $(s)$ immediately after the last of these blocks. That will make $((s))$ a singlet block. Continue until the set $S$ is finished and one obtains a registry with range $\overline n$.

\begin{thm}
 (the stripping lemma)
 If $\cal R$ is a registry with support $\overline n$ in which the only singlets are singlet blocks, if one strips away these singlets and then one reintroduces them with the above algorithm, then the resulting registry is $\cal R$.
 
\end{thm}

The main result of this section is the following decomposition of any permutation into an ordered list of very simple pieces, families, with no restrictions on the image. 

\begin{thm}
Permutations $\pi$ of $\overline n$ are in bijection with the stripped registries ${\cal R}=(F_1,\dots,F_k)$ for which, as sets, $\bigcup_i F_i\subset \overline n$.

The bijection maps each singlet block of $\pi$ into the empty registry. 

It maps each nonsinglet block of $\pi$ into a registry block with the same support, with pattern the image of the pattern of that block.

The images of individual blocks are concatenated in order.

\end{thm}

The algorithm applies to individual blocks separately, and the relative order and colors within the resulting families depends only on the relative order within a block. The last element of the block, which is its minimum, is in the last set of the corresponding block of the registry. 

As such, the algorithm can be extended to arbitrary injective maps which have finite blocks between successive minima. 

Composed with the cycle transform, the algorithm can be extended to arbitrary permutations of linearly ordered sets which have finite cycles. 

{\bf The idea of the algorithm}   

We create out of a permutation, inductively by inserting the largest element, an ordered list of mutually disjoint families which is a registry. In it the singlet families carry no information and can be stripped.

In the opposite direction of the algorithm, we map such a registry back to a permutation.  We remove the largest element, reduce the remaining registry to a smaller registry, map that registry to a permutation by induction, and insert into the latter the largest element. 

The two transforms are inverse to each other. The input of the reverse algorithm is an arbitrary set of mutually disjoint nonsinglet subsets of $\overline n$, organized as families by specifying the number of ascents of each. The remaining elements of $\overline n$ are singlet families and can be added to the stripped registry before the procedure, on set places.

In the bijection the ascents and descents (reds and blues) of the families in the registries end up precisely as the nonsinglet ascent respectively descent values of the permutation.

After an inverse cycle transform the ascents and descents of families become the values over and respectively under the diagonal, while singlets map to values on the diagonal. This way the algorithm allows us to count and construct all permutations with specified elements under and over the diagonal.

The decomposition of the graph of the permutation into families is well-defined and canonical.  Arranging the families in a linear order, so that they form a registry, leaves a certain amount of freedom. There are several different algorithms which are reversible, i.e., they result in bijections between registries and permutations.  We shall use an algorithm which, although more elaborate, has the advantage that it preserves in the registry much of the structure of the permutation itself.  

There are two distinct aspects of the algorithm, each of which is reversible. 

In the direct algorithm the highest new element of the permutation, when it does not fit into the family of the friend, pulls out its friend to start a new small family of two. In the reverse algorithm, the highest element is taken out of its small family of two, and leaves the friend, the other member, as an orphan in search of a family.

The first aspect concerns the location of the newly created family and of the remaining old family in the registry.  In the reverse algorithm, the orphan looks for and finds a family to enter and then places the newly enlarged family in the registry.  These two processes must be inverse to each other. There are different possibilities for this step which lead to different variants of the algorithm.

The second aspect concerns the position inside its family of the friend who is taken out to form a new small family.  This is the content of the entry/exit lemma.  The positions inside a family from which a friend cannot be taken out are the highest blue and the unique red.  For the reverse algorithm, the orphaned friend enters its newfound family, and finds its relative position and its color in that family (see the illustration).  Only the value (the height) of the orphan is known.  The regular entry algorithm directs the orphan in its new family precisely on the positions other than the highest blue and the unique red.  This ensures that exiting a family from above in the direct algorithm, and entering regularly a family from the side in the reverse algorithm, are inverse to each other.

{\bf Definition} Given a family $F$, $m\in F$ and an element $n>\max F$ we say that $n$ fits (or can enter) into the family $F$ with $m$ as its friend if inserting $n$ immediately before $m$ into the family $F$ written as a sequence satisfies the definition of a sequence family.

Whether $n$ fits in $F$ with $m$ as its friend depends only on $m$, not on $n$. If $F = (m)$, $n$ always fits. In the entry/exit lemma, we called the positions of $m$ for which $n$ does not fit {\bf exit} positions. Let us reformulate that lemma as follows.

Given a nonsinglet family $F$, $m\in F$ and an element $n>\max F$, let $F'$ denote the family $F$ with $m$ removed.
The following are equivalent.

$n$ does not fit into $F$ with $m$ as a friend

$F$ is obtained by the regular entry of $m$ into $F'$.

{\bf Example} 11 fits into $(6,8,9,3,2)$ with 3 as friend, since $(6,8,9,11,3,2)$ is a family, but not with 6 as friend, since $(11,6,8,9,3,2)$ is not a family. The latter property is equivalent by the entry/exit lemma with the fact that inserting regularly $6$ into its complement $(8,9,3,2)$ gives $(6,8,9,3,2)$.

\begin{center}
\includegraphics[width=1.0\linewidth]{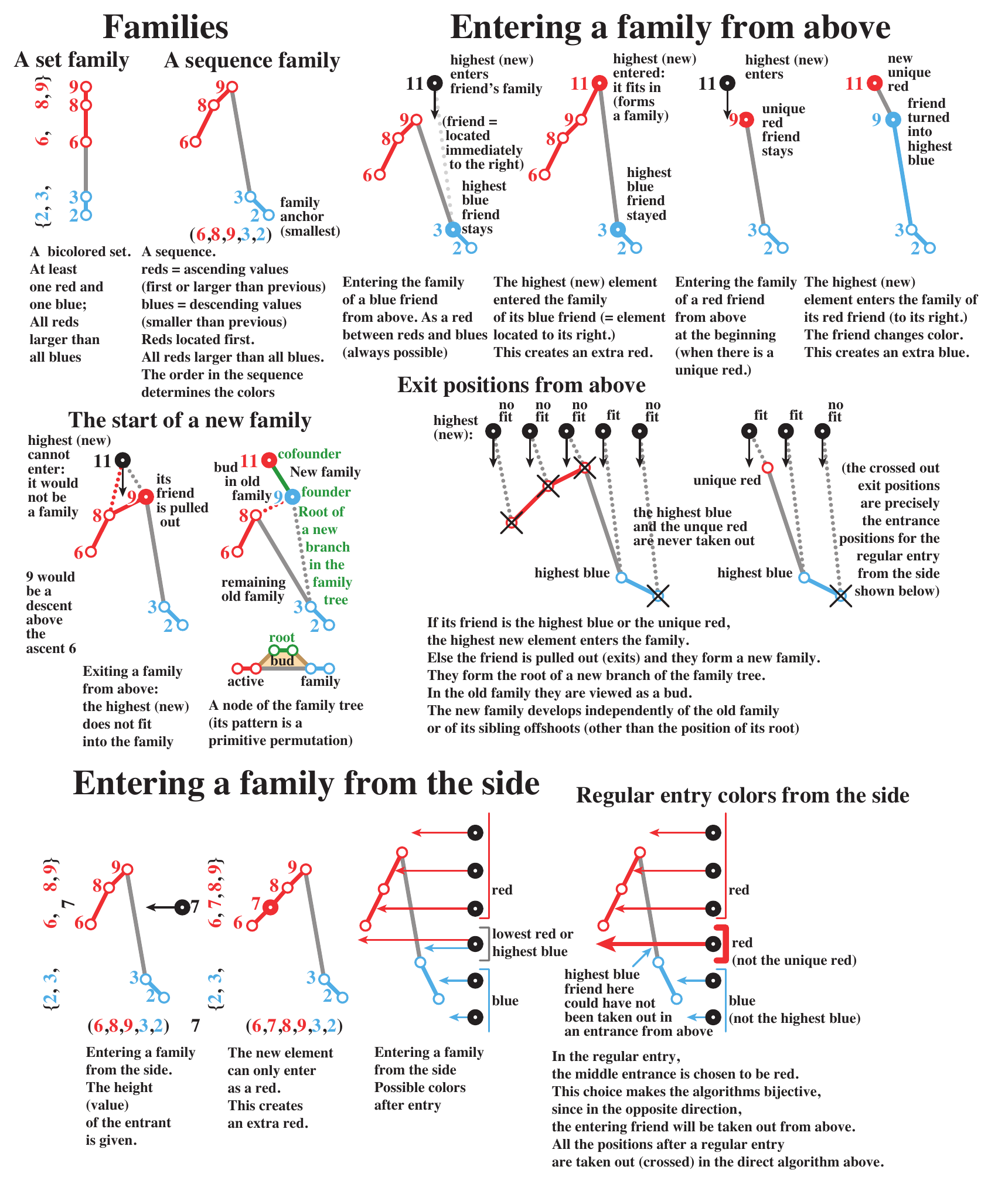}
\end{center}

{\bf Definition} We call the smallest element of a family or block of a permutation or block of a registry the {\bf anchor}. The anchor is located on the last position in a sequence family or permutation, and in the last position of the last family of a family block.

{\bf Definition} Call a family $F$ in a registry ${\cal R}$ in a {\bf left slide} position if it is not on the first position and if the anchor of the family $F'$ immediately to its left is bigger than the anchor of $F$. The name suggests that $F$ can slide under $F'$ when ordered by their anchors. We say that $F$ has {\bf slid to its left wall} if $F$ has exchanged positions in the registry, successively, with families immediately to its left under which it can slide, until it is no longer in a slide position, i.e. is either on the first (leftmost) position, or the family to its left has a smaller anchor than $F$.

Define similarly a {\bf right slide position} for $F$ and the slide of $F$ to its right wall.

{\bf Example} In the registry $((9,3,1),(7,4),(8),(6,5,2))$ the family $(6,5,2)$ is in a left slide position, since $2<8$ and slides left as $((9,3,1),(7,4),(6,5,2),(8))$ and $((9,3,1),(6,5,2),(7,4),(8))$ where it is at its left wall, since $1<2$.
We now start to describe the two directions of the algorithm.

{\bf Permutation to Registry (the direct algorithm)}.

We start with a permutation $\pi$ of $\overline{n}$. Let $\pi'$ be  the permutation obtained by removing from $pi$, written in line notation, its highest value $n$. We assume by induction that a registry ${\cal R}'$ was constructed for $\pi'$. We build a registry ${\cal R}$ for $\pi$.

\textbf{Case (A)}.  If $n$ is last in $\pi$ then we let $n$ form a singlet family $(n)$ which is placed at the right end of the registry ${\cal R}'$ to form the registry ${\cal R}$. 
Remark that Case (A) is precisely the case in which $n$ forms a singlet block in $\pi$. 

If $n$ is not last, the element $m$ which is positioned after $n$ in $\pi$ is called the {\bf friend} of $m$. Let $F'$ be the family of the friend $m$ in ${\cal R}'$.

If $n$ can be inserted into $F'$ with $m$ as friend, we insert $n$ immediately before $m$ to obtain a family $F$. We replace $F'$ in the registry ${\cal R}'$ by $F$ to obtain ${\cal R}$.

We distinguish two cases.

\textbf{Case (B)}. $F'$ is a singlet $(m)$ which by induction will be last in ${\cal R}'$. Inserting $n$ immediately before $m$ produces a last block $(n,m)$ of ${\cal R}$.

\textbf{Case (C)}. $F'$ is not a singlet, and $n$ can be inserted into it in front of $m$ to form a family $F$.

If $n$ does not fit into the family $F'$ with $m$ as friend, $m$ is pulled out of $F'$ to obtain a family $F$, and $m$ together with $n$ form a new family $N = (n,m)$. The old family $F'$ is removed from ${\cal R}'$ and $F$ and $N$ are inserted to form the registry ${\cal R}$. 

For the positions of $F$ and $N$ into ${\cal R}$ we distinguish two cases.

\textbf{Case (D)}. If the friend $m$ was the anchor of its nonsinglet family $F'$, in ${\cal R}'$ replace $F'$ by the pair $(F,N)$. Slide $F'$ (with $m$ in it) to its left wall, then remove $m$ from it to leave $F$. Leave the new family $N$ in the former place of $F'$. This gives ${\cal R}$.

\textbf{Case (E)}. If the friend $m$ was not the anchor of its family $F'$, in ${\cal R}'$ replace $F'$ by the pair $(N,F)$. Slide $N$ to its left wall and leave $F$ in the place of $F'$ to obtain ${\cal R}$.

This ends the direct algorithm.  

{\bf Comments}

Assume by induction that the blocks of ${\cal R}'$ were, in order, transforms by the above algorithm of the blocks of $\pi'$. In Case (A) a new singlet registry block $(n)$ was appended to ${\cal R}'$ at the end, and $\pi'$ has the same $n$ at the end as a block. This comes from the fact that the last successive minimum of both $\pi'$ and ${\cal R}'$ is their last element $n'$, and when appending $n$ the element $n'$ will be the next to last successive minimum, with $n$ the last one.

In the other cases $n$, being largest, does not change the successive minima. In Case (B) the new family $F$ stays in the place of the old one $F'$. In Case (D) and Case (E) the new families $F$ and $N$ never move to the right of the former position of $F'$. The family among $F$ and $N$ which contains the smallest element of $F'$ remains in the place of $F'$. The other one (which could not contain one of the successive minima) slides to the left only past families with a larger minimum or anchor, so it remains within the same block.

Thus the blocks of the registry ${\cal R}$ have as support the blocks of the permutation $\pi$, in order. As the step of the algorithm worked only within a block and depended only on the relative positions and relative values of the elements in it, the algorithm transforms each block $\pi^{(l)}$ of $\pi$ into a block ${\cal R}^{(l)}$ of ${\cal R}$ independently of the other blocks, and the pattern of the resulting ${\cal R}^{(l)}$ depends only on the pattern of $\pi^{(l)}$. 

This shows that it is enough to know the algorithm on single block permutations, i.e. on permutations with 1 on the last position. These will be mapped to registries with one block, having 1 on the last position of the last family. From this, the case of a general permutation $\pi$ is obtained as follows. For each block $\pi^{(l)}$ of $\pi$, let $\phi^{(l)}$ and $\psi^{(l)}$ denote the order preserving maps for which $\widetilde{\pi}^{(l)}=\phi^{(l)}\circ\pi^{(l)}\circ\psi^{(l)}$ is a permutation, the pattern of $\pi^{(l)}$. Let $\widetilde{\cal R}^{(l)}$ be the registry produced by the algorithm from $\widetilde{\pi}^{(l)}$. Let ${\cal R}^{(l)}$ be the registry obtained by applying $\left(\phi^{(l)}\right)^{-1}$ to each element of every family of $\widetilde{\cal R}^{(l)}$. Now concatenate ${\cal R}= \bigcup_l {\cal R}^{(l)}$ to obtain the final registry.

Since the singlet blocks of $\pi$ are mapped into the singlet blocks of $\cal R$ in order, their image carries no information in $\cal R$ and can be stripped off. Conversely, given a registry $\cal R'$ with no singlets and a range $\overline n$, the complement of the range of $\cal R'$ in $\overline n$, organized as singlet families, can be inserted as blocks between the blocks of $\cal R'$ in order to obtain a registry $\cal R$ in which the added elements are singlet blocks.

For the reverse algorithm, registry to permutation, given a stripped registry $\cal R'$ we complete it with singlets inserted as singlet blocks as above to make its range an interval, before proceeding.

\textbf{\bf Registry to Permutation (the reverse algorithm)}.

We start with a registry ${\cal R}$ in which the singlet families form the singlet blocks of ${\cal R}$. We construct, distinguishing cases corresponding to the ones in the direct algorithm, a registry ${\cal R}'$ not containing the highest element $n$ from ${\cal R}$. We construct by induction a permutation $\pi'$ from ${\cal R}'$ and we insert $n$ into $\pi'$ to obtain $\pi$.

We distinguish four cases.

\textbf{Case (A)$'$}. $n$ is in a singlet family $(n)$ of ${\cal R}$. By our hypothesis, $(n)$ is a block of ${\cal R}$, so since $n$ is the largest element in ${\cal R}$, $(n)$ will be placed last in ${\cal R}$. Remove $(n)$ from ${\cal R}$ to obtain ${\cal R}'$, map by induction ${\cal R}'$ to a permutation $\pi'$ and add $n$ at the end of $\pi'$ to get the permutation $\pi$, in which $n$ will form the last block.

In the rest of the algorithm, $n$ is in a family $F$ which is not a singlet. Being largest, $n$ cannot be on the last place in its family.  Call the friend of $n$ the element $m$ immediately to its right in $F$. 
We shall remove $n$, find a family for its friend $m$ and position it in the registry to obtain a registry ${\cal R}'$, map ${\cal R}'$ to a permutation $\pi'$, inductively, and then insert $n$ in $\pi'$ in front of its friend $m$ to obtain $\pi$. 

We describe the way in which we obtain ${\cal R'}$ for the inductive step. 

\textbf{Case (C)$'$}. $n$ is in a large family $F$ of ${\cal R}$. Being largest, $n$ cannot be on the last place in $F$, written as a sequence. Remove $n$ from $F$ to obtain a nonsinglet family $F'$, and call ${\cal R}'$ the registry obtained by replacing $F$ by $F'$ in ${\cal R}$.

Remark that inserting or removing the largest element $n$ on any position other than as a singlet block (which case was treated as Case (A)) does not change the block structure of $\pi$ or ${\cal R}$.

In the remaining cases, $n$ is in a small family $N=(n,m)$ of $ {\cal R}$ with its friend $m$. Removing $n$ from $N$ leaves the friend $m$ as an orphan looking for a family. We distinguish two cases.

If $(m)$ as a singlet family placed in the position of $N$ in ${\cal R}$ is not in a left slide position, then let $(m)$ slide to its right wall. 

\textbf{Case (B)$'$}. If $(m)$ slid right to the last position, keep it as a singlet family there to obtain the registry ${\cal R}'$. Map ${\cal R}'$ inductively to a permutation $\pi'$ in which $(m)$ will be the last block (and thus $m$ is the largest in ${\cal R}'$.) Insert now $n$ before its friend $m$ in $\pi'$ to obtain $\pi$, where $(n,m)$ will be the last block. 

\textbf{Case (E)$'$}. If $(m)$ slid to its right wall and is not last, let $F$ be the family to the right of $(m)$. $F$ cannot be a singlet $(p)$, since in that case $(p)$ would be a singlet block. $N=(n,m)$, being to its left, would be in an earlier block, so the minimum $m$ of $N$ would be smaller than $(p)$ so $(m)$ could have slid past $(p)$ to the right, contradicting the hypothesis that $(m)$ was at its right wall.
So $F$ is not a singlet family. Insert now in a regular way $m$ into $F$ to obtain a large family $F'$ and thus obtain ${\cal R'}$. 

\textbf{Case (D)$'$}. In the remaining case $(m)$ as a singlet family placed in the position of $N$ in ${\cal R}$ is in a left slide position. Make $(m)$ slide to its left wall. Let $F$ be the family immediately to its right. Insert $m$ into $F$ in a regular way to obtain a family $F'$. As $(m)$ has slid past $F$, the anchor of $F$ is larger than $m$. Thus when introduced in $F$, $m$ will become the new minimum, or anchor, of $F'$. Move now $F'$ on the position occupied by $N$ in ${\cal R}$ to obtain ${\cal R'}$.

This ends the inverse algorithm.  

We now prove that the algorithm is bijective, case by case. In the direct case, $n$ is the largest element of a permutation $\pi$, which produces a registry ${\cal R}$. Removing $n$ leaves a permutation $\pi'$ which maps to a registry ${\cal R}'$.
By induction, we may assume that ${\cal R}'$ gives back $\pi'$ by the reverse algorithm. We have to show that ${\cal R}$ gives back $\pi$ by the reverse algorithm.

In Case (A), the new element $n$ is at the end of the permutation $\pi$. In the registry ${\cal R}$, $n$ forms a singlet placed at the end of the registry ${\cal R}'$. In reverse, this is Case (A)$'$ in which $n$ is in a singlet family $(n)$ at the end of ${\cal R}$ and is removed, to obtain ${\cal R}'$. That gives back $\pi'$, in which $n$ is inserted at the end to obtain $\pi$.

In the remaining cases, $n$ is not last in $\pi$ and the element following it is called its friend $m$. $n$ is removed to obtain a permutation $\pi'$ which maps to a registry ${\cal R}'$. We also mapped $\pi$ to a registry ${\cal R}$. 

For the reverse algorithm, we define a friend 
of $n$ in ${\cal R}$. We must make sure that this friend is $m$. 
Then ${\cal R}$ is reduced to a registry not containing $n$. 
We must make sure that this registry is ${\cal R}'$. 

By induction, ${\cal R}'$ is mapped back onto $\pi'$, into which $n$ is inserted in front of its friend $m$, which gives back $\pi$ and ends this part of the proof.

In Case (B), $m$ is last in $\pi'$ and forms a block. Its family $F'$ in ${\cal R}'$ is a singlet $(m)$ positioned at the end, as a block. In ${\cal R}$ the last block is $(n,m)$. Thus in the reverse algorithm, $m$ will be the friend of $n$, and we shall be in the Case (B)$'$ in which after removing $n$, the remaining family is the singlet $(m)$ positioned at the end. This is ${\cal R}'$, and we are done.

In Case (C), $m$ is in ${\cal R}'$ in a nonsinglet family $F'$, $n$ is inserted into $F'$ in front of $m$ to form a large family $F$ which is replacing $F'$ in ${\cal R}'$ to obtain ${\cal R}$. In the reverse direction $n$ is in the large family $F$, so we are in Case (C)$'$. The element following $n$ in $F$ is its friend, and is thus $m$. The new registry is obtained by removing $n$ from $F$, and is thus ${\cal R}'$, and we are done.

In the remaining cases $n$ does not fit into the family $F'$ of $m$ with $m$ as a friend, $m$ is pulled out leaving the family $F$ behind, and forms a new small family $N=(n,m)$. In the reverse algorithm $n$ will be pulled away and its friend will be $m$ which is left as an orphan. We have to show that the registry obtained from ${\cal R}$ is ${\cal R}'$. 

We are in the cases (D)$'$ and (E)$'$, as the other cases were exhausted. In the inverse algorithm, $n$ will be removed from $N$ to leave its friend $m$ an orphan. This is the most subtle part of the argument.

In Case (D), the friend $m$ was the anchor of its nonsinglet family $F'$. In ${\cal R}'$ $F'$ was replaced by the pair $(F,N)$. $F'$ with $m$ in it slid to its left wall, then $m$ was removed from $F'$ to leave $F$ and joined $N$ which stayed in the former place of $F'$ to obtain ${\cal R}$. Note that in this case, as $m$ was the smallest in $F'$, the smallest remaining element of the remaining $F$ will be bigger than $m$. Moreover, $F$ will slide to its left wall past families with even bigger minima. In any case, the family to the left of $m$ has a minimum larger than $m$, and this shows that we are in Case (D)$'$, since $(m)$ can slide to the left. $(m)$ does now slide past families with larger minima to its left wall, which is precisely where it left the rest $F$ of its former family $F'$. $(m)$ makes one last left slide past its former family $F$ which has a minimum larger than $m$, before ending at its left wall. $F$ is now immediately at $m$'s right, where in Case (D)$'$ $m$ finds it. Now $m$ enters $F$ in a regular way. The entry/exit lemma guarantees that $m$ will end up precisely in the place from which it was taken out by $n$ in Case (D), in this case the anchor, or smallest position, and thus the family $F'$ is  reconstituted. Then the family $F'$ is moved on the position of $N=(n,m)$ in $R$ which is precisely the place from which it was taken away in Case (D). So with $F'$, as it was before, in its old place we recovered ${\cal R}'$ and we are done.

In this case the idea was that the family head $m$, in a position of responsibility, could not leave home. So it sent its family on vacation South, to the left (block anchors get smaller, and thus move Southward to the left). To make sure that its family doesn't get lost, $m$ accompanied them all the way to $m$'s left wall and left them there, where they would be easy to find. Back home, $m$ remembered the situation by the fact that there was slide room to its left. Home alone, $m$ started a new family with $n$, but that didn't go very far, since $n$ left.  So $m$ went back, sliding all the way to its left wall, to find its former family there. Its former family gladly gave $m$ the honor position which it used to have. Then $m$ brought them all back to their old homestead, and all was happy like before (which is what bijections are all about.)

In Case (E), the friend $m$ was not the anchor of its nonsinglet family $F'$. In this case $N=(n,m)$ slid to its left wall. Thus after $n$ is removed, $(m)$ cannot slide further left, which shows that we are in Case (E)$'$. So $(m)$ slides right. It moves to its right past all the families under which $(n,m)$ slid left, but cannot slide right past its former family $F$, since $m$ was not the anchor, so $F$ has a smaller minimum than $m$. Thus $m$ has found its former family $F$ at the right of its right wall, and enters it regularly. By the entry/exit lemma, $m$ recovers its former position, to form $F'$ in the place of $F$, which is precisely the place where $F'$ used to be in ${\cal R}'$, and we are done.

In this case the idea was that a family member $m$, not being the family head, took some time off and went on vacation to its left wall. $m$ remembered that by the fact that there was no more room to go to further to its left. While there, $m$ started a family with $n$, but as vacation flings go, $n$ left. So $m$ went back sliding to its right, to be stopped only by its former family. It reentered it in its old place and, once again, all was the way it used to be. A happy ending for bijective proofs.

This ends the bijectivity proof in one direction.

For the other direction, we start with a registry ${\cal R}$. We remove from it the largest element $n$ and change it to a registry ${\cal R}'$. By induction ${\cal R}'$ is mapped onto a permutation $\pi'$ which by the direct algorithm gives back ${\cal R}'$. When $n$ is inserted into $\pi'$ to give $\pi$, and $\pi$ is mapped onto a registry, we must show that that registry is ${\cal R}$.

In Case (A)$'$, $n$ is in a singlet family $(n)$, the last family of ${\cal R}$. ${\cal R}'$ is obtained by removing $(n)$, is mapped into the permutation $\pi'$. $n$ is inserted at the end of  $\pi'$, so we are in Case (A). The new registry is obtained by appending $(n)$ at the end of ${\cal R}'$, and that is precisely ${\cal R}$.

In the remaining cases, $n$ is in a family $F$ (written as a sequence) which is not a singlet, and has a friend $m$ to its right in $F$. $n$ is removed from $F$ to obtain a new registry ${\cal R}'$ which is mapped into the permutation $\pi'$. $n$ is introduced in $\pi'$ immediately before its friend $m$ to obtain $\pi$. By induction, $\pi'$ is mapped back into ${\cal R}'$. We need to show that $\pi$ is mapped into ${\cal R}$.

In Case (C)$'$, the family $F$ of $n$ is large, with $\ge 3$ elements, $n$ is removed from it to get a family $F'$ with $\ge 2$ elements, which replaces $F$ and thus one obtains ${\cal R}'$. So $F'$ is not a singlet. Inserting $n$ back in front of its friend $m$ produces $F$, which is a family. So $n$ can be inserted, and we are in Case (C). We replace $F'$ with $F$ and we obtain the registry ${\cal R}$ from which we started.

In the remaining cases, $n$ forms,  with its friend $m$, a small family $N=(n,m)$ of ${\cal R}$. Removing $n$ from $N$ leaves the friend $m$ an orphan.

If $(m)$ as a singlet family placed in the position of $N$ in ${\cal R}$ was not in a left slide position, we let $(m)$ slide to its right wall.

Case (B)$'$. If $(m)$ was on the last position, we kept it as a singlet family there to obtain the registry ${\cal R}'$. We mapped ${\cal R}'$ inductively to a permutation $\pi'$ in which $(m)$ will be the last block (and thus $m$ is the largest in ${\cal R}'$.) We inserted $n$ before its friend $m$ in $\pi'$, to obtain $\pi$ where $(n,m)$ will be the last block. This is precisely the Case (B). $\pi'$ is mapped to ${\cal R}'$, and the registry obtained by replacing $(m)$ by $(n,m)$ in ${\cal R}'$ is precisely the ${\cal R}$ we started from, so $\pi$ is mapped into ${\cal R}$.

Again, the remaining cases are more delicate. We started from ${\cal R}$ in which $n$ was in a small family $N=(n,m)$. We removed $n$ and were left with $(m)$.

Case (E)$'$. $(m)$ could not slide left, so $(m)$ slid to its right wall and was stopped by a family $F$ from sliding further right. $F$ was a nonsinglet family into which $m$ was inserted in a regular way, to obtain a large family $F'$. With $F'$ in the place of $F$, we called the registry  ${\cal R'}$. 

Note that to get back $\cal R$ out of ${\cal R}'$, we pull $m$ out of $F'$, leave the remaining $F$ in place of $F'$, start a small family $(n,m)$and slide it to its left wall.

${\cal R'}$ was inductively mapped to a permutation $\pi'$ which was then sent inductively back to ${\cal R}'$. We have to show that the permutation $\pi$ obtained by inserting $n$ in front of its friend $m$ is mapped by the direct algorithm back into $\cal R$.

Since $m$ was inserted in a regular way into $F$ to get $F'$, by the entry/exit lemma $m$ is in an exit position in $F'$, that is, when $n$ is placed in front of $m$ in $F'$, $n$ will not fit and will pull its friend $m$ out of its family to leave back $F$. In ${\cal R'}$, $(m)$ could not slide past $F$, so $m$ was larger than the minimum of $F$. So $m$ was not the anchor of $F'$, which places us in Case (E). 

In Case (E), $F'$ had $m$ pulled out by $n$. So  after pulling $m$ out of $F'$ we get back $F$. In Case (E)  $N=(n,m)$ was slid to its left wall (where it cannot slide further left), and $F$ replaced $F'$ to get a new registry, so sliding back $m$ to its right wall it finds $F$ and enters it regularly, by the entry/exit lemma on its former position in $F'$ . These operations are precisely the reciprocal for Case (E)$'$, 

From our original $\cal R$, in Case (E)$'$ we obtained  ${\cal R}'$ from $\cal R$ by removing $n$ from its small family $(n,m)$, by having the orphaned $m$, which could not slide left, slide right to its right wall, in front of a family $F$ into which $m$ was inserted regularly to get a family $F'$. Thus reciprocally, to get $\cal R$ from ${\cal R}'$, we remove $m$ from $F'$ to leave $F$, form the new family $(n,m)$, and slide it to the left wall where it cannot slide further left, and from where, when $(m)$ is sliding right, it will be stopped by $F$. These are precisely the transformations in Case (E), so those lead back to our original registry $\cal R$.

Case (D)$'$.  Here $(m)$ as a singlet family placed in the position of $N$ in ${\cal R}$ is in a left slide position. Make $(m)$ slide to its left wall. Let $F$ be the family immediately to its right. Insert $m$ into $F$ in a regular way to obtain a family $F'$. As $(m)$ has slid past $F$, the anchor of $F$ is larger than $m$. Thus when introduced in $F$, $m$ will become the new minimum, or anchor, of $F'$. Move now $F'$ on the position occupied by $N$ in ${\cal R}$ to obtain ${\cal R'}$.

The reverse of these operations is the following, giving back $\cal R$ out of ${\cal R'}$. Take the family $F'$ and slide it to its left wall. Remove $m$ from it to remain with $F$. Slide $m$ to the former position of $F'$ and insert $n$ to make the family $(n,m)$ there. This gives back ${\cal R}$.

As $m$ was the anchor of $F'$ and $n$ did not fit before its friend $m$ into $F'$ (since $m$ had entered regularly into $F$ to get $F'$) we are in Case (D). The operation in that case for getting $\cal R$ out of ${\cal R}'$ are precisely the ones we described earlier as reversing Case (D)$'$ to get the original $\cal R$ out of ${\cal R}'$. So Case (D) gives back our original  $\cal R$ and we are finished.

This ends the proof of the bijectivity of the algorithm.

The algorithm put arbitrary permutations of $\overline n$ in bijection with registries in which singlets were blocks. By the stripping lemma, the latter are in bijection with arbitrary stripped registries, and the main theorem is proved.

The pattern of a nonsinglet permutation block is a permutation of $\overline p$ which ends in 1. By the inverse cycle transform, this corresponds to a nonsinglet single cycle permutation. The algorithm, as we described it, maps single blocks into single block registries, i.e. registries having 1 in the last family. The number of such registries, for a given number total number $k$ of reds, or ascent values and $l$ of blues, or descent values, is the shifted multinomial $N^{(-1)}_{k,l}$, since all families except for the one containing 1 can be positioned arbitrarily. The shifted multinomials $N^{(-1)}_{k,l}$ are the Eulerian numbers $E_{k,l}$.

In general the algorithm as given works on block patterns. That is, in order to map an arbitrary permutation $\pi$, decompose $\pi$ into blocks $\pi^{(l)}$. Bring each block, of size $n_l$,  to its pattern, which is a permutation of $\overline{n_l}$ which ends in 1. Lift back the resulting registry ${\cal R}^{(l)}$ to the original values taken by $\pi^{(l)}$, and then concatenate the results in order.

In what follows we write permutations in line notation, though written enclosed in parentheses.
We define inductively a primitive permutation 

{\bf Definition} A permutation $\pi$ of $\overline n$ is called {\bf primitive} and has active family ${\cal A}(\pi)$ if 

$n=2$ and $\pi = {\cal A}(\pi) = (2,1)$

or $n>2$ and denoting by $\pi'$ obtained from $\pi$ by removing $n$,

the permutation $\pi'$ is primitive, and 

$n$ is followed in $\pi$ by $m\in {\cal A}(\pi')$ its friend, and either

{\bf Case (A)} ${\cal A}(\pi')$ with $n$ inserted (in the order from $\pi$) is a family, which is then ${\cal A}(\pi)$, or

{\bf Case (B)} ${\cal A}(\pi')$ with $n$ inserted is not a family, and then ${\cal A}(\pi)$ is ${\cal A}(\pi')$ with $m$ removed.

In the latter case, $(n,m)$ is a small family called a {\bf bud} of $\pi$. The buds of $\pi$ are the buds obtained at all the steps, when building $\pi$ out of $(2,1)$.
The buds are mutually disjoint, and their complement in $\pi$ is the active family ${\cal A}(\pi)$. 

The {\bf pairing pattern} ${\cal P}(\pi)$ of the primitive permutation $\pi$ is a sequence of integers of length $n$, defined inductively as follows.

If $\pi = (2,1)$ then ${\cal P}(\pi)=(1,2)$, and if $n>2$, let $b$ denote the number of buds, and $f$ the length of ${\cal A}(\pi)$, so $2b + f=n$

in Case (A) the pattern ${\cal P}(\pi)$ is obtained by appending $b+f$ to ${\cal P}(\pi')$.
in Case (B) the pattern ${\cal P}(\pi)$ is obtained by appending to ${\cal P}(\pi')$ the negative number $-{\cal P}(\pi')(m)$  (where ${\cal P}(\pi')(m)$ is the $m$-th entry of ${\cal P}(\pi')$.)

In Case (A), the last entry in ${\cal P}(\pi)$ is positive, not paired, on the $n$-th place with $n$ part of the active family ${\cal A}(\pi)$.
In Case (B), if $-p$ is the last entry in ${\cal P}(\pi)$, the pair $(p,-p)$ in ${\cal P}(\pi)$ is on positions $(m,n)$ and thus we have.

Define a pairing sequence ${\cal P}=(p_1,p_2,\dots,p_n)$ as follows.

The first 2 entries of ${\cal P}$ are $(1,2)$.

The positive numbers appear in natural increasing order $\{1,2,\dots \}$ in ${\cal P}$, 

Each negative number is paired with a previous positive number, with different negative numbers having different pairs.

The sum of the signs of the entries on positions $3,\dots,k$ of ${\cal P}$ is nonnegative, for all $k$.

Let $sgn=sgn({\cal P})$ denote the sequence of signs of ${\cal P}$. For each descent $sgn(i)=-1$ there are $\sum_{j<i} sgn(j)$ choices for ${\cal P}(i)$. 

We may think thus of $(sgn(3), sgn(4),\dots)$ as a Dyck path having nonnegative partial sums, with a decoration of any descent by a number between 1 and 2 + the path height before the descent. 

Remark now that the pairing sequence ${\cal P}(\pi)$ of a primitive permutation $\pi$ satisfies the above definition of a pairing sequence.

The sum of the signs of the numbers is the size $f$ of the active family ${\cal A}(\pi)$, and we always have $f \ge 2$, as $f=k+l$ with $k$ and $l$ the number of reds and blues (ascent values and descent values) of ${\cal A}(\pi)$, and $k,l\ge 1$

The positions of the unpaired positive numbers in ${\cal P}(\pi)$ are the numbers in the active family ${\cal A}(\pi)$.

The number of negative numbers in ${\cal P}(\pi)$ is the number of buds $b$. 

The positions of each pair $(i,-i)$ in ${\cal P}(\pi)$ form a bud of $\pi$.

Our main result is the following.

\begin{thm}
Pairs $({\cal P},(k,l))$ between a pairing sequence ${\cal P}$ and numbers $k,l\ge 1$ with $k+l$ equal to the sum of signs of ${\cal P}$ are in natural bijection with primitive permutations.

The bijection matches $({\cal P},(k,l))$ with a primitive permutation $\pi$ having pattern ${\cal P}(\pi)={\cal P}$ and with an active family ${\cal A}(\pi)$ having $k$ reds and $l$ blues.
\end{thm}

It is implemented by an algorithm which constructs the primitive permutations directly, without going through all permutations.

The singlet families in the registry thus carry no information other than the subset which consists of their union.  The algorithm establishes a bijection between permutations and registries consisting of families with $\ge 2$ elements, i.e., mutually disjoint subsets with $\ge 2$ elements, each with a specified, nonzero number of ascents and descents.  Singlet families are inserted into such a registry in a canonical way, with each singlet starting with the smallest one inserted immediately to the right of the set of families in the registry with minima smaller than it.  

Remark that the algorithm described above maps each block of the permutation into a block of the registry, independently of all the rest.  Thus the algorithm applies to permutations of the integers $\mathbb{Z}$ which have finite support, or by the inverse cycle transform, to permutations with a finite number of cycles.  More generally, the algorithm applies to permutations of $\mathbb{Z}$ with finite cycles, by transforming each cycle individually into a block of the registry.

{\bf An example.} In the sequel, we mark ascents in red and descents in blue for clarity. We start with the permutation 
$\rr{6}\textcolor{black}{2}\rr{5}\bb{1}\bb{4}\rr{7}\bb{3}$, 
with \rr{567} over diagonal in red, 2 on the diagonal in black and \bb{134} under the diagonal in blue.
Its cycle transform is 
$\pi = \rr{6}\rr{7}\bb{3}\rr{5}\bb{4}\bb{1}\textcolor{black}{2}$, with \rr{567} now ascent values in red, 2 a singlet block in black and  \bb{134} descent values in blue. We work with the cutoff $\pi^{(k)}$ of $\pi$ to values $\le k$, for $k=1,2,\dots$, and think, following Euler, of $\pi^{(k)}$ as obtained from $\pi^{(k-1)}$ by inserting the largest element $k$. We construct the families and the registries (lists of families) ${\cal R}^{(k)}$ corresponding to each $\pi^{(k)}$. 

Start with the singlet $\pi^{(1)}=1$ so ${\cal R}^{(1)} = ((1))$. The next element is 2, with permutation $\pi^{(2)}=12$. As 2 is last, it is inserted in the registry as the singlet (2) at the end, so ${\cal R}^{({2})}$ $ = ((1),(2)).$ The next cutoff is $\pi^{(3)}=\rr{3}\bb{1}2$. The friend of the largest element 3 is its right neighbor 1, with family $(1)$. 3 joins it and gives the family $(\rr{3}\bb{1})$, remaining in the place of $(1)$, in the first place in the registry, which now becomes ${\cal R}^{({3})}= ((\rr{3}\bb{1}),(2)).$ Then $\pi^{(4)}=\rr{3}\rr{4}\bb{1}2$. The friend of the largest element 4 is its right neighbor 1, with family $(\rr{3}\bb{1})$. With 4 inserted, that family becomes $(\rr{3}\rr{4}\bb{1})$, which is a family.  So ${\cal R}^{({4})}= ((\rr{3}\rr{4}\bb{1}),(2)).$

\begin{center}
\includegraphics[width=6.1 in]{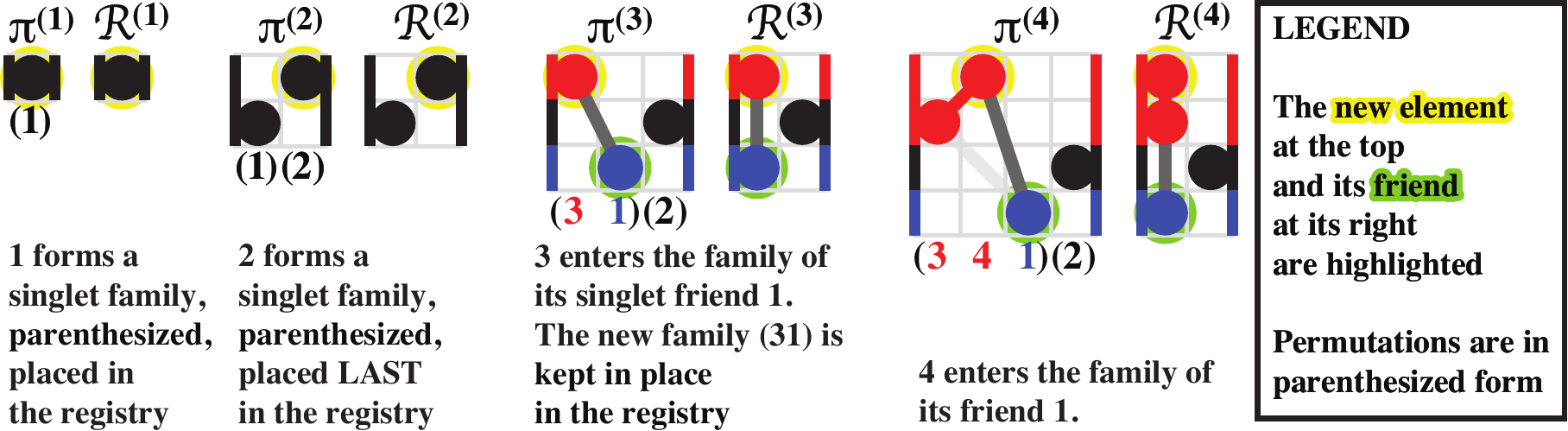}
\end{center}

The next step is $\pi^{(5)}=\rr{3}\rr{5}\bb{4}\bb{1}2$. The largest element 5 is followed by the friend 4. When inserted into the family $(\rr{3}\rr{4}\bb{1})$ of its friend, we obtain $(\rr{3}\rr{5}\bb{4}\bb{1})$.  This is not a family, as the descent $\bb{4}$ is bigger than the ascent \rr{3}. So 5 pulls its friend 4 out and they start a new family $\rr{5}\bb{4}$.  As the element $4$, taken away, is not the anchor of or minimum of its family, that family remains in place and the new family $(\rr{5},\bb{4})$ is sliding to its left wall, i.e., until it can move no further under elements with an anchor bigger than 4.  In this case, the left wall is the left end of the registry, and $(\rr{5},\bb{4})$ ends up in the first place in the registry.  Note that the idea behind this choice of position is the following.  In principle, in the reverse algorithm, 4 will look for its old family, which could be to its left or to its right.  The fact that $(\rr{5},\bb{4})$ cannot slide further left is a giveaway for the fact that it was not an anchor of its family, and that it must slide back right, stopped by its old family, which it will then rejoin.

Thus, the new family $(\rr{5},\bb{4})$ is placed first in the registry, to give ${\cal R}^{({5})}= ((\rr{5}\bb{4}),(\rr{3}\bb{1}),(2)).$  In the next step, $\pi^{(6)}=\rr{6}\bb{3}\rr{5}\bb{4}\bb{1}2$, 6 enters the family of its friend 3, with ${\cal R}^{({6})}= ((\rr{5}\bb{4}),(\rr{6}\bb{3}\bb{1}),(2)).$ In the last step $\pi = \pi^{(7)}=\rr{6}\rr{7}\bb{3}\rr{5}\bb{4}\bb{1}2$, 7 enters the family of its friend 3, with ${\cal R}={\cal R}^{({7})}= ((\rr{5}\bb{4}),(\rr{6}\rr{7}\bb{3}\bb{1}),(2)).$

\begin{center}
\includegraphics[width=6.1 in]{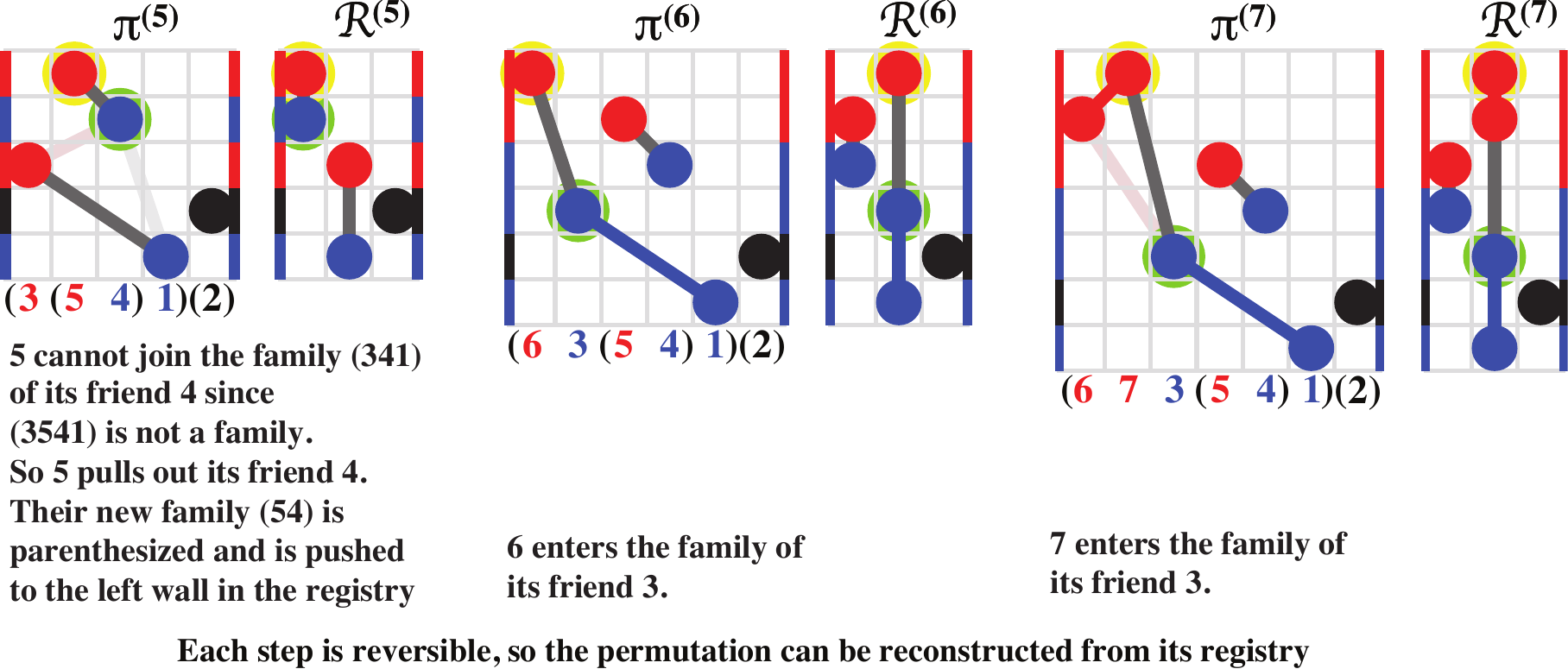}
\end{center}

For the reverse algorithm, we construct out of each registry a smaller registry, by removing its largest element, and which we map to a permutation by induction. Then we insert the largest element into that permutation.

We check the steps as follows. The registry ${\cal R}^{(1)} = ((1))$ gives the permutation $\pi^{(1)}=1$. ${\cal R}^{({2})}  = ((1),(2))$ gives the previous registry ${\cal R}^{(1)} = ((1))$, and then $\pi^{(2)}=12$ by inserting the singlet 2 at the end of $\pi^{(1)}$. 
The registry  ${\cal R}^{({3})}= ((\rr{3}\bb{1}),(2))$ leaves, after removing the largest element 3, its orphaned friend 1 which succeeded it in the family (\rr{3}\bb{1}). The orphan is left in place, so ${\cal R}^{({2})}  = ((1),(2))$. By our previous step, inductively, ${\cal R}^{({2})}$ corresponds to $\pi^{(2)}=12$.  The largest 3 is placed in the permutation $\pi^{(2)}=12$ immediately preceding its friend 1, to obtain $\pi^{(3)}=\rr{3}\bb{1}2$.  Now ${\cal R}^{({4})}= ((\rr{3}\rr{4}\bb{1}),(2))$ has largest element 4 followed by its friend 1. With 4 removed from its (large) family (\rr{3}\rr{4}\bb{1}) we get ${\cal R}^{({3})}= ((\rr{3}\bb{1}),(2))$ which by induction, as we saw above, gives the permutation $\pi^{(3)}=\rr{3}\bb{1}2$. In $\pi^{(3)}$, 4 is inserted before its friend 1 to get $\pi^{(4)}=\rr{3}\rr{4}\bb{1}2$. Take now ${\cal R}^{({5})}= ((\rr{5}\bb{4}),(\rr{3}\bb{1}),(2)).$ Its largest element 5 has friend 4. Removing 5 leaves 4 orphaned.  4 cannot slide further left, so it will slide right to find its family as (\rr{3}\bb{1}), which it enters, to give (\rr{4}\bb{3}\bb{1}), with 3 changing color from red to blue.  The registry is now ${\cal R}^{({4})}= ((\rr{3}\rr{4}\bb{1}),(2))$. By induction we showed that it corresponds to $\pi^{(4)}=\rr{3}\rr{4}\bb{1}2$, and we insert into it 5 before its friend 4, to get $\pi^{(5)}=\rr{3}\rr{5}\bb{4}\bb{1}2$. In ${\cal R}^{({6})}= ((\rr{5}\bb{4}),(\rr{6}\bb{3}\bb{1}),(2))$ the largest 6 has friend 3.  With 6 removed we get ${\cal R}^{({5})}= ((\rr{5}\bb{4}),(\rr{3}\bb{1}),(2))$ which we showed corresponded to $\pi^{(5)}=\rr{3}\rr{5}\bb{4}\bb{1}2$. Inserting 6 before its friend 3 in $\pi^{(5)}$ gives $\pi^{(6)}=\rr{6}\bb{3}\rr{5}\bb{4}\bb{1}2$. Finally ${\cal R}={\cal R}^{({7})}= ((\rr{5}\bb{4}),(\rr{6}\rr{7}\bb{3}\bb{1}),(2))$ has its largest element 7 in a large family (\rr{6}\rr{7}\bb{3}\bb{1}) having as friend its successor 3. With 7 removed, the registry is ${\cal R}^{({6})}= ((\rr{5}\bb{4}),(\rr{6}\bb{3}\bb{1}),(2))$ corresponding by induction, as we saw, to $\pi^{(6)}=\rr{6}\bb{3}\rr{5}\bb{4}\bb{1}2$.  By inserting 7 in front of its friend 3, we obtain the final result of the inverse algorithm, the permutation $\pi = \pi^{(7)}=\rr{6}\rr{7}\bb{3}\rr{5}\bb{4}\bb{1}2$, which corresponds to the registry ${\cal R}=((\rr{5}\bb{4}),(\rr{6}\rr{7}\bb{3}\bb{1}),(2))$.

The inverse algorithm starts in fact backwards, with the registry ${\cal R} = {\cal R}^{({7})}$ and proceeds to define the registries ${\cal R}^{({6})},\dots,{\cal R}^{({1})}$ inductively as above, and then constructs $\pi^{(1)},\dots,\pi^{(7)}=\pi$ as described.

At each step $k$ the ascent/descent colors in $\pi^{(k)}$ match the colors within the families in ${\cal R}^{({k})}$.

The inverse algorithm thus reverses the map from $\pi$ to $\mathcal{R}$ from the direct algorithm.

A family registry has singlets, which are elements of the permutation which form a block of 1 between successive minima, as singlet blocks of the registry, again as blocks of 1 between successive minima of the registry, as such their position is entirely determined by the nontrivial families in the registry.  If we want to have a bijection between permutations and unrestricted registries, we have to strip these singlets away from these unrestricted registries.  The reverse algorithm requires that singlets be present in the correct places, but as their positions are determined by the other families, the missing singlets can be inserted in place at the beginning of the algorithm.  A stripping and reinsertion algorithm for singlets is part of the implementation of the main theorem.  We formalize it in this paper as the stripping lemma.

{\bf Essentially the algorithm decomposes a permutation $\pi$ of $\overline n$ into an arbitrary ordered list $((F_1,a_1),\dots,(F_k,a_k))$ where $F_1,\dots ,F_k\subset \overline n$ are mutually disjoint and $1\le a_i<|F_i|$ for all $i$.}
The complement of $\bigcup_i F_i$ consists of singlets. The largest $a_i$ elements in each $F_i$ are its ascents, the rest its descents. The algorithm maps these ascents and descents of families into the ascent-values $\pi_i>\pi_{i-1}$ respectively descent-values $\pi_i<\pi_{i-1}$ of $\pi$, with singlets becoming the singlet blocks.  In the inverse cycle transform $\sigma$ of $\pi$ these become elements $\sigma_i $ with $\sigma_i \gtrless i$ and $\sigma_i = i$ respectively. The registries can be easily counted, resulting in the multinomial type formulae given at the beginning.

\section{The parenthesized form and the tree form of a permutation} The algorithm presented in the previous section reveals information on the internal structure 
of a permutation, decomposing it into a tree. Every node of the tree is called a primitive permutation. 

A permutation decomposed into families, as seen in the illustrations, has a two dimensional structure, with families positioned vertically, each above an opening in a family below, in a tree structure. This structure can be captured in a very natural manner, by grouping the permutation into nested parentheses.   

We construct the parenthesized form of a permutation $\pi$ as follows.  The outermost parentheses enclose the blocks of the permutation, which by inverse cycle transform correspond to cycles.  Then, every time a new family is started in the algorithm, that family starts a new tree branch, and we enclose that new family in parentheses.   

The result of this algorithm is the parenthesized form of the permutation $\pi$. 

Remark that, while the main bijection depends on a choice of algorithm for ordering new families in a linear order in the registry, the parenthesized form of the permutation, is, by contrast, canonical, as are the tree and the primitive permutation structures which are derived from it.

In our previous example, the parenthesized form of $\pi=\rr{6}\rr{7}\bb{3}\rr{5}\bb{4}\bb{1}2$ is constructed stepwise as follows.  We start with the singlet 
$\pi^{(1)}=(1)$, then $\pi^{(2)}=(1)(2)$, then $\pi^{(3)}=(\rr{3}\bb{1})(2)$, $\pi^{(4)}=(\rr{3}\rr{4}\bb{1})(2)$. As 5 starts a new family with its friend 4, we have $\pi^{(5)}=(\rr{3}(\rr{5}\bb{4})\bb{1})(2)$.  Afterwards $\pi^{(6)}=(\rr{6}\bb{3}(\rr{5}\bb{4})\bb{1})(2)$ and $\pi = \pi^{(7)}=(\rr{6}\rr{7}\bb{3}(\rr{5}\bb{4})\bb{1})(2)$.

The families in the registry ${\cal R}= ((\rr{5}\bb{4}),(\rr{6}\rr{7}\bb{3}\bb{1}),(2))$ are now visible in the body of the permutation $\pi = (\rr{6}\rr{7}\bb{3}(\rr{5}\bb{4})\bb{1})(2)$, as the content of each parenthesis from which the inner parentheses and their content have been removed.

The outermost parentheses are the blocks of the permutation $\pi$, corresponding to the cycles of the inverse cycle transform of $\pi$.

The figure below describes the graph of a permutation on the left and its cycle transform in the center.  Both are decomposed into families with elements over diagonal, respectively ascents, in red and elements under diagonal, respectively descents, in blue. Families are ordered in the registry on the right.  The family tree associated to the permutation is placed underneath. The permutation in the center is written in parenthesized form, reflecting the tree structure underneath.  The tree structure is also shown on the graph of that permutation, in the background. Roots and buds of each node are marked in green with branches marked in brown.  In the graph, the active family of each node is connected by a light purple line to that node.

\begin{center}
\includegraphics[width=6.2 in]{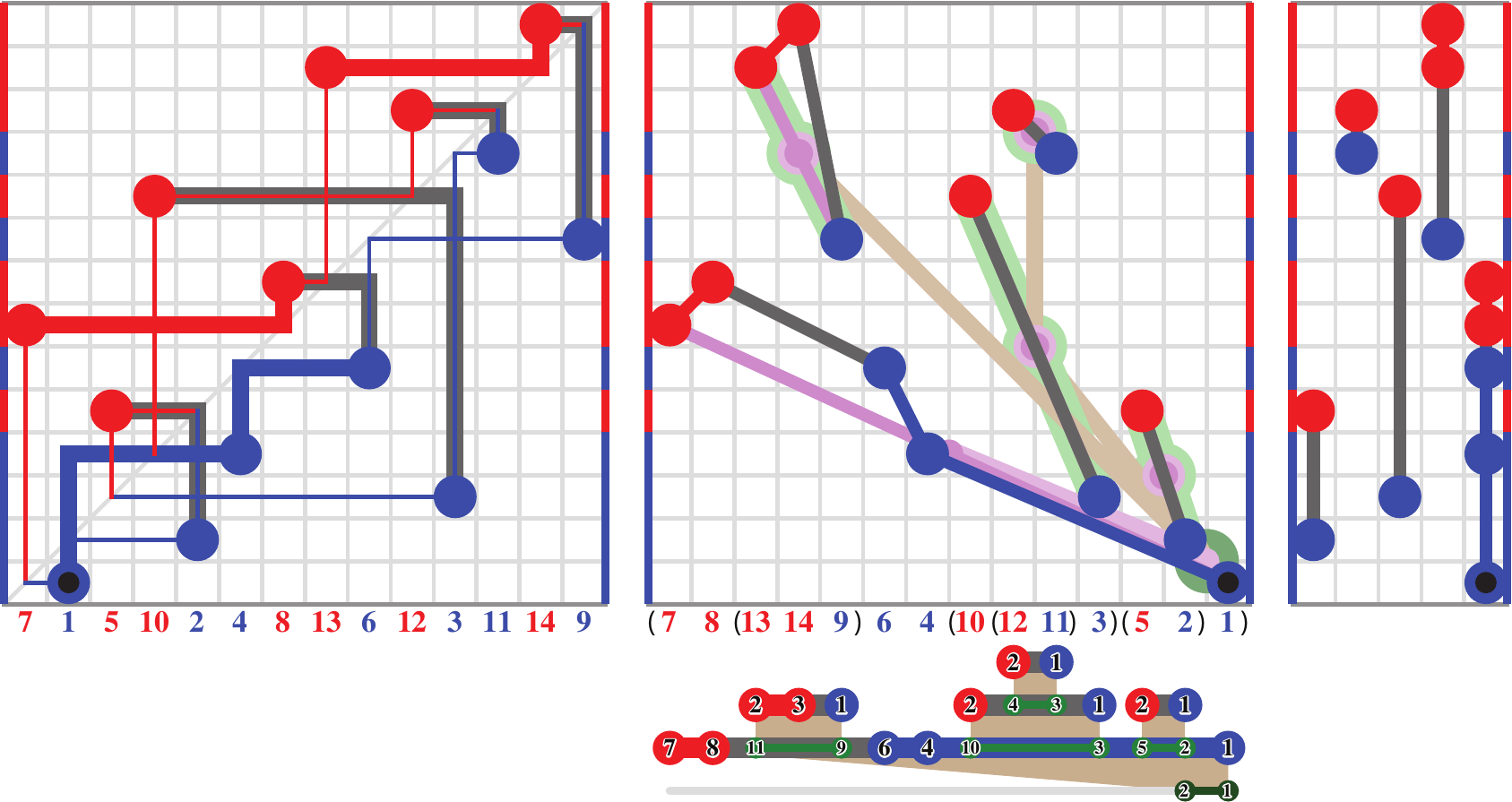}
\end{center}

\section{The multiparenthesized form of a permutation}  Given a permutation $\pi$, the main algorithm constructs a registry of families, a list of nonsinglet families, arranged in linear order, which encode the permutation $\pi$.  We can consider now the permutation $\pi^{(1)}$ which encodes the order of the families in the registry, relative to the standard order in which the minima of the families are increasing.  In turn, $\pi^{(1)}$ can be subjected to the same algorithm, will produce a registry, and a permutation $\pi^{(2)}$, encoding the order of the families of $\pi^{(1)}$ in the registry.  We can continue this way until we reach a family which consists only of singlets.

It is clear that the previous sequence of registries encodes the original permutation $\pi$.  We do not need, however, the order of families in any of the registries, since that is encoded inductively in the next permutations.  The only information we need to keep track of is the number of reds and blues in each family, in all iterations.  We obtain this way the following data, which encodes, bijectively, any permutation $\pi$ of the set $\overline{n}=\{1,\ldots,n\}$.

Partition the set $\overline{n}$ into nonempty subsets.  For each nonsinglet family, specify the number of reds and blues (e.g., by coloring a number of its elements starting from the top as red, with the rest blue).

Partition the set of nonsinglet sets from the previous step.  Specify for each nonsinglet at this step a number of reds and blues (e.g. by coloring some of its members red, starting with the biggest ones, as arranged in increasing order of the smallest number in each member).  We mark this second partition, by enclosing the previous partition in parentheses, and we mark the colors by coloring these parentheses.

At each subsequent step, partition the nonsinglets of the previous partition, enclose each nonsinglet part in a parenthesis, and color the immediate inner parenthesis within that red or blue, as before.  Repeat this until all elements of the last partition are singlets.  This will encode the original permutation $\pi$.

The idea is that if we put the elements of a set into bags with $\ge 2$ elements each, leaving out some singlets, and we mark on each bag a number of reds, then we put bags into bigger bags, each with $\ge 2$ bags inside, leaving aside some singlets, and mark on each bigger bag a number of reds, and we continue this way until there are only singlets left, the data in the successive partitions, together with the number of reds on each bag, encode precisely a permutation of the original set.

\textbf{Example.} For the identity permutation, the set consists of singlets and no bags are needed.  If the permutation is the reverse, $(n,n-1,\ldots,1)$, then we need one bag, with one red and the others blue.

In the picture of the large permutation attached in the appendix, the idea of the multiparenthesized form is to start with the permutation on the upper middle board, and continue with the permutation on the lower right board, describing the order of the families, and then proceed from the latter inductively, until the lower right board becomes trivial, i.e., there are $\le 1$ nonsinglet families.  Thus, the main idea of the multiparenthesized form of a permutation is that by marking the number of reds and blues of families in successive registries, we no longer need to specify the relative order in any of the registries.  In other words, with our algorithm, repeated bicolored partitions encode the relative order of families in registries, and thus ultimately the original permutation itself.

\begin{center}
\includegraphics[width=1.0\linewidth]{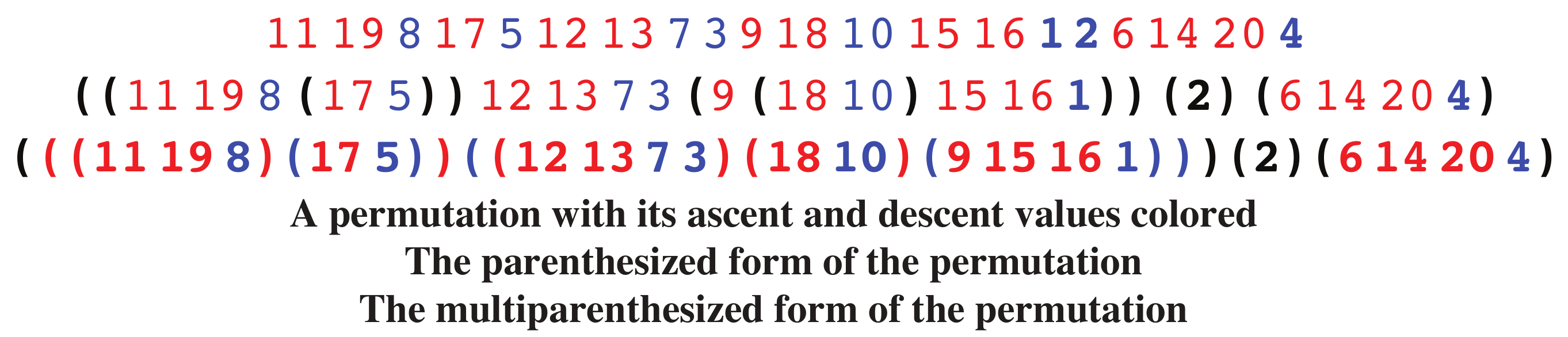}
\end{center}

We have implemented in Mathematica both directions of the algorithm, as well as supplementary aspects of the structure, fully interactively. One can input permutations or registries by clicking and dragging any point in the display.

{\bf The tree form of a permutation.}
The parentheses give a permutation a quite complex tree structure, visible in its graph. We now construct from each parenthesis a node of a tree. 
We shall retain at a node enough information for reconstituting the tree, and thus the whole permutation, and call its pattern a {\bf primitive permutation}.

A primitive permutation at a node has a {\bf branch root} in a previous node, it contains a family, as defined before, called the {\bf active family} of the node, as well as pairs of elements forming small families called ${\bf buds}$, which are the branch roots of its offshoot branches. The idea is that while a permutation decomposes into families as seen in the algorithm, the branch root and buds indicate precisely where different families are inserted into previous families.

In the example in the picture above a node is (7 8 (13 9) 6 4 (10 3) (5 2) 1), where only the smallest two of each inner parenthesis were retained as a bud which started that inner parenthesis. Its pattern is itself, so this is a primitive permutation. Its active family is 7 8 6 4 1 and its root is 2 1. Note that in the tree underneath the permutation, in the picture above, the part of the permutation at each node has been reduced to its pattern, which is a primitive permutation. For instance the node (10 (12 11) 3) has pattern (2 (4 3) 1).

The tree can be reconstituted from these primitive permutations by specifying for each branch root the bud on the previous root with which it will be identified (unless it is a root of a whole tree.) In addition one has to specify the shuffling of various branches, the way in which the numbers on different branches are made different while preserving the order on each branch.

Note that the parenthesized form of a permutation is natural and canonical. There are various algorithms for ordering the families in the registry, which are based on the history of the families, as in the previous example. The "visual" orders, based on the position of the smallest or largest ascent or descent of each family in $\pi$ do not lead to bijective algorithms.

The structure described, as well as the definition of primitive permutations below, are motivated by the following property. {\bf The pattern structures of different nodes} (an active family, buds and a root for each) {\bf are independent of each other.} In other words, once a new family is started, it has the potential to develop like any self standing permutation. The only relative interaction with the rest is that, $\pi$ being a permutation, each entry can be used only once, i.e. different branched need to be shuffled. This shuffling, which is a pattern on the union of the branches which reduces to the pattern of each node, is arbitrary.

\section{Primitive permutations}

The idea of the primitive permutation is the following.  In the evolution of a permutation which we described before, when a new element does not fit into a family, it picks its friend immediately to its right and they start a new family.  Keeping track of all these new families is complicated, and as it turns out, once their new family is started they do not interfere anymore with their old family or with their siblings.  Therefore, just like in a large modern family, we shall keep on the wall a wedding picture of each member who left, but we shall no longer follow in detail what happened to them.  Such a permutation, in which a departed pair of immediate neighbors, called a bud (since they started a new branch of the family tree) is no longer developed, will be called a primitive permutation.  It consists of the active family members still in the household, together with the positions of the buds.  It forms a node of the family tree, with the buds indicating the precise insertion points of the branches above.

The grand assembly of all these nodes into the family tree -- in fact a family forest, as blocks of the permutation develop into separate trees -- is axiomatized and described by a decimal code, similar to the one which organizes a whole library.

The delicate algorithm which describes the construction of a primitive permutation, given only the pairing between the values in each bud and the number of colors in the active family, is described in what follows.  The assembly algorithm for the decimal code of a permutation is described in the chapter which follows.

The data from which we construct a primitive permutation is of the following type.  Construct a primitive permutation in which the fifth highest element was taken away by the eighth highest, and from the remaining six members the highest four are ascent images and the lowest are two descent images.  We shall show that there is precisely one primitive family satisfying these conditions, and we shall construct algorithmically its whole history.

In the decimal notation which encodes the whole permutation, this family will have a decimal head like $3.1.7$, and then each member will be assigned two digits.  In our case, those will be $1.\bfb,\ 2.\bfb.,\ 3.\bfr,\ 4.\bfr,\ 5.1$, as the fifth highest was taken away, $6.\bfr,\ 7.\bfr,\ 5.2$, which is on the position of the highest element which took 5 away.  When $5.1$ and $5.2$ start their new branch of the family tree, the family will append to the label $3.1.7$ of the family which they left the suffix 5, which will give the label $3.1.7.5$ for their new branch.  In it, they will be the elements $3.1.7.5.1.x$, and $3.1.7.5.2.y$.  For instance, $3.1.7.5.1.\bfb$ will indicate that the founder of the new family stayed in, as a blue descent value, while $3.1.7.5.2.1$ will show that the cofounder of the new family moved on, to start a new family of its own.

\begin{center}
\includegraphics[width=1\linewidth]{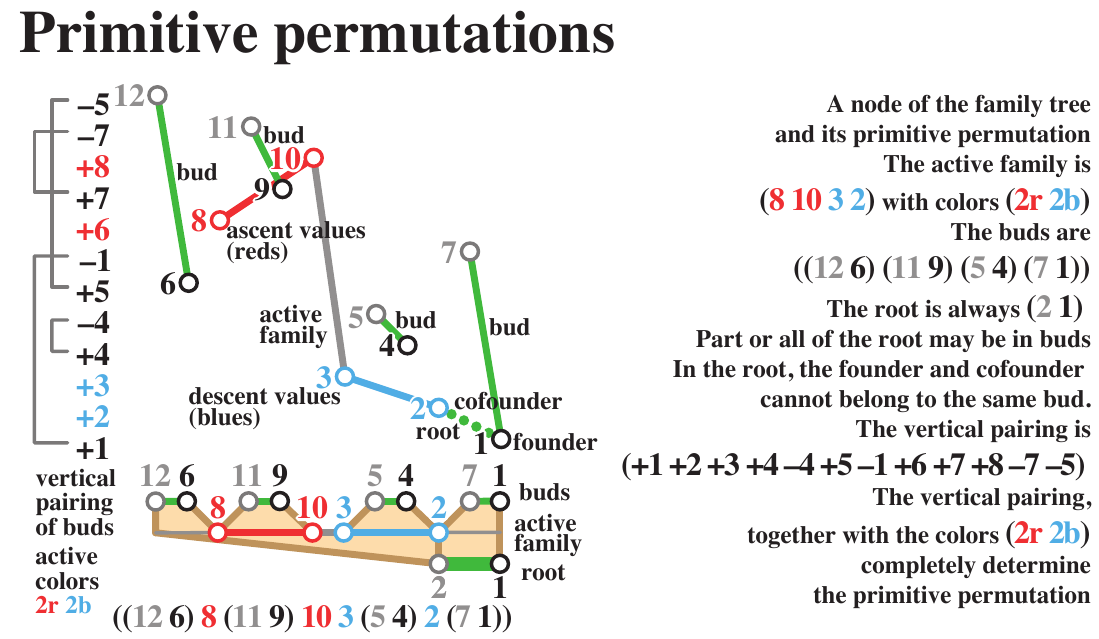}
\end{center}

{\bf Primitive permutations.} Recall that we named a permutation $\pi$ a family if it is a singlet or has the pattern $d+1,d+2,\dots,d+a,d,d-1,\dots,1$ with $a,d \ge 1$ (e.g. 456321 for a=d=3.) Any nonempty subset of a family is a family. The formal definition of a primitive permutation is the following.

We call a permutation $\pi$ of $\overline n$ {\bf primitive} if it has $\ge 2$ elements, there are mutually disjoint adjacent pairs $(i,i+1)_{i\in I}$ in $\overline n$, called {\bf buds} with $\pi(i)>\pi(i+1)$ such that (a) on the complement of the buds $\pi$ is a (nonempty) family (b) with one or more of the buds left in it, $\pi$ is not a family and (c) the properties (a) and (b) hold for any cutoff of $\pi$ by $c=1,\dots,n$, defined as the pattern of the restriction of $\pi$ to points $i$ with $\pi(i)\le c$. The complement of the buds will be called the {\bf active family} of the node. The smallest and the second smallest elements of the primitive permutation (which includes the buds) are called its {\bf root.} These are the two elements which started the node, and will also be called the {\bf founder} and {\bf cofounder}.

Primitive permutations are easy to construct inductively. Given a primitive permutation of $\overline n$, insert $n+1$ immediately before any element $m$ of the active family. The result is always a primitive permutation. Either the active family together with $n+1$ forms a family or else, if they don't, $n+1$ and $m$ form an extra bud. All primitive permutations are constructed this way from the singlet (1).

The {\bf graph generating the primitive permutations} is shown below.
Any primitive permutation is uniquely built by a walk on this graph.

Each downward edge produces family growth, by an ascent (red) or descent (blue).
Each upward edge produces a new (parenthesized) bud and decreases the family 
by an ascent (red) or descent (blue), in a number of ways marked by the arrow number.
The pattern of the active family at a node, with the number of ascents in red and descents in blue, 
is shown in examples.  The primitive permutation may contain previous buds. In the examples,
as the new element 6 enters before one of the 5 positions in the active family, the family may grow or bud.

\section{The algorithmic construction of primitive permutations}

\textbf{Definition.} A family pairing data is a pair (inds, n$_\text{Reds}$).  inds is a sequence with elements the numbers $1,\ldots,n$ and a subset $\{-1,\ldots,-n\}$.  The positive numbers appear in the sequence in their natural order, and each negative number $-k$ appears after the positive $k$.  Let $(s_1,s_2,\ldots,)$ denote the signs of elements in the sequence, with $s_i\in \{\pm 1\}$.  For any $k\ge 2$, we require that the partial sum $s_1+\cdots+s_k\ge 2$.  if $s=\sum s_i$, then $1\le n_\text{Reds}\le s-1$.

Remark that the sign condition is equivalent to the fact that $s_1=s_2=+1$ and $(s_3,s_4,\ldots)$ form the increments of a Dyck path.  Equivalently, the entire sequence of signs $(s_1,s_2,\ldots)$ forms the increments of a Dyck path, which reaches height 2 in 2 steps, and remains at or above height 2 afterwards.  The number of choices in the position set corresponding to a descent, $s_i=-1$, is precisely the height $s_1+s_2+\cdots +s_{i-1}$ of the Dyck path, before $s_i$.

Thus the family pairing data is equivalent to a Dyck path which reaches height 2 in 2 steps, and remains at or above height 2 afterwards, in which each descent is marked with a number $1,2,\ldots,h$, where $h$ is the height of the path before the descent, together with a number $n_\text{Reds}$ between 1 and the final height minus 1.  An example is the Dyck path $(s_1,\ldots,s_7)=(+1,+1,+1,-1,+1,+1,-1)$, decorated as $(+1,+1,+1,-1(2),+1,+1,-1(4))(n_\text{Reds}=2)$.  We call the preceding data a decorated Dyck path.  The paired values corresponding to this data is the sequence inds$=(1,2,3,-2,4,5,-5)$, where -2 is the second choice among the preceding $(1,2,3)$, and -5 is the fourth choice among the remaining $(1,3,4,5)$, as 2 was paired with -2.  The height of the Dyck path before the two negatives was 3, respectively 4.  The final height of the Dyck path is h=3; we have $n_\text{Reds}=2<h$, and we let $n_\text{Blues}=h-n_\text{Reds}$.  

The algorithm which follows will establish a bijection between paired values -- or equivalently decorated Dyck paths -- as above, having m descents, $n_\text{Reds}, n_\text{Blues}$ (and thus a total length  $n=n_\text{Reds}+n_\text{Blues}+2m$), and primitive permutations with active family $n_\text{Reds}$ reds $n_\text{Blues}$ blues, and $m$ buds.

With the algorithm, the data in the previous example will correspond to the primitive permutation $((4\ 2)\ 3\ 5\ (7\ 6)\ 1)$, with buds $(4\ 2)$ and $(7\ 6)$, and active family $(3\ 5\ 1)$, having ascent values (reds) $3,5$ and descent value (blue) $1$.

\subsection{The pairing sequence -- primitive permutation algorithm}
We start with paired values data as described above.  The only such data of length 2 is $(1,2)$ with $n_\text{Reds}=n_\text{Blues}=1$.  The corresponding primitive permutation is $\pi= (2\ 1)$ with no buds.  Assume now that paired values inds of length $n$ and $n_\text{Reds}$ are given.  We let inds$'$ denote the sequence obtained by removing the last entry of inds.  We compute below a modified $n_\text{Reds}'$ to be paired with inds$'$.  We let $\pi'$ denote the permutation constructed inductively by the algorithm from the data (inds$'$, $n_\text{Reds}'$), and we describe the construction of $\pi$ from $\pi'$.

Before we begin the algorithm, let us make the following observation.  If the last sign of inds is -1, then we compute the number $n_\text{choice}$ of its chosen sequence as follows.  In the unpaired positive values preceding -k, its pair +k appears on position $n_\text{choice}$.  Equivalently, $n_\text{choice}$ was the decoration of the last entry -1 of the Dyck path ($\cdots$, -1($n_\text{choice}$)).  In the example preceding the proof, $(+1,+1,+1,-1(2),+1,+1,-1(4))$, the last $n_\text{choice}=4$.

We shall construct $n_\text{Reds}'$ for the previous $\pi'$ by adding a correction to $n_\text{Reds}$.  In the second part of the algorithm, we shall show that the permutation $\pi$ which we construct has precisely $n_\text{Reds}$ active reds, and thus the correction was correct.

The previous $n_\text{Reds}'$ is computed by adding to $n_\text{Reds}$ the following correction.  $c=n_\text{Reds}'-n_\text{Reds}$.  

\textbf{Case (A)} If the last sign in inds is $+1$ and $n_\text{Reds}=1$, then $c=0$.  

\textbf{Case (B)} If the last sign in inds is $+1$ and $n_\text{Reds}>1$, then $c=-1$.

\textbf{Case (C)} If the last sign in inds is $-1$, and if the last choice $n_\text{choice}$, as described above, is $\le n_\text{Blues}$, the correction $c=0$.

\textbf{Case (D)} If the last sign in inds is $-1$, and $n_\text{choice}>n_\text{Blues}$, then $c=+1$.

This completes the first part of the algorithm.  We now have the data inds$'$, obtained by removing the last entry of inds, and $n_\text{Reds}'$, obtained by the correction in the cases above.

With these we construct inductively a primitive permutation $\pi'$.  We shall describe below the way in which we obtain the desired permutation $\pi$ from $\pi'$, and check that in each case the correction between $n_\text{Reds}$ of $\pi$ an $n_\text{Reds}$ of $\pi'$ is the one described above.  

Recall that $n$ is the length of the sequence inds, and also the length of and the highest number in the permutation $\pi$ which we are constructing.

\textbf{Case (A)} If the last sign in inds is $+1$ and $n_\text{Reds}=1$, then we insert $n$ immediately before the first element in the active family of $\pi'$ (obtained from $\pi'$ by removing its parenthesized buds).  In $\pi$, $n$ will be red and the previous unique red in the active family of $\pi'$ will turn blue.  Thus, the active family of $\pi$ will have a unique red, as we assumed in the above condition $n_\text{Reds}=1$.  Thus, $n_\text{Reds}=n_\text{Reds}'=1$, and  $c=n_\text{Reds}'-n_\text{Reds}=0$, as assumed.

\textbf{Case (B)} If the last sign in inds is $+1$ and $n_\text{Reds}>1$, then we insert $n$ immediately before the highest blue in the active family of $\pi'$.  Then the active part of $\pi$ (obtained by removing its parenthesized buds) is a family, with $n$ as red.  Thus, $n_\text{Reds}=n_\text{Reds}'+1>1$, and  $c=n_\text{Reds}'-n_\text{Reds}=-1$, as assumed.

\textbf{Case (C)} If the last sign in inds is $-1$, with the choice number $n_\text{choice}$ as before, and if $n_\text{choice}\le n_\text{Blues}$, let $b$ denote the entry number $n_\text{choice}$ on the increasing list of blues in the active family of $\pi'$.  We replace $b$ by the bud $(n,b)$ to obtain the permutation $\pi$.  The number of active blues decreased by one, the number of active reds stayed the same, so $n_\text{Reds}=n_\text{Reds}'$, and $c=n_\text{Reds}'-n_\text{Reds}=0$, as assumed.

\textbf{Case (D)} If the last sign in inds is $-1$, and $n_\text{choice}>n_\text{Blues}$, then let $n_\text{choice}'=n_\text{choice}-n_\text{Blues}>0$, and let $r$ denote the entry number $n_\text{choice}'$ on the increasing list of reds in the active family of $\pi'$.  We replace $r$ by the bud $(n,r)$ to obtain the permutation $\pi$.  The number of active reds decreased by one, the number of active blues stayed the same, so $n_\text{Reds}=n_\text{Reds}'-1$, and $c=n_\text{Reds}'-n_\text{Reds}=+1$, as assumed.

This ends the algorithm.

Remark that when the highest element $n$ is inserted into $\pi'$, in precisely the position in which we inserted it (and we have covered all possible positions in front of elements in the active family of $\pi'$), then the result is the primitive permutation $\pi$, as we described it in the algorithm.  In the first two cases, the active family $n$ enters the active family (and the paired value is positive and the Dyck path is rising).  In the last two cases, $n$ does not fit into the active family, so it removes its friend $b$ respectively $r$, to start a new family as a bud.  There are $n_\text{choice}$ choices of friends, and $n_\text{choice}$ is recorded as a decoration on the descending Dyck path, while in inds, the last entry $-k$ is paired with the $n_\text{choice}$ choices among the remaining unpaired positive entries in inds preceding it.  These last observations show that the last entry of inds$'$, can be recovered from the pair $\pi'$ and $\pi$.  As $n_\text{Reds}$ is the number of active ascent values, or reds, in $\pi$, this shows that the algorithm is reversible, i.e., the paired values, or decorated Dyck path, together with $n_\text{Reds}$, can be recovered out of $\pi$.  Thus the algorithm is invertible.  

This ends the proof that primitive permutations are in bijective correspondence with paired values, or decorated Dyck paths, together with a specified number of reds and blues in the active family.  The algorithm also gives an fast and explicit construction for this bijection.  The smallest cases of data and primitive permutations are the following.

\vskip 0.1 in

\begin{tabular}{|c|c|c|c|c|c|}
 $n_\text{Reds}$ & $n_\text{Blues}$ &$n_\text{Buds}$  & Paired values & Dyck paths & Primitive permutations \\ 
\hline 		\rr{1} & \bb{1} & 0 & $(1,2) $ & $(+,+)$ & $(\rr{2}\ \bb{1})$ \\ 
\hline	  	\rr{2} & \bb{1} & 0 & (1,2,3) & (+,+,+) & $(\rr{2}\ \rr{3}\ \bb{1})$\\
\hline	 	\rr{1}  & \bb{2} & 0 & (1,2,3) & (+,+,+) & $(\rr{3}\ \bb{2}\ \bb{1})$\\
\hline	  \rr{3}  & \bb{1} & 0 & (1,2,3,4) & (+,+,+,+) & $(\rr{2}\ \rr{3}\ \rr{4}\ \bb{1})$\\
\hline	  \rr{2}  & \bb{2} & 0 & (1,2,3,4) & (+,+,+,+) & $(\rr{3}\ \rr{4}\ \bb{2}\ \bb{1})$\\
\hline	  \rr{1}  & \bb{3} & 0 & (1,2,3,4) & (+,+,+,+) & $(\rr{4}\ \bb{3}\ \bb{2}\ \bb{1})$\\
\hline \rr{1} & \bb{1} & 1 & (1,2,3,--1) & (+,+,+,--(1)) & $(\rr{3}\ \bb{2}\ (4\ 1))$ \\ 
 	   	 &  &  & (1,2,3,--2) &  (+,+,+,--(2)) & $((4\ 2)\ \rr{3}\ \bb{1})$ \\ 
		 &  &  & (1,2,3,--3) &  (+,+,+,--(3)) & $(\rr{2}\ (4\ 3)\ \bb{1})$\\ 
\hline
\end{tabular} 
\vskip 0.1 in
Remark that besides the paired values or Dyck paths, one needs to specify also the number of reds (ascent values) and blues (descent values) desired, in order to obtain the primitive permutation.  The structure of the primitive permutation changes considerably with the specification of the colors.

The number of primitive permutations of each type depends only on the sum $n_\text{Fam}=n_\text{Reds}+n_\text{Blues}$, rather than on $n_\text{Reds}$ and $n_\text{Reds}$ individually, as can be easily checked in the algorithm above.  These numbers, in small cases, are given in the table below.

\vskip 0.1 in
\begin{tabular}{c|c|c|c|c|c|c|c|}
\multicolumn{1}{c}{} & \multicolumn{7}{c}{$n_{Buds}$}\tabularnewline
\cline{2-8} 
\multirow{6}{*}{$n_{Fam}$} &  & 0 & 1 & 2 & 3 & 4 & 5\tabularnewline
\cline{2-8} 
 & 2 & 1 & 3 & 21 & 207 & 2529 & 36243\tabularnewline
\cline{2-8} 
 & 3 & 1 & 7 & 69 & 843 & 12081 & 197127\tabularnewline
\cline{2-8} 
 & 4 & 1 & 12 & 159 & 2388 & 40221 & 751032\tabularnewline
\cline{2-8} 
 & 5 & 1 & 18 & 309 & 5628 & 110781 & 2361222\tabularnewline
\cline{2-8} 
 & 6 & 1 & 25 & 540 & 11760 & 268365 & 6495345\tabularnewline
\cline{2-8} 
\end{tabular}
\vskip 0.1 in

The first line of the table is the sequence oeis A167872: $1,3,21,207,\ldots$, which is a sequence of moments connected with Feynman numbers.  Its double is the sequence A115974: $2,6,42,414,\ldots$, which counts the number of Feynman diagrams at perturbative order n which are proper, i.e., cannot be obtained from simpler diagrams by concatenation.  Note that our algorithm provides a direct construction.  For Feynman diagrams, the known constructions give all diagrams, counted by the sequence A000698: $2,10,74,706,\ldots$.  From these, the improper diagrams have to be removed case by case.

The other lines of the diagram, as well as the table itself, do not appear to have been studied before.

\subsection{The generating function for primary permutations}
\

Let 
\begin{eqnarray*}
\varphi(t) & = & \frac 1 2\left(\frac 1 t - \frac{1}{1-\frac{1}{\sum_{k=0}^{\infty}(2k-1)!!t^k}}\right)= \\
& = & 1+3t+21t^2+207t^3+\cdots,
\end{eqnarray*}
which is the generating function for the number of primitive permutations with small families ($n_\text{fam}=2$) and $0,1,2,\ldots$ buds, related to the count of fermionic Feynman diagrams.
This function counts our pairing data for constructing primitive permutations, as $(2k-1)!!$ counts the number of ways to make pairs out of $2k$ elemnts.

Let 
\begin{eqnarray*}
\psi(x) & = & \int_0^x e^{y^2}dy,
\end{eqnarray*}
which is up to a normalization the imaginary error function.  The generating series
\begin{eqnarray*}
\Phi(x,t) & = & \sum n_{\text{prim}}x^{n_\text{fam}}t^{n_\text{buds}}\\
& = & -x+\left(e^{\frac{1-(1-x)^2}{2t}}-1\right)(1-2 \varphi(t))+\\
& & + 2 e^{-\frac{(1-x)^2}{2t}}\left(\psi\left(\frac{1-x}{\sqrt{2t}}\right)-\psi\left(\frac{1}{\sqrt{2t}}\right)\right)\frac{1-t-2 t \varphi(t)}{\sqrt{2t}}.
\end{eqnarray*}
with the boundary condition 
\begin{eqnarray*}
\lim\limits_{x\rightarrow 0}x^{-2}\Phi(x,t) & = &\varphi(t) 
\end{eqnarray*}
which satisfies the equation

\begin{eqnarray*}
\frac{\partial \Phi(x,t)}{\partial x} = \frac{x-x^2-4 t}{t x} \Phi(x,t)+\frac{2 }{x^3}\varphi(t)-\frac{1}{t x^2}
\end{eqnarray*}

This function is obtained as a solution to the differential equation which encodes the recurrence relation for families which we described before. 

\subsection{The decimal code of a permutation}

Let us start with a permutation $\pi$ written in line notation. We associate to it a sequence resembling the Dewey code for the books in a library.

The data for a primitive permutation $\pi$ consists of a pairing sequence, e.g $$(+1, +2, +3, -2, +4, -1, +5, +6, -4)$$ together with colors like (2, 1), i.e. 2 reds and 1 blue, for the unpaired numbers.

Let us rewrite the data as follows. The paired elements $(+4, -4)$ become $4.1$ and $4.2$.  The ending $.2$ stands for the negative sign.  The unpaired elements are followed by their color \bfsr or \bfb, with all blues before all reds.

The previous sequence becomes $$(1.1,\ 2.1,\ 3.\bfb,\ 2.2,\ 4.1,\ 5.\bfr,\ 6.\bfr,\ 4.2).$$  We shall call this a {\bf colored pairing sequence}.  The primitive permutation algorithm produces a primitive permutation, in which the values, from lowest to highest, correspond exactly to our data sequence $(1.1,\ 2.1,\ 3.\bfb,\ 2.2,\ 4.1,\ 5.\bfr,\ 6.\bfr,\ 4.2),$ in increasing order of the values.  The colors $\bfr$ and $\bfb$ become the colors indicating ascent values and descent values.

A bud pair, for instance $(+4, -4)$, now $(4.1,\ 4.2)$, is in the tree structure of the permutation, the root, the smallest elements of a new branch. Thus $(4.1,4.2)$ become in the new branch the nodes $(1,\ 2)$, part of a new branch node which is described as before by a primitive permutation, e.g  $$(1.1,\ 2.\bfb,\ 1.2,\ 3.\bfb,\ 4.\bfr).$$  The numbers $1.\cdots 2.\cdots$ which form its root, are $(1.1,\ 2.\bfb)$, correspond to the bud $(4.1,\ 4.2)$ in the previous colored paired sequence.  To unify the notation $(4.1,4.2)$ with $(1.1,2.\bfb)$, we call both of them $(4.1.1,4.2.\bfb)$.

We give the prefix $4$ to the whole new branch and write it as 
$$(4.1.1,\ 4.2.\bfb,\ 4.1.2,\ 4.3.\bfb,\ 4.4.\bfr).$$ 
The previous sequence is now modified to 
$$(1.1,\ 2.1,\ 3.\bfb,\ 2.2,\ 4.1.1,\ 5.\bfr,\ 6.\bfr,\ 4.2.\bfb).$$ 

Finally, we prepend to each element the number of its block. 
Recall that the blocks of the permutation are, in the line notation, the segments between successive minima.

Mark a singlet, which is the block number $i$, by $i.\bfk$. 

{\bf Remark} Singlets are singlet blocks. At this point, note that each decimal number which we have obtained corresponds throughout the process to a value of the permutation.  They are either first or preceded by the minimum of the previous block, which is smaller. So they are always ascent values, to be colored red. By the inverse cycle transform, they correspond to on diagonal elements, which are colored black. So depending on the application, singlets, or singlet blocks, may be all colored red or black $i.\bfk$ (K is the label of black for computers).

This way, every value of the permutation gets a unique {\bf decimal number} separated by periods, with all elements numbers except for the last element which is a color.

In the permutation, a decimal number will correspond to a value of the permutation, equal to the position of that decimal number in the list obtained from the permutation.  The decimal numbers of a primitive permutation will appear in decreasing order. We therefore order all decimal numbers in the order of the corresponding values of the permutation.

The resulting sequence will be called the \textbf{decimal code} of the permutation.

We shall show that it encodes the permutation we started from, and we shall axiomatize it in what follows.  Remark at this point that the decimal code has a structure which is very far from the one of a permutation.  As a permutation, it is closest to $\pi^{-1}\circ\pi=\id$, since we used the values of the permutation in ascending order $1,2,3,\ldots$.  Most of the features of the decimal code resemble the identity permutation.  The subsequence corresponding to a primitive node, for instance, has most numbers in natural ascending natural order, except for a few pairings to previous numbers.  The colors in such a branch are also in ascending order from blue to red.  The blocks of the permutation which give the first digit of the code are in ascending order as well.  

All of the information which allows us to recover the permutation comes from the pairings between elements $\cdots .1.\cdots.2.\cdots$, as well as from the shuffling between branches which have different heads.  Thus, the decimal code is a translation between the order structure of a permutation and a pairing and shuffle structure.  The latter is typical in many different areas, starting from particle physics, with creation paired with annihilation, and time shuffles of particles, to biological growth and to the very growth of libraries by addition and loss of books, which inspired the Dewey code.

We shall axiomatize the decimal code structure as follows.  

{\bf Definition}

A {\bf decimal number} $\nu$ is a sequence $i_1.i_2.\cdots.i_k.c$ separated by dots, in which $i_1,\dots,i_k\in\{1,2,\dots\}$ and the last entry $c$ is a color, red $\bfr$ or blue $\bfb$. Any cutoff $i_1.i_2.\cdots.i_l$ to length $l$ is called a {\bf head of the decimal number $\nu$}. A numerical sequence $i_1.i_2.\cdots.i_l$ where $i_1,\dots,i_l\in\{1,2,\dots\}$ is called a {\bf decimal head}.

A {\bf decimal code} $\delta$ is a finite sequence of decimal numbers with the following properties.

Any decimal head $h$ of an entry of $\delta$ is immediately followed, among the entries of $\delta$, 

(i) either by a color $c$. In that case, $h$ appears only once as head, $h.c$ is called a {\bf leaf}, and $h$ is called a {\bf leaf head}.

(ii) or by a number. In that case $h$ is called a {\bf branch head}, and is always immediately followed only by numbers in the entries of $\delta$.

Define the {\bf node cut} of the branch head $h$ as the entries obtained by 

(i) selecting the entries of $\delta$ with head $h$, in order,

(ii) cutting off each such entry to length 2 + the length of $h$, and 

(iii) keeping only the ones ending in 1, 2, or a color, in which each $h.n.n'$ is continued as $h.n.n'.1.\cdots.1.c$ with a number $\ge 0$ of 1's after $n'$.

Cut now off the head $h$ and keep the following two digits from each number, in the order of the decimal code $\delta$.  We ask these numbers to form a colored pairing sequence, i.e., such that 
(i) the numbers of the form $n.1$, $n.\bfb$, $n.\bfr$, should appear in the natural order of $n$, i.e., one of each with $n=1,n=2,n=3,\ldots$ in ascending order,

(ii) each element of the form $n.2$ should be paired with an element $n.1$ which appeared before it, with at most one such pair for each $n$,

(iii) the colors $c$ in the elements $n.c$, which we call leaves, should appear in ascending order, first blues then reds, with one of each color, and

(iv) the sequence of signs obtained by replacing the elements of type $n.2$ by $-1$ and $n.1$, $n.\bfb$, $n.\bfr$ by $+1$ should form a sequence with the property that the sum of the first $k$ terms is $\ge 2$ for any $k\ge 2$ (a form of a Dyck sequence).  In addition, we require the following property.

The first digits in each decimal numbers in the decimal code should form a set $\{1,2,\ldots,m\}$.  We call these numbers the \textbf{block labels.}

We require that for each block label $i$, there should exist exactly one element $\alpha(i)$, called the block anchor, of the form $\alpha(i)=i.1.\cdots.1.c$, with $c$ a color and with a number $\ge 0$ of 1's after $i$.  The color $c$ is black if there are no 1's after $i$, and we call this kind of $\alpha(i)$ a singlet.  Otherwise, the color $c$ is red.  We require the block anchors to appear in ascending order in the code $\delta$, so
$$\delta=(\cdots,\alpha(1),\cdots,\alpha(2),\cdots,\alpha(3),\cdots).$$
From the previous conditions $\alpha(1)$ will be on the first place of $\delta$.

This ends the axiomatization of the code.

\begin{thm}

Decimal codes consisting of a sequence of $n$ decimal numbers are in natural bijection with permutations of length $n$.

The entries of the decimal code are in bijection with the permutation $\pi$ correspond to the values of $\pi$ in ascending order, i.e. to the line form of the inverse permutation $\pi^{-1}$. 

The decimal number corresponding to a permutation value indicates its position in the family tree, followed by its color (red or blue for ascent or descent.)
The elements with a same branch head $h$ correspond, in increasing order, to the values of the elements enclosed in a parenthesis in the parenthesized form of $\pi$, which form a branch of the tree form of $\pi$. The elements in the node cut of a branch head $h$ form the pairing sequence of the primitive permutation corresponding to the root node of the branch of $h$. The elements with a head $i$ of length 1 are the values in the $i$-th block of $\pi$, the $i$-th tree in the forest of the tree form of $\pi$.
\end{thm}

\begin{proof}

Given a decimal number, which always ends in a color, we call a \textbf{head} of it the decimal number obtained by cutting off a number of elements at its end.  Cutting off the last entry, the color, leaves heads which will be called \textbf{leaf heads.}  A leaf head is continued in the decimal code by only one color.  Cutting off two or more elements at the end leaves nonempty heads which will be called \textbf{branch heads.}  Branch heads correspond to primitive permutations, and to the parentheses in the parenthesized form of the permutation.  The algorithm for constructing a permutation $\pi$ from a decimal code proceeds as follows.

\begin{center}
\includegraphics[height=1.0\textheight]{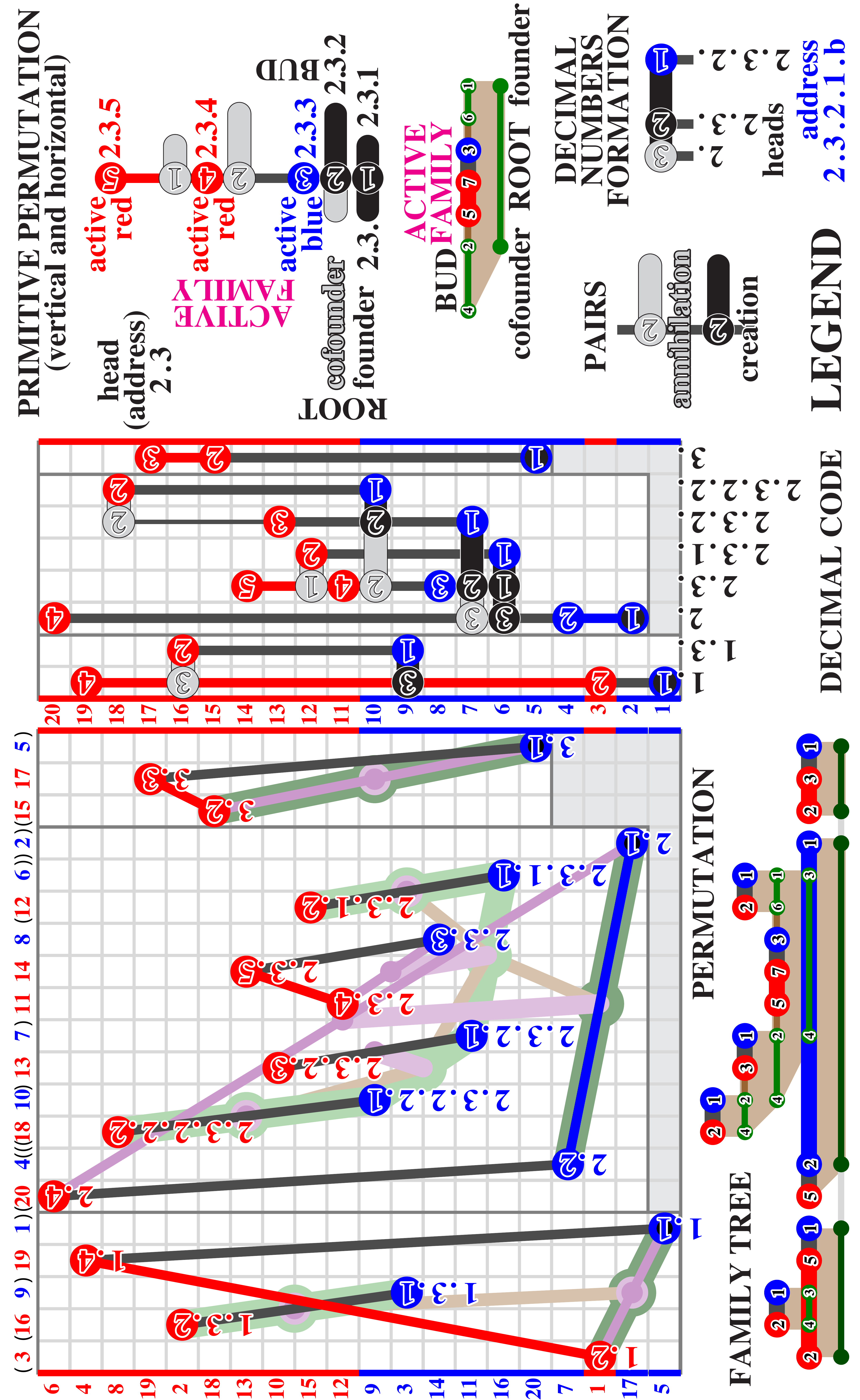}
\end{center}

For a branch head $h$, find all its continuations in the given decimal code which have length $\ge \vert h \vert+2,$ where $\vert h \vert$ denotes the length of $h$.  These continuations will be of the form $h.n.c,$ with $c$ a color, which we keep, and $h.n,n'$, where $n,n'$ are numbers.  We keep only those for which $n'=1$ or $2$ and in which the corresponding full decimal number is $h.n.n'.1.\cdots.1.c,$ where $c$ is a color and the number of $1$'s after $h.n.n'$ is $\ge 0$.  We now cut off the head $h$ and leave the pairs $n.c$ or $n.n'$ with $n'=1$ or $2$, in the order in which the corresponding full decimal numbers appear in the decimal code $\delta$.  Our main requirement for a decimal code was that such a sequence of pairs is a colored pairing data which encodes a primitive permutation $\sigma$, which we construct with the primitive permutation algorithm.  The outer parentheses for the permutation $\sigma$ are labeled by the head $h$.  The elements of the active family in $\sigma$ are labeled by decimal numbers of the form $h.n.c$, where $c$ is a color encoding the ascent/descent of that member of the active family.  The buds of $\sigma$ correspond to paired elements in the data sequence and their entries are labeled by pairs $h.n.1$ and $h.n.2$.  Now, assuming that $h$ has length $\ge 2$, let $h=h'.p,$ where $p$ is a number.  The primitive permutation $\sigma'$ corresponding to $h'$ will have in it a bud labeled $(h'.p.1, h'.p.2)=(h.1,h.2).$  We now replace this small parenthesis in $\sigma'$ with the whole parenthesis containing $\sigma$, which as we recall was labeled by $h$ and contained as a root precisely the pair $(h.1,h.2)$.  

The algorithm proceeds as follows.  

The block anchors are the elements $\alpha(i)=i.1.\cdots.1.c$, where $c$ is a color and the number of 1's from the second place on is $\ge 0$.  We cut them off to their first two digits, which will be called $\tilde{\alpha}(i)$ and place them in parentheses, to form a list $$((\tilde{\alpha}(1)),\ldots,(\tilde{\alpha}(m))),$$ where $m$ is the number of blocks.  We call this list the \textbf{parenthesized decimal code.}  

Take now all branch heads of length 1, which label the non singlet blocks.  Their two-digit extensions, as before, encode primitive permutations.  Replace in the parenthesized decimal code the singlet $(\tilde{\alpha}(i))$ with the primitive permutation corresponding to the head $i$.  Continue then with branch heads of length 2 and bigger, inductively by head length.  

Construct the primitive permutation corresponding to each branch node head $h$.  Replace now in the parenthesized decimal code the bud labeled $(h.1,h.2)$ by the whole parenthesis of the primitive permutation built from the head $h$.  We continue to enlarge the parenthesized decimal code until all the branch heads have been used and all the corresponding primitive permutations have been introduced in it.  We now replace each entry in parenthesized decimal code, which is a decimal number, by the number indicating the position of that decimal number in the list which is the given decimal code $\delta$.  We obtain this way the permutation $\pi$ corresponding to $\delta$, in parenthesized form.  Removing the parentheses leaves $\pi$ in line notation.  The parentheses correspond to the nodes which are primitive permutation of the tree structure of $\pi$, and are precisely the primitive permutations encoded in the decimal code $\delta$.  The colors at the end of each decimal number give the ascent/descent/singlet statute of the corresponding permutation value.  It is not hard to see that the two directions of the algorithms matching a permutation with a decimal code are inverse to each other.  The main difficulties are concentrated in the primitive permutation algorithm, which we described in detail.  
\end{proof}
\newpage
\textbf{An example}
The code for the permutation 
$$(\rr{3},\ \rr{16},\ \bb{9},\ \rr{19}\ ,\bb{1},\ \rr{20},\ \bb{4},\ \rr{18},\ \bb{10},\ \rr{13},\ \bb{7},\ \rr{11},\ \rr{14},\ \bb{8},\ \rr{12},\ \bb{6},\ \bb{2},\ \rr{15},\ \rr{17},\ \bb{5})$$
in which the colors indicate ascent/descent values, is 
\begin{eqnarray*}
\begin{array}{rl}
 1 &\ 1.1.\bfb  \\
 2 &\ 2.1.\bfb \\
 3 &\ 1.2.\bfr \\
 4 &\ 2.2.\bfb  \\
 5 &\ 3.1.\bfb  \\
 6 &\ 2.3.1.1.\bfb  \\
 7 &\ 2.3.2.1.\bfb  \\
 8 &\ 2.3.3.\bfb  \\
 9 &\ 1.3.1.\bfb  \\
 10 &\ 2.3.2.2.1.\bfb  \\
 11 &\ 2.3.4.\bfr  \\
 12 &\ 2.3.1.2.\bfr  \\
 13 &\ 2.3.2.3.\bfr  \\
 14 &\ 2.3.5.\bfr  \\
 15 &\ 3.2.\bfr  \\
 16 &\ 1.3.2.\bfr  \\
 17 &\ 3.3.\bfr  \\
 18 &\ 2.3.2.2.2.\bfr  \\
 19 &\ 1.4.\bfr  \\
 20 &\ 2.4.\bfr  \\
\end{array}
\end{eqnarray*}

The block anchors are $\alpha(1)=1.1.\bfb,$ $\alpha(2)=2.1.\bfb,$ $\alpha(3)=3.1.\bfb$.  The parenthesized decimal code starts with the two digit heads of these as 

$$((1.1),\ (2.1),\ (3.1)).$$

The branch heads, which correspond to primitive permutations are obtain by cutting away $\ge 2$ digits at the end of the decimal numbers. They are the following.
\begin{eqnarray*}
\begin{array}{l}
1  \\
2 \\
3  \\
1.3  \\
2.3\\
2.3.1 \\
2.3.2 \\
2.3.2.2\\
\end{array}
\end{eqnarray*}

The decimal numbers on the branch with head $h=2.3$, selected and truncated to the two digits following $h$, are in order 
$$1.1\quad 2.1\quad 3.\bfb\quad 2.2\quad 4.\bfr\quad 1.2\quad 5.\bfr.$$  
In pairing notation, they correspond to $(1,2,3,-2,4,-1,5)$.  The corresponding primitive permutation is $((4,2),5,7,3,(6,1))$, with active family $(\rr{5},\rr{7},\bb{3})$.  The buds are $(4,2)$ and $(6,1)$.  Recalling that the decimal head was $h=2.3$, the first four digits of the decimal code of the numbers in each parenthesis are 

$$((2.3.2.2,\ 2.3.2.1),\ 2.3.4.\bfr,\ 2.3.5.\bfr,\ 2.3.3.\bfb,\ (2.3.1.2,\ 2.3.1.1)).$$
Note that the head $h=2.3$ encodes this whole primitive permutation, and in particular labels its outer parentheses.  The inner buds, and their parentheses are encoded by $2.3.2$ and $2.3.1$, respectively.  The endings $(.2,\ .1)$ in each indicate the members of the bud, which will become the root elements, cofounder and founder, of the branches in which these buds would develop. These emerging branches have heads $2.3.2$ and respectively $2.3.1$.  The branches will be inserted in the parenthesized decimal code, in the place of these buds.

The last two digits of these, which encode the values of the primitive permutation\\
 $((4,2),5,7,3,(6,1))$ are 

$$((2.2,2.1),4.\bfr,5.\bfr,3.\bfb,(1.2,1.1)).$$

In the end, we obtain a parenthesized decimal code which contains all the entries in the given decimal code and starts with $$((1.2.\bfr,\ 1.3.2.\bfr,\ 1.3.1.\bfb,\ 1.4.\bfr,\cdots)).$$

We replace these numbers by their position in the decimal code $\delta$, to obtain the parenthesized 
$$((\rr{3},\ (\rr{16},\ \bb{9}),\ \rr{19}\ ,\bb{1}),\ (\rr{20},\ \bb{4},\ (((\rr{18},\ \bb{10}),\ \rr{13},\ \bb{7}),\ \rr{11},\ \rr{14},\ \bb{8},\ (\rr{12},\ \bb{6})),\ \bb{2}),\ (\rr{15},\ \rr{17},\ \bb{5})).$$
This is precisely our initial permutation.  The branch with head $h=2.3$ which we analyzed before, called there $((4,2),5,7,3,(6,1)),$ have now become 
$$(((\rr{18},\bb{10}),\ \rr{13},\ \bb{7}),\ \rr{11},\ \rr{14},\ \bb{8},(\rr{12},\ \bb{6})),$$
in which the bud $(4,\ 2)$ has expanded to the branch $((\rr{18},\bb{10}),\ \rr{13},\ \bb{7})$.  We keep out of this branch its smallest elements $(10,\ 7)$, which form its root as cofounder and founder.  That branch started in the permutation when $10$ pulled out $7$ from its family to start a new branch.  Replacing that branch by its root, we obtain the primitive node 
$$((10,\ 7),\ \rr{11},\ \rr{14},\ \bb{8},\ (12,\ 6)).$$

The pattern of this node, obtained by reducing the numbers, in order, to the smallest possible values, is 
$$((4,\ 2),\ \rr{5},\ \rr{7},\ \bb{3},\ (6,\ 1)),$$
which is the primitive permutation of the node.  This primitive permutation is inserted into the final parenthesized permutation by the presence of the head $h=2.3$ which labeled it in the decimal code.

The decimal code is shown this way to encode the structure of the permutation $\pi$ by means of the values corresponding to the pairings in each primitive component, together with the way in which the values corresponding to different branches are shuffled with each other.  Other than this, the components of the decimal code are all in natural ascending order.  The order change described by a permutation has been replaced with pairing and shuffling information, of an entirely different nature.

\section{Family statistics: modeling cell biology}

Just like the combinatorics of binomial coefficients, the $2^n$ subsets of $\overline n$ into subsets with the same cardinality, led to the normal distribution and the start of modern probability theory, we expect the map from of each of the $n!$ permutations of $\overline n$ into trees of primitive permutations and into registries of families to open new directions in probability theory. The fact that families only grow when they are very small, but overwhelmingly bud into new branches when they are large, and the fact that each branch grows independently, with branches ``competing'' for new numbers, could be ideal for modeling biological processes. 

The $n$ elements of a line permutation are partitioned into families playing the role of cells with $k$ blue and $l$ red elements each, counting the ascents and descents. As new numbers, modeling food, are introduced into the permutation, cells receive it proportional to their size.

Out of the $k+l$ positions of a $(k,l)$ cell, $1/(k+l)$ (or $2/(k+l)$ in case $k$ is equal to 1), produced cell growth.  The other positions diminish the cell by one  and produce a $(1,1)$ offspring. Thus in early stages cells mostly grow and when they become larger they mostly reproduce.

We have computed the statistical distribution of cells by $(k,l)$ type and the result was remarkably simple, a multinomial with multivariate Gaussian distribution.  

The number of permutations of $\overline n$ with $n_1$ singlets and $n_{(i,j)}$ families with $i$ ascents and $j$ descents is
$$ \frac{n!}{n_1!\prod (i+j)!^{n_{(i,j)}}}\cdot  \frac{(\sum n_{(i,j)})!}{\prod{n_{(i,j)}!}}$$
The configurations are all the sequences of numbers $(n_1,n_{(1,1)},n_{(2,1)},n_{(1,2)},n_{(3,1)},n_{(2,2)},n_{(1,3)},\dots)$ with 
$$n_1+2 n_{(1,1)}+3 (n_{(2,1)}+n_{(1,2)})+4(n_{(3,1)}+n_{(2,2)}+n_{(1,3)})+\dots=n$$
Summed over these the multinomials above add up to $n!$. The reason behind this very simple form of the sum at the end of a very complex process is that our structure theorem gives a bijection between permutations of $\overline n$ and unrestricted compositions of $n$ into families.

Thus a complex cell growth model has its outcome exactly described, in a simple way, by our new structure theory. It opens a gate towards the study of the deeper structure of the model, the possible tree structures arising from it, and corresponding new structures in probability and statistics.

\section{Conclusions}
In this paper, we have developed aspects of the internal structure of permutations, as well as the algebraic structure of multinomials, which count permutations by means of bicolored set partitions.

In shortly forthcoming papers, we shall describe applications of these results in simplicial geometry and in theoretical physics.

\section{Acknowledgments}
This paper started as our student Nick Early insisted on taking typed notes of our presentation of our study of permutations. We expanded with comments and examples all the parts which were not immediately obvious to him. Without his tremendous help, and his insistence that initial stage of the study be put out, this paper could not have reached the publication stage.

We would like to thank our undergraduate students Sean Cotner and Andrew Hanlon who listened patiently and commented during our initial development of this topic.

We would like to thank Ira Gessel for his historical comments on the cycle transform. George Andrews and Alexandre Kirllov, Sr. for useful and encouraging comments.

\begin{center}
\includegraphics[height=.95\textheight]{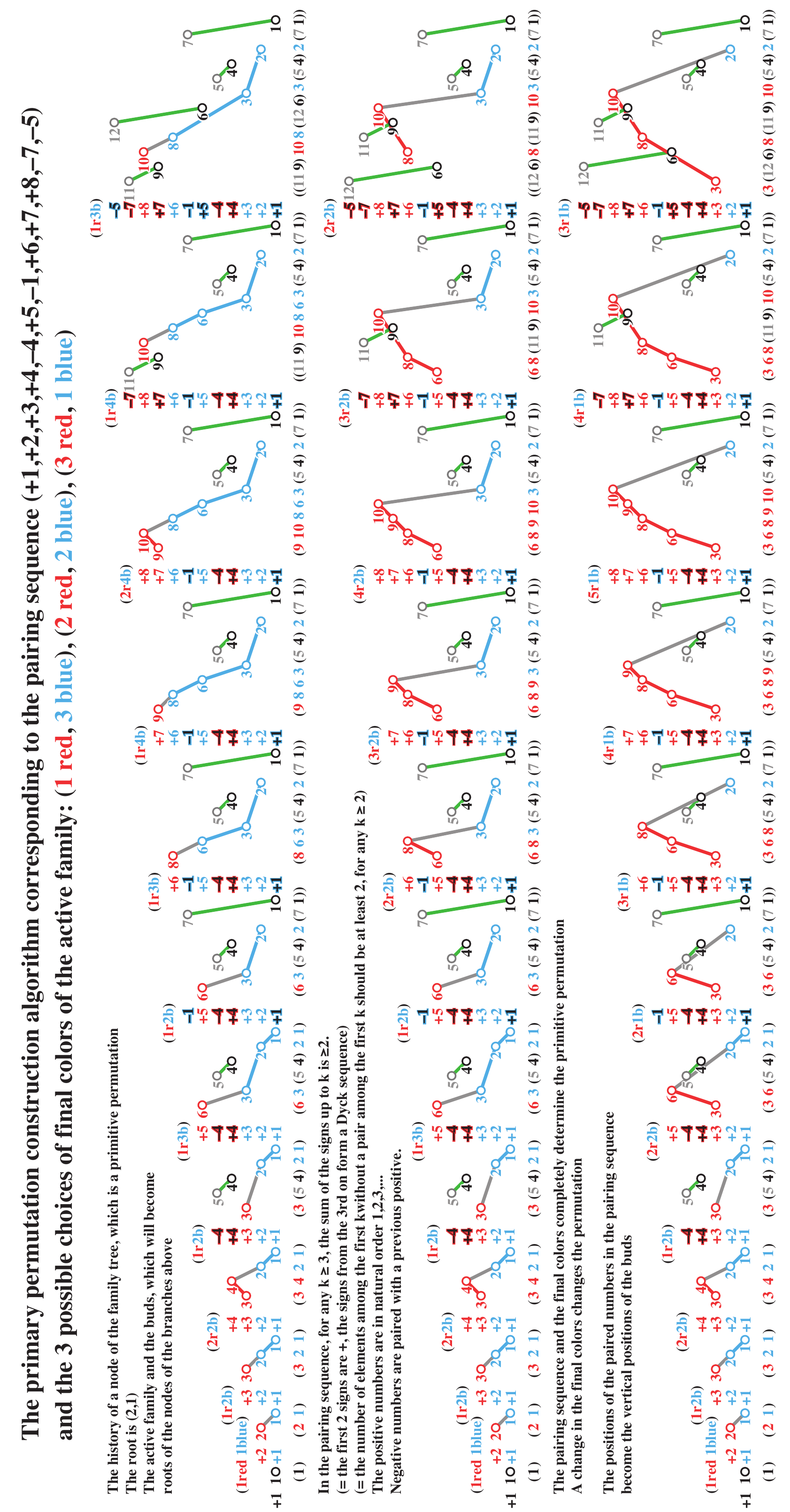}
\end{center}
\begin{center}
\includegraphics[height=.99\textheight]{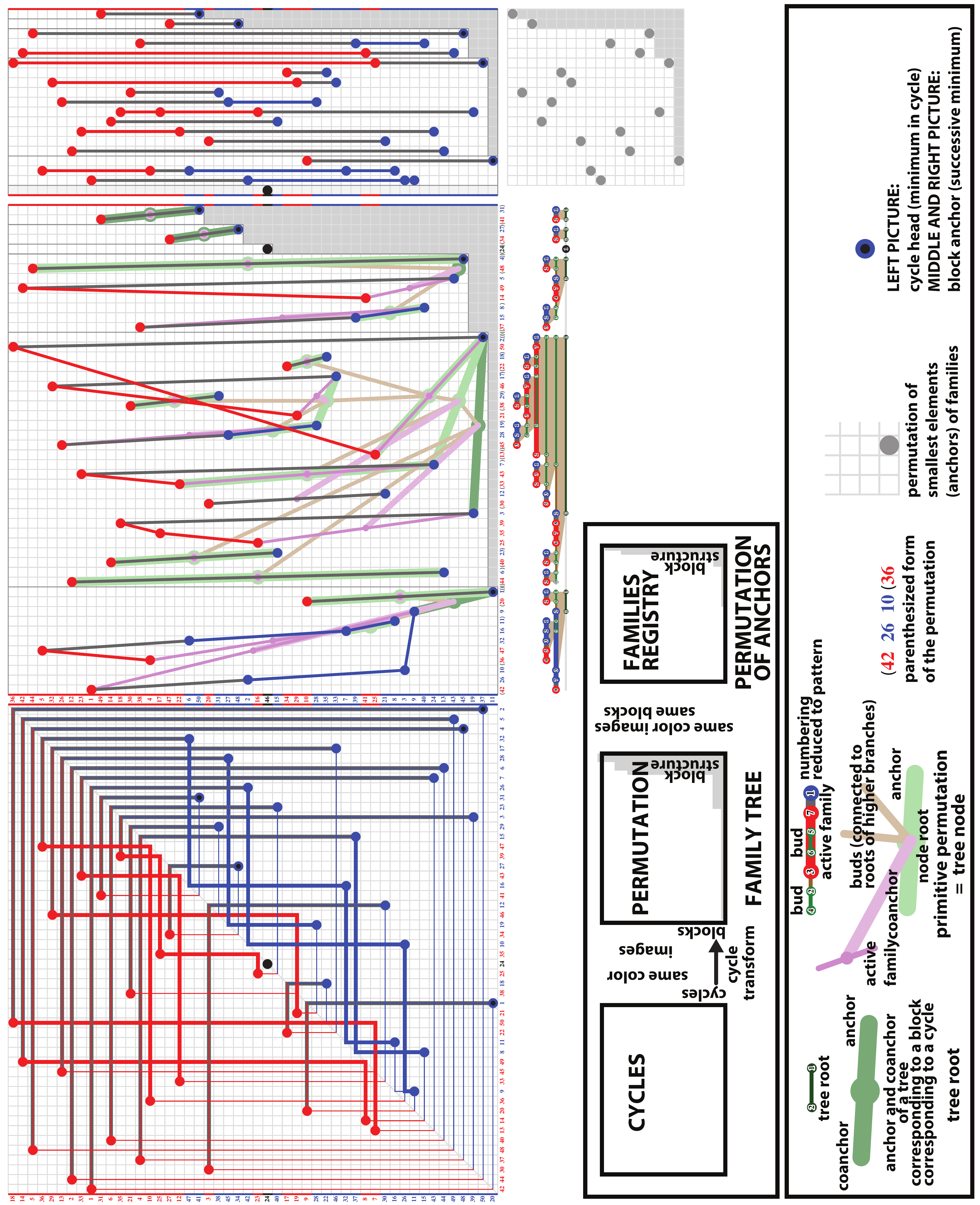}
\end{center}

\nocite{*}

\bibliographystyle{amsplain}

\bibliography{PermutationsPaper}

\providecommand{\bysame}{\leavevmode\hbox to3em{\hrulefill}\thinspace}
\providecommand{\MR}{\relax\ifhmode\unskip\space\fi MR }
\providecommand{\MRhref}[2]{%
  \href{http://www.ams.org/mathscinet-getitem?mr=#1}{#2}
}
\providecommand{\href}[2]{#2}
\begin{thebibliography}{10}

\bibitem{MR2103213}
Marcelo Aguiar and Frank Sottile, \emph{Structure of the
  {M}alvenuto-{R}eutenauer {H}opf algebra of permutations}, Adv. Math.
  \textbf{191} (2005), no.~2, 225--275. \MR{2103213 (2005m:05226)}

\bibitem{MR2194965}
\bysame, \emph{Structure of the {L}oday-{R}onco {H}opf algebra of trees}, J.
  Algebra \textbf{295} (2006), no.~2, 473--511. \MR{2194965 (2006k:16078)}

\bibitem{MR0557013}
George~E. Andrews, \emph{The theory of partitions}, Addison-Wesley Publishing
  Co., Reading, Mass.-London-Amsterdam, 1976, Encyclopedia of Mathematics and
  its Applications, Vol. 2. \MR{0557013 (58 \#27738)}

\bibitem{MR1634067}
\bysame, \emph{The theory of partitions}, Cambridge Mathematical Library,
  Cambridge University Press, Cambridge, 1998, Reprint of the 1976 original.
  \MR{1634067 (99c:11126)}

\bibitem{arkani2012scattering}
Nima Arkani-Hamed, Jacob~L Bourjaily, Freddy Cachazo, Alexander~B Goncharov,
  Alexander Postnikov, and Jaroslav Trnka, \emph{Scattering amplitudes and the
  positive grassmannian}, arXiv preprint arXiv:1212.5605 (2012).

\bibitem{behpeapetzub:bound}
R.E. Behrend, P.A. Pearce, V.B. Petkova, and J.-B. Zuber, \emph{Boundary
  conditions in rational conformal field theories}, Nuclear Phys.B \textbf{579}
  (200), 707--773.

\bibitem{bpz:infcon}
A.A. Belavin, A.N. Polyakov, and Zamolodchikov.A.V., \emph{Infinite conformal
  symmetries in two dimensional quantum field theory}, Nucl. Phys. B
  \textbf{241} (1984), 333--380.

\bibitem{bidlou:racwig}
L.C Biedenharn and J.D. Louck, \emph{The {R}acah-{W}igner algebra in quantum
  theory}, Addison-Wesley, 1981.

\bibitem{MR1485138}
Mikl{\'o}s B{\'o}na, \emph{Exact enumeration of {$1342$}-avoiding permutations:
  a close link with labeled trees and planar maps}, J. Combin. Theory Ser. A
  \textbf{80} (1997), no.~2, 257--272. \MR{1485138 (98j:05003)}

\bibitem{MR2676667}
\bysame, \emph{The absence of a pattern and the occurrences of another},
  Discrete Math. Theor. Comput. Sci. \textbf{12} (2010), no.~2, 89--102.
  \MR{2676667 (2011h:05005)}

\bibitem{MR2924741}
\bysame, \emph{Non-overlapping permutation patterns}, Pure Math. Appl.
  (PU.M.A.) \textbf{22} (2011), no.~2, 99--105. \MR{2924741}

\bibitem{MR2919720}
\bysame, \emph{Combinatorics of permutations}, second ed., Discrete Mathematics
  and its Applications (Boca Raton), CRC Press, Boca Raton, FL, 2012, With a
  foreword by Richard Stanley. \MR{2919720}

\bibitem{MR2582703}
Mikl{\'o}s B{\'o}na and Philippe Flajolet, \emph{Isomorphism and symmetries in
  random phylogenetic trees}, J. Appl. Probab. \textbf{46} (2009), no.~4,
  1005--1019. \MR{2582703 (2011b:60029)}

\bibitem{capitzzub:ade}
A.~Capelli, C.~Itzykson, and J.-B. Zuber, \emph{The {A}-{D}-{E} classification
  of minimal and ${A_1}^1$ conformal invariant theories}, Comm. Math. Phys.
  \textbf{13} (1987), 1--26.

\bibitem{MR0371673}
L.~Carlitz, \emph{Permutations, sequences and special functions}, SIAM Rev.
  \textbf{17} (1975), 298--322. \MR{0371673 (51 \#7891)}

\bibitem{MR621729}
\bysame, \emph{Up-down permutations of higher order}, Collect. Math.
  \textbf{31} (1980), no.~3, 243--258. \MR{621729 (82m:05007)}

\bibitem{MR644697}
\bysame, \emph{Enumeration of permutations by sequences. {II}}, Fibonacci
  Quart. \textbf{19} (1981), no.~5, 398--406, 465. \MR{644697 (83e:05014)}

\bibitem{MR509280}
Leonard Carlitz, \emph{Enumeration of permutations by sequences}, Fibonacci
  Quart. \textbf{16} (1978), no.~3, 259--268. \MR{509280 (81j:05009)}

\bibitem{MR1217619}
William Y.~C. Chen and Richard~P. Stanley, \emph{Derangements on the
  {$n$}-cube}, Discrete Math. \textbf{115} (1993), no.~1-3, 65--75. \MR{1217619
  (94k:05012)}

\bibitem{coq:twi}
R.~Coquereaux, \emph{Twisted partition functions for ade quantum field theories
  and ocneanu algebras of quantum symmetries}, hep-th/0107001, 2001.

\bibitem{coq:a2}
Robert Coquereaux, \emph{The {A}2 {O}cneanu quantum groupoid}, Contemporary
  Mathematics, vol. 376, AMS, 2003.

\bibitem{coq:rac}
\bysame, \emph{{R}acah - {W}igner quantum 6j symbols, {O}cneanu cells for
  {A}{N} diagrams, and quantum groupoids}, Journal of Geometry and Physics
  \textbf{57} (2007), 387--434.

\bibitem{coqsci:twi}
Robert Coquereaux and Gil Schieber, \emph{Twisted partition functions for ade
  boundary conformal field theories and {O}cneanu algebras of quantum
  symmetries}, Journal of Geometry and Physics \textbf{42} (2002), 216.

\bibitem{MR1357285}
Jes{\'u}s~A. de~Loera, Bernd Sturmfels, and Rekha~R. Thomas, \emph{Gr\"obner
  bases and triangulations of the second hypersimplex}, Combinatorica
  \textbf{15} (1995), no.~3, 409--424. \MR{1357285 (97b:13035)}

\bibitem{MR1608551}
P.~Di~Francesco, \emph{Meander determinants}, Comm. Math. Phys. \textbf{191}
  (1998), no.~3, 543--583. \MR{1608551 (99e:05007)}

\bibitem{MR1843567}
\bysame, \emph{The meander determinant and its generalizations},
  Calogero-{M}oser-{S}utherland models ({M}ontr\'eal, {QC}, 1997), CRM Ser.
  Math. Phys., Springer, New York, 2000, pp.~127--144. \MR{1843567
  (2002h:82041)}

\bibitem{confft}
Philippe Di~Francesco, Pierre Mathieu, and David S{\'e}n{\'e}chal,
  \emph{Conformal field theory}, Graduate Texts in Contemporary Physics,
  Springer-Verlag, New York, 1997. \MR{1424041 (97g:81062)}

\bibitem{2013arXiv1301.0192D}
T.~{Dimofte}, M.~{Gabella}, and A.~B. {Goncharov}, \emph{{K-Decompositions and
  3d Gauge Theories}}, ArXiv e-prints (2013).

\bibitem{dri:qua}
V.G. Drinfeld, \emph{Quantum groups}, Proceedings of the International Congress
  of Mathematicians, Academic Press, Berkeley, 1986, vol.1, pp.~798--820.

\bibitem{MR781736}
Leonhard Euler, \emph{Commentationes mechanicae ad theoriam machinarum
  pertinentes. {V}ol. {III}}, Leonhardi Euleri Opera Omnia, Series Secunda:
  Opera Mechanica et Astronomica, XVII, Orell F\"ussli, Z\"urich, 1982, Edited
  and with a preface by Charles Blanc and Pierre de Haller. \MR{781736
  (86j:01054)}

\bibitem{MR961255}
\bysame, \emph{Introduction to analysis of the infinite. {B}ook {I}},
  Springer-Verlag, New York, 1988, Translated from the Latin and with an
  introduction by John D. Blanton. \MR{961255 (89g:01067)}

\bibitem{MR1025504}
\bysame, \emph{Introduction to analysis of the infinite. {B}ook {II}},
  Springer-Verlag, New York, 1990, Translated from the Latin and with an
  introduction by John D. Blanton. \MR{1025504 (91i:01143)}

\bibitem{MR2848686}
\bysame, \emph{Vollst\"andige {A}nleitung zur niedern und h\"ohern {A}lgebra.
  {V}olume 1}, Cambridge Library Collection, Cambridge University Press,
  Cambridge, 2009, Edited and with a foreword by Johann Philipp Gr{\"u}son,
  Reprint of the 1796 original. \MR{2848686}

\bibitem{MR2866059}
\bysame, \emph{Vollst\"andige {A}nleitung zur niedern und h\"ohern {A}lgebra.
  {V}olume 2}, Cambridge Library Collection, Cambridge University Press,
  Cambridge, 2009, Edited and with a foreword by Johann Philipp Gr{\"u}son,
  Reprint of the 1797 original. \MR{2866059}

\bibitem{MR2859265}
\bysame, \emph{Vollst\"andige {A}nleitung zur niedern und h\"ohern {A}lgebra.
  {V}olume 3}, Cambridge Library Collection, Cambridge University Press,
  Cambridge, 2009, Reprint of the 1796 original. \MR{2859265}

\bibitem{MR2735922}
Stefan Forcey, Aaron Lauve, and Frank Sottile, \emph{Hopf structures on the
  multiplihedra}, SIAM J. Discrete Math. \textbf{24} (2010), no.~4, 1250--1271.
  \MR{2735922 (2012e:16090)}

\bibitem{homfly:inv}
P.~Freyd, J.~Yetter, D.and~Hoste, W.~Lickorish, K.~Millet, and A.~Ocneanu,
  \emph{A new polynomial invariant of knots and links}, Bull. Amer. Math. Soc.
  \textbf{12} (1985), 239--246.

\bibitem{gan:clas}
T.~Gannon, \emph{The classification of su(3) modular invariants revisited},
  Annales de L'Institut Henri Poincare: Phys. Theor. \textbf{65} (1996),
  15--55.

\bibitem{gel:hyp}
I.M. Gel$'$fand, \emph{General theory of hypergeometric functions}, Soviet
  Math. Dokl. \textbf{33} (1986), no.~3, 573--577.

\bibitem{gessel2001smith}
Ira~M Gessel, \emph{The smith college diploma problem}, The American
  Mathematical Monthly \textbf{108} (2001), no.~1, 55--57.

\bibitem{MR1373677}
Ira~M. Gessel and Richard~P. Stanley, \emph{Algebraic enumeration}, Handbook of
  combinatorics, {V}ol.\ 1,\ 2, Elsevier, Amsterdam, 1995, pp.~1021--1061.
  \MR{1373677 (96m:05007)}

\bibitem{gooharjon:cox}
Fred Goodman, Pierre de~la Harpe, and Vaughan~F.R. Jones, \emph{{C}oxeter
  graphs and towers of algebras}, M.S.R.I. publications, vol.~14,
  Springer-Verlag, 1989.

\bibitem{MR2398786}
Peter Huggins, Bernd Sturmfels, Josephine Yu, and Debbie~S. Yuster, \emph{The
  hyperdeterminant and triangulations of the 4-cube}, Math. Comp. \textbf{77}
  (2008), no.~263, 1653--1679. \MR{2398786 (2009c:52021)}

\bibitem{MR644144}
Gordon James and Adalbert Kerber, \emph{The representation theory of the
  symmetric group}, Encyclopedia of Mathematics and its Applications, vol.~16,
  Addison-Wesley Publishing Co., Reading, Mass., 1981, With a foreword by P. M.
  Cohn, With an introduction by Gilbert de B. Robinson. \MR{644144 (83k:20003)}

\bibitem{knutsontao:hon}
A.~Knutson and T.~Tao, \emph{Honeycombs and sums of hermitian matrices},
  Notices of the AMS \textbf{48} (2001), 175--186.

\bibitem{sunder:tqft}
V.~Kodiyalam and V.~S. Sunder, \emph{Topological quantum field theories from
  subfactors}, CRC, vol.~12, Chapman and Hall, 2000.

\bibitem{MR1354144}
I.~G. Macdonald, \emph{Symmetric functions and {H}all polynomials}, second ed.,
  Oxford Mathematical Monographs, The Clarendon Press Oxford University Press,
  New York, 1995, With contributions by A. Zelevinsky, Oxford Science
  Publications. \MR{1354144 (96h:05207)}

\bibitem{2012arXiv1202.1223M}
A.~{Mir}, F.~{Rossello}, and L.~{Rotger}, \emph{{A new balance index for
  phylogenetic trees}}, ArXiv e-prints (2012).

\bibitem{moosei:qft}
G.~Moore and N.~Seiberg, \emph{Classical and quantum conformal field theory},
  Comm. Math. Phys. \textbf{123} (1989), 177--254.

\bibitem{ocn:su3}
A.~Ocneanu, \emph{The classification of subgroups of quantum su(n)}, Quantum
  symmetries in theoretical physics and mathematics, Contemp. Math., AMS, 2002,
  pp.~133--159.

\bibitem{ocn:act}
Adrian Ocneanu, \emph{Actions of discrete amenable groups on factors}, Lecture
  Notes in Math., vol. 1138, Springer-Verlag, 1985.

\bibitem{ocn:qugr}
\bysame, \emph{Quantized group string algebras and {G}alois theory for
  algebras}, Operator Algebras and Applications, vol.~2, Cambridge University
  Press, 1987, pp.~119--172.

\bibitem{ocn:graphgeo}
\bysame, \emph{Graph geometry, quantized groups and nonamenable subfactors},
  Lake Tahoe Lectures, June-July, 1989.

\bibitem{ocn:inv}
\bysame, \emph{An invariant coupling between $3$-manifolds and subfactors, with
  connections to topological and conformal quantum field theory}, Notes, 1991.

\bibitem{ocn:cdf}
\bysame, \emph{Lectures at {C}oll\'ege de {F}rance}, 1991.

\bibitem{ocn:qsymtokyo}
\bysame, \emph{Quantum symmetry, differential geometry of finite graphs and
  classification of subfactors}, 1991, Recorded by Y.Kawahigashi.

\bibitem{ocn:chir}
\bysame, \emph{Chirality for operator algebras}, Subfactors: Proceedings of a
  Taniguchi Symposium, Singapore:World Sicientific, 1994, pp.~39--63.

\bibitem{MR2302545}
Lior Pachter and Bernd Sturmfels, \emph{The mathematics of phylogenomics}, SIAM
  Rev. \textbf{49} (2007), no.~1, 3--31. \MR{2302545 (2008a:92035)}

\bibitem{petzub:ocncel}
V.B. Petkova and J.-B. Zuber, \emph{The many faces of ocneanu cells}, Nuclear
  Phys.B \textbf{603} (2001), 449--496.

\bibitem{pop:cla}
S.~Popa, \emph{Classification of subfactors: Reduction to commuting squares},
  Invent. Math. \textbf{101} (1990), 19--43.

\bibitem{restur:inv}
N.~Reshetikhin and V.~G. Turaev, \emph{Invariants of 3-manifolds via link
  polynomials and quantum groups}, Invent. Math. \textbf{103} (1991), 547--598.

\bibitem{roc:cell}
P.~Roche, \emph{{O}cneanu cell calculus and integrable lattice models}, Comm.
  Math. Phys. \textbf{127} (1990), 395--424.

\bibitem{MR2825238}
Raman Sanyal, Frank Sottile, and Bernd Sturmfels, \emph{Orbitopes}, Mathematika
  \textbf{57} (2011), no.~2, 275--314. \MR{2825238 (2012g:52001)}

\bibitem{schvar:hyp}
V.V. Schechtman and A.N. Varchenko, \emph{Hypergeometric solutions of
  {K}nizhnik-{Z}amolodchikov equations}, Lett. Math. Phys. \textbf{20} (1990),
  279--283.

\bibitem{sch:alg}
Gil Schieber, \emph{L'alg{\'e}bre des symŽtries quantiques d'{O}cneanu et la
  classification des systmes conformes ˆ 2d}, Ph.D. thesis, Rio de Janeiro and
  Marseille, 2003.

\bibitem{MR2039975}
Frank Sottile, \emph{Toric ideals, real toric varieties, and the moment map},
  Topics in algebraic geometry and geometric modeling, Contemp. Math., vol.
  334, Amer. Math. Soc., Providence, RI, 2003, pp.~225--240. \MR{2039975
  (2005g:14100)}

\bibitem{MR2071813}
David Speyer and Bernd Sturmfels, \emph{The tropical {G}rassmannian}, Adv.
  Geom. \textbf{4} (2004), no.~3, 389--411. \MR{2071813 (2005d:14089)}

\bibitem{MR1322068}
Richard~P. Stanley, \emph{A survey of {E}ulerian posets}, Polytopes: abstract,
  convex and computational ({S}carborough, {ON}, 1993), NATO Adv. Sci. Inst.
  Ser. C Math. Phys. Sci., vol. 440, Kluwer Acad. Publ., Dordrecht, 1994,
  pp.~301--333. \MR{1322068 (95m:52024)}

\bibitem{MR1379568}
\bysame, \emph{Hyperplane arrangements, interval orders, and trees}, Proc. Nat.
  Acad. Sci. U.S.A. \textbf{93} (1996), no.~6, 2620--2625. \MR{1379568
  (97i:52013)}

\bibitem{MR1676282}
\bysame, \emph{Enumerative combinatorics. {V}ol. 2}, Cambridge Studies in
  Advanced Mathematics, vol.~62, Cambridge University Press, Cambridge, 1999,
  With a foreword by Gian-Carlo Rota and appendix 1 by Sergey Fomin.
  \MR{1676282 (2000k:05026)}

\bibitem{MR2310744}
\bysame, \emph{Alternating permutations and symmetric functions}, J. Combin.
  Theory Ser. A \textbf{114} (2007), no.~3, 436--460. \MR{2310744
  (2008a:05269)}

\bibitem{MR2821562}
\bysame, \emph{Two enumerative results on cycles of permutations}, European J.
  Combin. \textbf{32} (2011), no.~6, 937--943. \MR{2821562}

\bibitem{MR2868112}
\bysame, \emph{Enumerative combinatorics. {V}olume 1}, second ed., Cambridge
  Studies in Advanced Mathematics, vol.~49, Cambridge University Press,
  Cambridge, 2012. \MR{2868112}

\bibitem{MR1363949}
Bernd Sturmfels, \emph{Gr\"obner bases and convex polytopes}, University
  Lecture Series, vol.~8, American Mathematical Society, Providence, RI, 1996.
  \MR{1363949 (97b:13034)}

\bibitem{tur:sha}
V.~G. Turaev, \emph{Topology of shadows}, Preprint, 1991.

\bibitem{turvir:stasum}
V.~G. Turaev and O.Y. Viro, \emph{State sum invariants of $3$-manifolds and
  quantum $6j$-symbols}, Topology \textbf{31} (1992), 865--902.

\bibitem{ver:fus}
E.~Verlinde, \emph{Fusion rules and modular transformations in {2-D} conformal
  field theory}, Nuc. Phys. B \textbf{300} (1988), 360--376.

\bibitem{wen:heck}
H.~Wenzl, \emph{{H}ecke algebras of type {A} and subfactors}, Invent. Math.
  \textbf{92} (1988), 345--383.

\bibitem{wit:tqft}
E.~Witten, \emph{Topological quantum field theory}, Comm. Math. Phys.
  \textbf{117} (1988), 353--386.

\bibitem{wit:qft}
\bysame, \emph{Quantum field theory}, Comm. Math. Phys. \textbf{121} (1989),
  351--399.

\bibitem{zub:discrsym}
J.-B. Zuber, \emph{Discrete symmetries of conformal theories}, Nuclear Phys.B
  \textbf{176} (1986), 127--129.

\bibitem{difrazub:sunlat}
\bysame, \emph{Su(n) lattice integrable models associated with graphs}, Nuclear
  Phys.B \textbf{338} (1990), 602--.

\end{thebibliography}

\end{document}